\documentclass[a4paper,11pt]{amsart}
\usepackage[english]{babel}
\usepackage[latin1]{inputenc}
\usepackage[T1]{fontenc}
\usepackage{amsmath}
\usepackage{amsfonts}
\usepackage{amssymb}
\usepackage{amsthm,epic}
\usepackage{mathrsfs}
\usepackage{enumitem}
\usepackage{graphicx}
\usepackage{footnote}
\usepackage{hyperref}
\hypersetup{colorlinks=true,    linkcolor=blue,
citecolor=red,       filecolor=BrickRed,       urlcolor=darkgreen
}

\renewcommand{\geq}{\geqslant}
\renewcommand{\leq}{\leqslant}
\newtheorem*{theo}{Theorem}
\newtheorem{thm}{Theorem}
\newtheorem*{defn}{Definition}
\newtheorem{rem}[thm]{Remark}
\newtheorem{cor}[thm]{Corollary}
\newtheorem{prop}[thm]{Proposition}
\newtheorem{lem}[thm]{Lemma}
\let\o=\omega

\usepackage{color}
\definecolor{darkgreen}{rgb}{0,0.4,0}
\definecolor{MyDarkBlue}{rgb}{0,0.08,0.50}
\definecolor{BrickRed}{rgb}{0.65,0.08,0}

\title[New steps in walks with small steps in the quarter plane]{New steps in walks with small steps\\in the quarter plane\\{\tiny Series expressions for the generating functions}}

\author{I.\ Kurkova} \address{Laboratoire de Probabilit\'es et
        Mod\`eles Al\'eatoires, Universit\'e Pierre et Marie Curie,
        4 Place Jussieu, 75252 Paris Cedex 05, France} \email{Irina.Kourkova@upmc.fr}

\author{K.\ Raschel} \address{CNRS \& F\'ed\'eration de recherche Denis Poisson \& Laboratoire de Math\'ematiques et Physique Th\'eorique, Universit\'e de Tours, Parc de Grandmont, 37200 Tours, France} \email{Kilian.Raschel@lmpt.univ-tours.fr}


\date{\today}

\begin{document}

\begin{abstract}
In this article we obtain new expressions for the generating functions counting (non-singular) walks with small steps in the quarter plane. Those are given in terms of infinite series, while in the literature, the standard expressions use solutions to boundary value problems. We illustrate our results with three examples (an algebraic case, a transcendental D-finite case, and an infinite group model).
\end{abstract}


\maketitle

\section{Introduction}
\label{sec:introduction}

\subsection{Context}
In the field of enumerative combinatorics, much progress has been recently made in the study of walks in the quarter plane ${\bf Z}_{+}^{2}=\{0,1,\ldots \}^2$. The general aim is the following: given a set $\mathcal{S}$ of allowed steps (or jumps), it is a matter of counting the number of walks constructed from these steps, which start from a given point and end at a given point or subdomain of the quarter plane. Without hypotheses on $\mathcal{S}$, this problem is, still today, out of reach. In this article, following the seminal work of Bousquet-M\'elou and Mishna \cite{BMM}, we shall assume that the steps are small: in other words, $\mathcal{S}\subset\{-1,0,1\}^2\setminus \{(0,0)\}$. See Figures \ref{Ex}, \ref{ExExEx}, \ref{The_five_singular_walks}, \ref{ExExEx-6} and \ref{Exinfinite} for examples. There are obviously $2^8=256$ models. But one is easily convinced that some models are trivial; some models are equivalent (by diagonal symmetry) to other ones; and finally, some models are equivalent to models of walks confined in a half-plane, for which the general theory already exists \cite{BT}. It happens that out of the $256$ models, only $79$ inherently different ones remain to be studied \cite{BMM}. Let $q(i,j;n)$ denote the number of paths in ${\bf Z}_{+}^{2}$
having length $n$, starting from $(0,0)$ and ending at $(i,j)$.
 Define their generating function (GF) as
  \unitlength=0.6cm
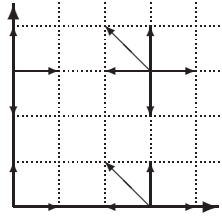
\begin{figure}[t]
  \begin{center}
\begin{tabular}{cccc}
    \hspace{-0.9cm}
        \begin{picture}(5,5.5)
    \thicklines
    \put(1,1){{\vector(1,0){4.5}}}
    \put(1,1){\vector(0,1){4.5}}
    \thinlines
    \put(4,4){\vector(1,0){1}}
    \put(4,4){\vector(-1,0){1}}
    \put(4,4){\vector(0,-1){1}}
    \put(4,4){\vector(0,1){1}}
    \put(4,4){\vector(-1,1){1}}
    \put(4,1){\vector(1,0){1}}
    \put(4,1){\vector(-1,0){1}}
    \put(4,1){\vector(0,1){1}}
    \put(4,1){\vector(-1,1){1}}
    \put(1,4){\vector(1,0){1}}
    \put(1,4){\vector(0,-1){1}}
    \put(1,4){\vector(0,1){1}}
    \put(1,1){\vector(1,0){1}}
    \put(1,1){\vector(0,1){1}}
    \linethickness{0.1mm}
    \put(1,2){\dottedline{0.1}(0,0)(4.5,0)}
    \put(1,3){\dottedline{0.1}(0,0)(4.5,0)}
    \put(1,4){\dottedline{0.1}(0,0)(4.5,0)}
    \put(1,5){\dottedline{0.1}(0,0)(4.5,0)}
    \put(2,1){\dottedline{0.1}(0,0)(0,4.5)}
    \put(3,1){\dottedline{0.1}(0,0)(0,4.5)}
    \put(4,1){\dottedline{0.1}(0,0)(0,4.5)}
    \put(5,1){\dottedline{0.1}(0,0)(0,4.5)}
    \end{picture}
    \end{tabular}
  \end{center}
  \vspace{-4mm}
\caption{Example of model (with an infinite group) considered in this article (note that on the boundary, the jumps are the natural ones: those that
would take the walk out ${\bf Z}_{+}^{2}$ are discarded)}
\label{Ex}
\end{figure}
    \begin{equation}
     \label{def_CGF}
          Q(x,y;z)=\sum_{i,j,n\geq 0} q(i,j;n)x^{i}y^{j}z^{n}.
     \end{equation}
There are then three usual key challenges:
 \begin{enumerate}[label=(\Roman{*}),ref={\rm (\Roman{*})}]
     \item\label{Challenge_1}Finding an expression for the numbers $q(i,j;n)$, or for $Q(x,y;z)$;
     \item\label{Challenge_2}Determining the nature of $Q(x,y;z)$: is it holonomic?\footnote{In other words, see \cite[Appendix B.4]{FLAJ}, is the vector space over ${\bf C}(x,y,z)$---the field of rational functions in the three variables $x,y,z$---spanned by the set of all derivatives of $Q(x,y;z)$ finite dimensional?} In that event, is it algebraic, or even rational?
     \item\label{Challenge_3}What is the asymptotic behavior, as their length goes to infinity,
      of the number of walks ending at some given point or domain (for instance one axis)?
\end{enumerate}
The functional equation \eqref{functional_equation} below served as the basis for all previous analyses \cite{BK2, BMM,FIM, FR, KRG, KRIHES,MMM,MM2, Ra}. It determines $Q(x,y;z)$ through the boundary functions $Q(x,0;z)$, $Q(0,y;z)$ and $Q(0,0;z)$, as follows:
  \begin{multline}
   \label{functional_equation}
   K(x,y;z)Q(x,y;z)\\=K(x,0;z)Q(x,0;z)+K(0,y;z)Q(0,y;z)-K(0,0;z) Q(0,0;z)-x y,
     \end{multline}
where
     \begin{equation}
     \label{def_kernel}
          K(x,y;z)=xyz[\textstyle\sum_{(i,j)\in\mathcal{S}}x^{i}y^{j}-1/z]
     \end{equation}
is called the {kernel of the walk}. We refer to \cite{BMM} for the (short and intuitive) proof of Equation \eqref{functional_equation}. It has been anticipated in \cite{BMM} and confirmed in the articles \cite{BK2, FR, KRG, KRIHES,MMM,MM2, Ra} that the analysis of both problems \ref{Challenge_1} and \ref{Challenge_2} highlights the notion of the {group of the walk}, introduced by
Malyshev \cite{MA,MAL,MALY}. This is the group
     \begin{equation}
     \label{group}
          \langle\xi,\eta\rangle
     \end{equation}
of birational transformations of $({\bf C}\cup\{\infty\})^2$, 
 which is generated by
     \begin{equation}
     \label{xietaf}
          \xi(x,y)= \left(x,\frac{1}{y}\frac{\sum_{(i,-1)\in\mathcal{S}}x^{i}}
          {\sum_{(i,+1)\in\mathcal{S}}x^{i}}\right),
          \qquad  \eta(x,y)=\left(\frac{1}{x}\frac{\sum_{(-1,j)\in\mathcal{S}}y^{j}}
          {\sum_{(+1,j)\in\mathcal{S}}y^{j}},y\right).
     \end{equation}
Each element of $\langle\xi,\eta\rangle$ leaves invariant $\sum_{(i,j)\in\mathcal{S}}x^{i}y^{j}$, the GF of the step set $\mathcal{S}$. Further, $\xi^2=\eta^2={\rm id}$, and $\langle \xi,\eta\rangle$ is a dihedral group of order even and larger than or equal to four. It has been proved in \cite{BMM} that $23$ of the $79$ walks have a finite group, while the $56$ others admit an infinite group.\footnote{Proving that a given model has a finite group is easy: it suffices to compute the group defined in \eqref{group}. As examples, the models of Figure \ref{ExExEx} all have finite groups, of order $4$, $6$, $8$ and $8$, respectively. On the other hand, it is much more complicated to prove that a model has an infinite group (examples are proposed in Figures \ref{Ex} and \ref{Exinfinite}). To that purpose, methods are presented in \cite{BMM}; see also \cite{FIM,FR2}.}

\subsection{Existing results in the literature}
For $22$ of the $23$ models with finite group, GFs $Q(x,0;z)$, $Q(0,y;z)$
and $Q(0,0;z)$---and hence $Q(x,y;z)$
by \eqref{functional_equation}---have
been computed in \cite{BMM} by means of certain (half-)orbit sums of the functional equation \eqref{functional_equation}. For the $23$rd model with finite group, known as Gessel's walks (see Figure \ref{ExExEx}), the GFs have been expressed by radicals in \cite{BK2} thanks to a guessing-proving method using computer calculations; they were also found in \cite{KRG} by solving some boundary value problems. All $23$ finite group models admit a holonomic GF: $19$ walks turn out to have a holonomic but non-algebraic GF, while for $4$ walks (among which Kreweras' and Gessel's models), $Q(x,y;z)$ is algebraic. This was first proved in \cite{BK2,BMM}; see also \cite{FR,KRIHES} for alternative proofs.

  \unitlength=0.6cm
\begin{figure}[t]
  \begin{center}
\begin{tabular}{cccc}
    \hspace{-0.9cm}
        \begin{picture}(4,4.5)
    \thicklines
    \put(1,1){{\vector(1,0){3.5}}}
    \put(1,1){\vector(0,1){3.5}}
    \thinlines
    \put(3,3){\vector(1,0){1}}
    \put(3,3){\vector(-1,0){1}}
    \put(3,3){\vector(0,-1){1}}
    \put(3,3){\vector(0,1){1}}
    \linethickness{0.1mm}
    \put(1,2){\dottedline{0.1}(0,0)(3.5,0)}
    \put(1,3){\dottedline{0.1}(0,0)(3.5,0)}
    \put(1,4){\dottedline{0.1}(0,0)(3.5,0)}
    \put(2,1){\dottedline{0.1}(0,0)(0,3.5)}
    \put(3,1){\dottedline{0.1}(0,0)(0,3.5)}
    \put(4,1){\dottedline{0.1}(0,0)(0,3.5)}
    \end{picture}
    \hspace{0.6cm}
&
    \begin{picture}(4,4)
    \thicklines
    \put(1,1){{\vector(1,0){3.5}}}
    \put(1,1){\vector(0,1){3.5}}
    \thinlines
    \put(3,3){\vector(1,1){1}}
    \put(3,3){\vector(-1,0){1}}
    \put(3,3){\vector(0,-1){1}}
    \linethickness{0.1mm}
    \put(1,2){\dottedline{0.1}(0,0)(3.5,0)}
    \put(1,3){\dottedline{0.1}(0,0)(3.5,0)}
    \put(1,4){\dottedline{0.1}(0,0)(3.5,0)}
    \put(2,1){\dottedline{0.1}(0,0)(0,3.5)}
    \put(3,1){\dottedline{0.1}(0,0)(0,3.5)}
    \put(4,1){\dottedline{0.1}(0,0)(0,3.5)}
    \end{picture}
    \hspace{0.6cm}
&   \begin{picture}(4,4)
\thicklines
    \put(1,1){\vector(1,0){3.5}}
    \put(1,1){\vector(0,1){3.5}}
    \thinlines
    \put(3,3){\vector(1,1){1}}
    \put(3,3){\vector(-1,-1){1}}
    \put(3,3){\vector(1,0){1}}
        \put(3,3){\vector(-1,0){1}}
        \linethickness{0.1mm}
    \put(1,2){\dottedline{0.1}(0,0)(3.5,0)}
    \put(1,3){\dottedline{0.1}(0,0)(3.5,0)}
    \put(1,4){\dottedline{0.1}(0,0)(3.5,0)}
    \put(2,1){\dottedline{0.1}(0,0)(0,3.5)}
    \put(3,1){\dottedline{0.1}(0,0)(0,3.5)}
    \put(4,1){\dottedline{0.1}(0,0)(0,3.5)}
    \end{picture}
    \hspace{0.6cm}
 &   \begin{picture}(4,4)
 \thicklines
    \put(1,1){\vector(1,0){3.5}}
    \put(1,1){\vector(0,1){3.5}}
    \thinlines
    \put(3,3){\vector(1,0){1}}
    \put(3,3){\vector(-1,1){1}}
    \put(3,3){\vector(-1,0){1}}
    \put(3,3){\vector(1,-1){1}}
        \linethickness{0.1mm}
    \put(1,2){\dottedline{0.1}(0,0)(3.5,0)}
    \put(1,3){\dottedline{0.1}(0,0)(3.5,0)}
    \put(1,4){\dottedline{0.1}(0,0)(3.5,0)}
    \put(2,1){\dottedline{0.1}(0,0)(0,3.5)}
    \put(3,1){\dottedline{0.1}(0,0)(0,3.5)}
    \put(4,1){\dottedline{0.1}(0,0)(0,3.5)}
    \end{picture}
    \end{tabular}
  \end{center}
  \vspace{-4mm}
\caption{Four famous examples, known as simple, Kreweras', Gessel's and Gouyou-Beauchamps' walks, respectively}
\label{ExExEx}
\end{figure}
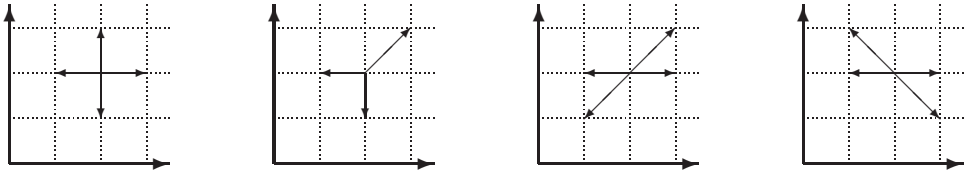

The $5$ models with infinite group on Figure \ref{The_five_singular_walks} are commonly known as {singular walks}. They are notably distinct from the others, since they have no jumps to the West, South-West and South. These $5$ models are studied in detail in \cite{MMM,MM2}; they all have non-holonomic GFs.

  \unitlength=0.6cm
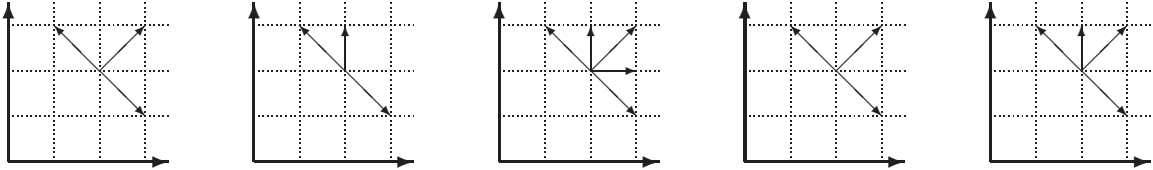
\begin{figure}[t]
  \begin{center}
\begin{tabular}{ccccc}
    \hspace{-0.9cm}
        \begin{picture}(4,4.5)
    \thicklines
    \put(1,1){{\vector(1,0){3.5}}}
    \put(1,1){\vector(0,1){3.5}}
    \thinlines
    \put(3,3){\vector(1,-1){1}}
    \put(3,3){\vector(-1,1){1}}
    \put(3,3){\vector(1,1){1}}
    \linethickness{0.1mm}
    \put(1,2){\dottedline{0.1}(0,0)(3.5,0)}
    \put(1,3){\dottedline{0.1}(0,0)(3.5,0)}
    \put(1,4){\dottedline{0.1}(0,0)(3.5,0)}
    \put(2,1){\dottedline{0.1}(0,0)(0,3.5)}
    \put(3,1){\dottedline{0.1}(0,0)(0,3.5)}
    \put(4,1){\dottedline{0.1}(0,0)(0,3.5)}
    \end{picture}
    \hspace{0.35cm}
&
    \begin{picture}(4,4)
    \thicklines
    \put(1,1){{\vector(1,0){3.5}}}
    \put(1,1){\vector(0,1){3.5}}
    \thinlines
    \put(3,3){\vector(1,-1){1}}
    \put(3,3){\vector(-1,1){1}}
    \put(3,3){\vector(0,1){1}}
    \linethickness{0.1mm}
    \put(1,2){\dottedline{0.1}(0,0)(3.5,0)}
    \put(1,3){\dottedline{0.1}(0,0)(3.5,0)}
    \put(1,4){\dottedline{0.1}(0,0)(3.5,0)}
    \put(2,1){\dottedline{0.1}(0,0)(0,3.5)}
    \put(3,1){\dottedline{0.1}(0,0)(0,3.5)}
    \put(4,1){\dottedline{0.1}(0,0)(0,3.5)}
    \end{picture}
    \hspace{0.35cm}
&   \begin{picture}(4,4)
\thicklines
    \put(1,1){\vector(1,0){3.5}}
    \put(1,1){\vector(0,1){3.5}}
    \thinlines
    \put(3,3){\vector(1,-1){1}}
    \put(3,3){\vector(-1,1){1}}
    \put(3,3){\vector(1,1){1}}
        \put(3,3){\vector(1,0){1}}
    \put(3,3){\vector(0,1){1}}
        \linethickness{0.1mm}
    \put(1,2){\dottedline{0.1}(0,0)(3.5,0)}
    \put(1,3){\dottedline{0.1}(0,0)(3.5,0)}
    \put(1,4){\dottedline{0.1}(0,0)(3.5,0)}
    \put(2,1){\dottedline{0.1}(0,0)(0,3.5)}
    \put(3,1){\dottedline{0.1}(0,0)(0,3.5)}
    \put(4,1){\dottedline{0.1}(0,0)(0,3.5)}
    \end{picture}
    \hspace{0.35cm}
 &   \begin{picture}(4,4)
 \thicklines
    \put(1,1){\vector(1,0){3.5}}
    \put(1,1){\vector(0,1){3.5}}
    \thinlines
    \put(3,3){\vector(1,-1){1}}
    \put(3,3){\vector(-1,1){1}}
    \put(3,3){\vector(1,1){1}}
        \linethickness{0.1mm}
    \put(1,2){\dottedline{0.1}(0,0)(3.5,0)}
    \put(1,3){\dottedline{0.1}(0,0)(3.5,0)}
    \put(1,4){\dottedline{0.1}(0,0)(3.5,0)}
    \put(2,1){\dottedline{0.1}(0,0)(0,3.5)}
    \put(3,1){\dottedline{0.1}(0,0)(0,3.5)}
    \put(4,1){\dottedline{0.1}(0,0)(0,3.5)}
    \end{picture}
        \hspace{0.35cm}
&   \begin{picture}(4,4)
\thicklines
    \put(1,1){\vector(1,0){3.5}}
    \put(1,1){\vector(0,1){3.5}}
    \thinlines
    \put(3,3){\vector(1,-1){1}}
    \put(3,3){\vector(-1,1){1}}
    \put(3,3){\vector(1,1){1}}
    \put(3,3){\vector(0,1){1}}
        \linethickness{0.1mm}
    \put(1,2){\dottedline{0.1}(0,0)(3.5,0)}
    \put(1,3){\dottedline{0.1}(0,0)(3.5,0)}
    \put(1,4){\dottedline{0.1}(0,0)(3.5,0)}
    \put(2,1){\dottedline{0.1}(0,0)(0,3.5)}
    \put(3,1){\dottedline{0.1}(0,0)(0,3.5)}
    \put(4,1){\dottedline{0.1}(0,0)(0,3.5)}
    \end{picture}
    \hspace{0.6cm}
    \end{tabular}
  \end{center}
  \vspace{-4mm}
\caption{The $5$ singular walks studied in \cite{MMM,MM2}}
\label{The_five_singular_walks}
\end{figure}

At this step, there remain $51=79-23-5$ models. In \cite{Ra} the problem \ref{Challenge_1} was resolved for all these $51$ models---and in fact for all the $74$ non-singular walks. This was done via a unified approach: integral representations were obtained for GFs $Q(x,0;z)$, $Q(0,y;z)$ and $Q(0,0;z)$ in certain domains, by solving boundary value problems (of Riemann-Carleman type). However, these complicated explicit expressions have not been helpful for solving problem \ref{Challenge_2}, that is for determining the nature of the GFs for the $51$ non-singular walks with infinite group.


This problem has been finally solved in \cite{KRIHES}, as follows. Since the transformations $\xi$ and $\eta$ of $({\bf C}\cup\{\infty\})^2$ leave  $\sum_{(i,j)\in\mathcal{S}}x^{i}y^{j}$ invariant, one can consider a group $\langle\xi,\eta\rangle_{|{\bf T}_z}$ of automorphisms of the elliptic curve
\begin{equation}
\label{tz}
       {\bf T}_z=\{(x,y)\in({\bf C}\cup\{\infty\})^2: K(x,y;z)=0\}
\end{equation}
generated by its automorphisms $\xi$ and  $\eta$. Due to the obvious inclusion ${\bf T}_z\subset ({\bf C}\cup\{\infty\})^2$, it may happen that the group  $\langle\xi,\eta\rangle_{|{\bf T}_z}$ is finite for some $z$, while the group $\langle\xi,\eta\rangle$ on $({\bf C}\cup\{\infty\})^2$ be infinite.
Let
     \begin{equation}
     \label{eq:first_definition_H}
          \mathcal{H}= \{z \in  (0, 1/|\mathcal{S}|) : |\langle\xi,\eta\rangle_{|{\bf T}_z}| <\infty \}.
     \end{equation}
Clearly $\mathcal{H}= (0, 1/|\mathcal{S}|)$ for any of the $23$ models with finite group $\langle \xi, \eta \rangle$. The following result was proved in \cite{KRIHES}:\footnote{The key tool for the proof in \cite{KRIHES} is the following: for any $z\in(0, 1/|\mathcal{S}|)$, GFs $Q(x,0;z)$ and $Q(0,y;z)$ can be continued on the whole of ${\bf C}$ as {multi-valued} functions with infinitely many (and explicit) meromorphic branches. Then the set of poles of all these branches is proved to be infinite for any $z \in (0,1/|\mathcal{S}|)\setminus\mathcal{H}$ , which leads to the non-holonomy of the GFs.}
\begin{theo}[\cite{KRIHES}]
For all $51$ non-singular models with infinite group:
\begin{enumerate}
     \item\label{thmKR1}The subsets $\mathcal{H}$ and $ (0,1/|\mathcal{S}|)\setminus \mathcal{H}$ are both dense in $(0,1/|\mathcal{S}|)$;
     \item\label{thmKR2}For all $z \in \mathcal{H}$, the GFs $x \mapsto Q(x,0;z)$  and $y\mapsto Q(0,y;z)$ are holonomic;
     \item\label{thmKR3} For all $z \in(0,1/|\mathcal{S}|)\setminus  \mathcal{H} $,  $x \mapsto Q(x,0;z)$  and $y \mapsto Q(0,y;z)$ are non-holonomic.
\end{enumerate}
\end{theo}

Assertion \ref{thmKR3} of this theorem combined with usual properties of holonomic functions (see, e.g., \cite[Appendix B.4]{FLAJ}) implies the non-holonomy of the trivariate GF $Q(x,y;z)$ for all $51$ models  and therefore proves Bousquet-M\'elou's and Mishna's conjecture stated in \cite{BMM}. On the other hand, it seems not possible to deduce from this theorem any information on the holonomy of $Q(0,0;z)$ or $Q(1,1;z)$, or more generally of $Q(x,y;z)$ as a function of the variable $z$. Moreover, it has been remarked in \cite{KRIHES} that assertion \ref{thmKR2} above 
suggests a promising start in order to achieve better understanding and easier representations of the GFs, than those in \cite{KRIHES,Ra} mentioned above. This is the subject of the present paper.

\subsection{Main results}
\label{subsec:main_results} We analyze the GFs $Q(x,0;z)$ and
$Q(0,y;z)$ for all $74$ non-singular models and all $z \in
\mathcal{H}$.

  For any fixed $z \in (0, 1/|{\mathcal S}|)$, thus in particular for any $z\in\mathcal H$,  the elliptic curve  ${\bf T}_z$ is of genus
  $1$. The universal covering ${\bf C}$ of ${\bf T}_z$ can be considered as a union of an
infinite number of parallelograms glued together, see
Figure~\ref{The_universal_covering}, with some periods $\o_1\in
i{\bf R}$ and $\o_2\in {\bf R}$, and uniformisation formulas $\{(x(\o),y(\o)):\o\in{\bf C}\}$.
 The exact expression of these
periods (depending on $z$) is given in \eqref{expression_omega_1_2},
the uniformisation functions
 $x(\o)$ and $y(\o)$ are written down in \eqref{expression_uniformization}
  in terms of a $\wp$-Weierstrass function $\wp(\o;\o_1,\o_2)$ with periods
$\o_1$ and~$\o_2$.

 Rather than deriving an expression directly for the
GFs $Q(x,0;z)$ and $Q(0,y;z)$, we shall obtain  $Q(x(\o),0;z)$ and
$Q(0,y(\o);z)$ for all $\o \in {\bf C}$.

 Our main result is an expression for the functions $Q(x(\o),0;z)$ and $Q(0,y(\o);z)$
  as infinite series in $\o$, each term of the series being a simple rational function in $\o$.
   To state it we need to introduce some notation.
\unitlength=0.6cm
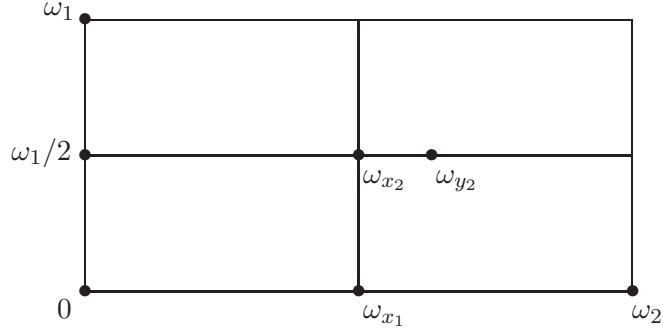
\begin{figure}[t]
\vspace{40mm}
\begin{center}
\begin{tabular}{cccc}
\begin{picture}(0,0)(0,0)
\put(-2,0){\line(1,0){4}}
\put(2,0){\line(1,0){4}}
\put(-6,0){\line(1,0){4}}
\put(-6,3){\line(1,0){12}}
\put(-6,0){\line(0,1){3}}
\put(6,0){\line(0,1){3}}
\put(-2,6){\line(1,0){4}}
\put(0,0){\line(0,1){6}}
\put(2,6){\line(1,0){4}}
\put(-6,6){\line(1,0){4}}
\put(-6,3){\line(0,1){3}}
\put(6,3){\line(0,1){3}}
\put(-6.15,-0.15){{$\bullet$}}
\put(-6.15,2.85){{$\bullet$}}
\put(-0.15,-0.15){{$\bullet$}}
\put(1.45,2.85){{$\bullet$}}
\put(-0.15,2.85){{$\bullet$}}
\put(-6.15,5.85){{$\bullet$}}
\put(5.85,-0.15){{$\bullet$}}
\put(-6.6,-0.6){{$0$}}
\put(-6.9,6.1){{$\o_1$}}
\put(-7.6,2.9){{$\o_1/2$}}
\put(6.0,-0.6){{$\o_{2}$}}
\put(0.1,-0.6){{$\o_{x_1}$}}
\put(0.1,2.4){{$\o_{x_2}$}}
\put(1.7,2.4){{$\o_{y_2}$}}
\end{picture}
\end{tabular}
\end{center}
\vspace{5mm}
\caption{The rectangle $\o_1[0,1)+\o_2[0,1)$ is a fundamental parallelogram for the Weierstrass elliptic function $\wp(\o;\o_1,\o_2)$, in terms of which $x(\o)$ and $y(\o)$ are expressed. We have also represented the points $\o_{x_1}$, $\o_{x_2}$ and $\o_{y_2}$.}
\label{The_rectangle}
\end{figure}
Define
\begin{equation}
\label{eq:first_def_f_y}
      f_y(\o)=\frac{1}{2z}\frac{x'(\o)}{\delta_{1,1}x(\o)+\delta_{0,1}+\delta_{-1,1}/x(\o)},
\end{equation}
which is an elliptic function with periods $\o_1,\o_2$, where we have noted
\begin{equation*}
     \delta_{i,j}=\left\{\begin{array}{ccc}
     1& \text{if}&(i,j)\in\mathcal S,\\
     0& \text{if}&(i,j)\notin\mathcal S.
    \end{array}\right.
\end{equation*}
The function $f_y(\o)$ has at most six poles
 in the fundamental parallelogram $\Pi_{0,0}=\o_1[0,1)+\o_2[0,1)$, see Figure~\ref{The_rectangle}.
Let $\o_3$ be defined in \eqref{expression_omega_3}. We have
$z\in\mathcal H$ if and only if $\o_3/\o_2$ is rational, see \cite{KRIHES}.
Then for any $z\in\mathcal H$
\begin{equation}
\label{kn}
     \o_3=\frac{k}{\ell}\o_2
\end{equation}
for some $k,\ell\in{\bf Z}$ with $\ell>k>0$
 ($k$ and $\ell$ depend on $z$ in the infinite group case, but not in the finite group case).
 Finally we introduce  $\o_{x_1}=\o_1/2$, $\o_{x_2}=(\o_1+\o_2)/2$ as
 well as $\o_{y_2}=(\o_1+\o_2+\o_3)/2$.

\begin{defn}
Let $g(\o)$ be a function meromorphic in the neighborhood of $\o_0\in{\bf C}$,
 where it has a pole of order $p$,
 such that
$
     g(\o) = \sum_{k=-p}^{\infty} g_k(\o-\o_0)^k.
$
Its principal part at $\o_0$ is given by
$
     \sum_{k=-p}^{-1} g_k(\o-\o_0)^k.
$
\end{defn}

\begin{thm}
\label{serie-introduction} For any of the $74$ models of
non-singular walks in the quarter plane and any $z\in \mathcal H$
{\rm(}see \eqref{eq:first_definition_H}{\rm)},
 let the $f_i^y$'s be the poles of the function $f_y(\o)$,
 and let $F_{f_i^y,y}(\o)$ be the principal parts of this function at them.
 We have
\begin{multline*}
 K(0,y(\o);z)Q(0,y(\o);z)-K(0,y(\o_{y_2});z)Q(0,y(\o_{y_2});z)
 =\\\sum_{p=-\infty}^{\infty}\sum_{n=0}^{\infty}\sum_{s=0}^{k-1}\sum_{f_i ^y\in \Pi_{0,0}-\o_{x_1}}
 A_{s,p,n}^{f_i^y} (\o),
 \end{multline*}
 where
 \begin{align}
 \label{eq:first_def_AA}
  A_{s,p,n}^{f_i^y}(\o) =
 -&(\lfloor n/\ell\rfloor+1)
     F_{f_i^y,y}(\o+s\o_2+n\o_3+p\o_1)\\
   -&
      (\lfloor n/\ell\rfloor+1)F_{f_i^y,y}(-\o+2\o_{y_2}+s\o_2+n\o_3+p\o_1)
   \nonumber\\
+2&(\lfloor n/\ell\rfloor+1)
F_{f_i^y,y}(\o_{y_2}+s\o_2+n\o_3+p\o_1).\nonumber
\end{align}
   The series of terms  $(\sum_{f_i ^y\in \Pi_{0,0}-\o_{x_1}}
 A_{s,p,n}^{f_i^y} (\o))$  is absolutely convergent.

A similar expression holds for $Q(x(\o),0;z)$ and will be given in
Theorem \ref{serie}. Let $\o_0^x \in \Pi_{0,0}$
 and $\o_0^y \in \Pi_{0,0}+\o_3/2$ be such that $x(\o_0^x)=0$ and $y(\o_0^y)=0$.
  Then the functions $r_x(\o)=K(x(\o),0;z)Q(x(\o),0;z)$ and $r_y(\o)=K(0,y(\o);z)Q(0,y(\o);z)$ can be found from the above expressions as follows:
   \begin{align}
 \label{eq:first_def_AA12}
r_y(\o)&=(r_x(\o^x_0)-r_x(\o_{x_2}) )-(r_x(\o)-r_x(\o_{x_2})) +x(\o)y(\o), \\
\label{eq:first_def_AA13} r_x(\o)&= (r_y(\o^y_0)-r_y(\o_{y_2})
)\hspace{.45mm}-(r_y(\o)-r_y(\o_{y_2})) \hspace{0.42mm}+x(\o)y(\o).
 \end{align}
\end{thm}

We now give some remarks around Theorem \ref{serie-introduction}
 (we refer to Theorem \ref{serie} for the complete statement).
\begin{itemize}
     \item First, this theorem applies both in the finite and infinite group cases and gives series
     expressions for the GFs in a unified way for all $74$ non-singular models.     This result provides an alternative and completely different representation
     of these functions than the one given in \cite{Ra} in terms of solutions to boundary value problems.
     It seems us more explicit and in this sense more satisfactory.
   As opposed to \cite{Ra}, this approach provides expressions for the unknown functions thanks to the set of all their
        algebraic branches.

     \item Second, natural questions are: How to apply this theorem? How explicit is it?
     Here are the concrete things to do, and where to find them in the article.
     \begin{itemize}
          \item First, compute $x(\o)$ $\longrightarrow$ \eqref{expression_uniformization}.
          \item Then $f_y(\o)$ $\longrightarrow$ \eqref{eq:first_def_f_y}.
          \item Then its poles and principal parts at them $\longrightarrow$ no formula, since a priori this is a case by case analysis.
           Theoretically this is not difficult, since zeroes and poles of elliptic functions are well understood.
           We refer to Section \ref{sec:focus_poles} for an analysis of the poles of
           $f_y(\o)$, which are of order $1$, $2$ or $3$.
          \item Compute $\o_1$, $\o_2$ and $\o_3$ $\longrightarrow$
           \eqref{expression_omega_1_2} and \eqref{expression_omega_3}.
          \item Deduce $k$ and $\ell$ $\longrightarrow$ \eqref{kn}.
          \item Compute $A_{s,p,n}^{f_i^y}(\o)$ $\longrightarrow$ \eqref{eq:first_def_AA}.
          \item Make the sum as in Theorem \ref{serie-introduction}.
     \end{itemize}

  \item One can in some cases deduce an expression of $Q(x,0;z)$ in terms of $x$ and $z$.
     This will be further commented in Section \ref{sec:expressions_x/y}.

     \item The function $Q(0,y;z)$ is algebraic in $y$ if and only if the function $Q(0,y(\o);z)$
      is elliptic on ${\bf C}$, with periods $\widetilde \o_1$ and $\widetilde \o_2$ which are rational multiples of $\o_1$ and $\o_2$. We shall give a simple necessary and
      sufficient condition for this in terms of the principal parts of
      $f_y(\o)$ at their poles, see Section \ref{subsec:orbit-sums}.
       There will be cancellations
     or considerable simplifications between the different terms
     \eqref{eq:first_def_AA} in this case.
\end{itemize}

 We shall give in Part \ref{part:examples} three examples to
illustrate our approach (in particular, how to work out and simplify
the expressions given in Theorem \ref{serie-introduction}),
 namely Kreweras' walk, the simple walk, and an infinite group case model.
 Further, our approach is also successfully applied in \cite{BKR}, so as to obtain a human proof of Gessel's conjecture.
 
 Although our starting point in this article is the main equation \eqref{functional_equation}
     and the study of the kernel \eqref{def_kernel}
            as in the recent papers \cite{FR,FR2,KRG,KRIHES,Ra}, our approach is deeply
    different.
 It has been first suggested by Malyshev in \cite[\S6 in Chapter 5]{MA},
 at that time for studying the stationary probabilities generating functions of ergodic
  random walks in ${\bf Z}_+^2$. The main idea consists in using, in a constructive way, Mittag-Leffler's
  theorem.\footnote{Roughly speaking, Mittag-Leffler's theorem states that it is possible to
   construct a meromorphic function with prescribed (discrete) set of poles and prescribed principal parts at
   these poles, see \cite[Section 4.13]{SG}.}
   Specifically, we first find the poles (and the principal parts at these poles)
   of the lifted 
GFs on the universal covering ${\bf C}$ of ${\bf T}_z$;
then we show how to express the GFs as an infinite series of principal parts at its poles.
 Another possible method for obtaining the GFs for $z \in \mathcal{H}$ is based on \cite[Chapter 4]{FIM}.
 It is commented in Appendix \ref{sec:another_possible_approach}.

This article closes our study of the walks with small steps in the quarter plane.
  Below we mention some open problems:
\begin{enumerate}
     \item In this article, we find expressions as infinite series for $Q(x,0;z)$ and $Q(0;y;z)$ for a dense set of values of $z$ (namely, $z\in\mathcal H$). By continuity this must provide an expression for the GFs for all values of $z\in(0,1/|\mathcal S|)$. It is an open problem to determine whether there exists an expression as an infinite series for the GFs when $z\in(0,1/|\mathcal S|)\setminus \mathcal H$.
     \item We know from \cite{KRIHES} that (as functions of the variables $x$ and $y$) the functions $Q(x,0;z)$ and $Q(0,y;z)$ are non-holonomic for all $z\in(0,1/|\mathcal S|)\setminus \mathcal H$. It is an open problem to reprove this fact, starting from the results (in particular Theorems \ref{dva} and \ref{serie}) of this paper. In other words, is it possible to prove the non-holonomy of the GFs from our holonomy results?
\end{enumerate}

\subsection{Example of application of Theorem \ref{serie-introduction}}
As a typical example of result furnished by Theorem \ref{serie-introduction}, let us consider Kreweras' walk (see Figure \ref{ExExEx}). We shall obtain that (this result is the subject of Proposition \ref{prop:Kreweras_consequence_serie} in Part \ref{part:examples})\footnote{This comes indeed from Proposition \ref{prop:Kreweras_consequence_serie} because we have $y(\o)=x(\o+\o_3/2)$, and for this reason $K(x(\o),0;z)Q(x(\o),0;z)=K(0,x(\o);z)Q(0,x(\o);z)=K(0,y(\o-\o_3/2);z)Q(0,y(\o-\o_3/2);z)$.}
\begin{multline*}
      x(\o)zQ(x(\o),0;z)= \frac{1}{2z}(\zeta_{1,2}(\o_2/3)-\zeta_{1,2}(4\o_2/3))+\frac{1}{z}(\zeta_{1,2}(\o_2)-\zeta_{1,2}(2\o_2/3))\\+\frac{1}{2z}\zeta_{1,2}(\o+\o_2/3)-\frac{1}{z}\zeta_{1,2}(\o)+\frac{1}{z}\zeta_{1,2}(\o-\o_2/3)-\frac{1}{2z}\zeta_{1,2}(\o-2\o_2/3),\qquad \forall \o\in{\bf C},
\end{multline*}
where $\zeta_{1,2}$ is
the $\zeta$-elliptic function with
periods $\o_1,2\o_2$ (see Property \ref{expression_zeta} in Appendix \ref{appendix-elliptic} for its definition).
The periods $\o_1$ and $\o_2$ are defined in \eqref{expression_omega_1_2}.
 (The dependency on $z$ of the RHS above is precisely in the periods.) Further, one has (for this model)
\begin{equation*}
     x(\o) = \frac{\wp(\o)-1/3}{-4z^2},
 \end{equation*}
where $\wp$ is the $\wp$-elliptic function with periods $\o_1,\o_2$ (see
Equation \eqref{eq:first_time_expansion_wp} for its definition).
 Using the transformation theory of elliptic functions (which typically allows to link elliptic functions
 with different couples of periods), we shall be able to deduce from the two above equations the
 following classical result \cite{BK2,BM,BMM,FIM,FL1,FL2,Kreweras}:
\begin{equation*}
     Q(x,0;z) = \frac{1}{zx}\left(\frac{1}{2z}-\frac{1}{x}-\left(\frac{1}{W}-\frac{1}{x}\right)\sqrt{1-xW^2}\right),
\end{equation*}
the function $W$ being the only power series in $z$ satisfying $W=z(2+W^3)$.

\subsection{Structure of this article}
Our paper has the following organization: Part \ref{part:general} is concerned with the general theory.
 In Section \ref{sec:ucau} we introduce the general
  framework in which we shall work throughout the article: we do the
  construction of the elliptic curve ${\bf T}_z$ and of its universal covering ${\bf C}$.
  In Section \ref{sec:explicit_expression_GFs}, functions $Q(x,0;z)$ and $Q(0,y;z)$ are lifted on
  $ {\bf T}_z$, and then on its universal covering ${\bf C}$.
   They are then continued as meromorphic functions on the whole of ${\bf C}$.
    This continuation procedure, which is described in detail in \cite{KRIHES},
    is based on ideas of \cite{FIM}. We only sketch it here, restricting ourselves with notation
    and details necessary for the next sections. The function $f_y(\o)$
    in \eqref{eq:first_def_f_y}, which appears in the statement of our main result (Theorem \ref{serie-introduction}),
     plays a crucial role in it.
    Section \ref{sec:focus_poles} is devoted to the detailed study of the poles of $f_y(\o)$ and the principal parts at
    them, that figure out in the series expression \eqref{eq:first_def_AA} of  Theorem \ref{serie-introduction}.
     In Sections \ref{sec:expression_GFs_series}--\ref{sec:mainth}
     we prove the main result of the paper:
     for all $74$ models and any $z \in \mathcal{H}$,
     all branches of $Q(x,0;z)$ and $Q(0,y;z)$ are computed as the sums of absolutely convergent
     series announced in Theorem \ref{serie-introduction}.
      In Section \ref{subsec:orbit-sums} we identify the algebraic nature of the functions $Q(x,0)$
      and $Q(0,y)$ (question \ref{Challenge_2} stated at the very beginning of this paper)
      in terms of the principal parts of functions $f_x(\o)$ and $f_y(\o)$
       (the function $f_x(\o)$ is the analogue of $f_y(\o)$ but for the function $r_x(\o)$)
       at their poles in a parallelogram of periods.
       In Section \ref{sec:expressions_x/y}
       we explain how to obtain expressions of $Q(x,0;z)$ and $Q(0,y;z)$ in terms of $x$ and $y$
       (which can be more convenient than expressions of $r_x(\o)$ and $r_y(\o)$ in terms of $\o$).

In Part \ref{part:examples} we illustrate our results with three
examples (finite group and algebraic, finite group and holonomic but
non-algebraic, infinite group).
 The first example, which we treat in full detail,
 is that of Kreweras (see Figure \ref{ExExEx}).
 In Section \ref{sec:compute}
 we obtain an expression for the generating function in terms of Weierstrass $\zeta$-functions.
 In Sections \ref{sec:Q(0,0)}--\ref{sec:proof_complete_Kreweras} we show how to derive
  the well-known expressions for $Q(0,0;z)$ and $Q(x,0;z)=Q(0,x;z)$ by the theory of
  transformation of elliptic functions. The second example is the simple walk,
  in Section \ref{sec:simple_walk}. We close Part \ref{part:examples} by Section \ref{sec:infinite_group_example},
  which is devoted to an example of walk having an infinite group.

Finally, in Appendix \ref{sec:another_possible_approach} we present
another possible approach following \cite{FIM}, and in Appendix
\ref{appendix-elliptic} we recall needed properties of elliptic
functions.

In what follows, most of the time, we shall drop the dependence of all quantities with respect to the variable $z$,
 writing, for example, $Q(0,0)$ instead of $Q(0,0;z)$.

\part{GFs for all non-singular models of walks}
\label{part:general}

\section{Universal covering and uniformization}
\label{sec:ucau}

In this section we introduce the general framework in which we shall work throughout the article: this mainly consists in defining a unifomization of the set $\{(x,y)\in ({\bf C}\cup \{\infty\})^2:K(x,y)=0\}$, where $K(x,y)$ is defined in \eqref{def_kernel}.

\subsection*{Branch points}
Let us fix $z\in(0, 1/|\mathcal{S}|)$. The kernel $K(x,y)$ is a polynomial of the second degree in both $x$ and $y$. The algebraic function $X(y)$ defined by $K(X(y),y)=0$ has thus two branches and four branch points, that we call $y_i$, $i\in\{1,\ldots ,4\}$.\footnote{That is, four points such that the two branches of $X(y)$ are equal.} The latter are the roots of  the discriminant of the second degree equation $K(x,y)=0$ in the variable~$x$:
\begin{equation}
\label{tildedx} \widetilde d(y)=
  (zy^2\delta_{0,1}-y+z\delta_{0,-1})^2-4z^2(y^2\delta_{1,1}+y\delta_{1,0}+\delta_{1,-1})(y^2\delta_{-1,1}+y\delta_{-1,0}+\delta_{-1,-1}).
\end{equation}
We refer to \cite[Section 2.1]{KRIHES} for the numbering and for further properties of the $y_i$'s. For any non-singular model, the Riemann surface of $X(y)$ is a torus (i.e., a Riemann surface of genus $1$) ${\bf T}_y$ composed of two complex spheres ${\bf C}\cup\{\infty\}$ which are properly glued together along the cuts $[y_1,y_2]$ and $[y_3, y_4]$. The analogous statement is true for the algebraic function $Y(x)$ defined by $K(x,Y(x))=0$, and for its four branch points that are roots of
\begin{equation}
\label{dx} d(x)=
  (zx^2\delta_{1,0}-x+z\delta_{-1,0})^2-4z^2(x^2\delta_{1,1}+x\delta_{0,1}+\delta_{-1,1})(x^2\delta_{1,-1}+x\delta_{0,-1}+\delta_{-1,-1}),
\end{equation}
 and its Riemann surface which is a torus ${\bf T}_x$. Since ${\bf T}_x$ and ${\bf T}_y$ are equivalent, in what follows we consider a {single} Riemann surface ${\bf T}$ with two different coverings $x, y : {\bf T} \to {\bf S}$.

\subsection*{Galois automorphisms}
Any point $s \in {\bf T}$ admits the two coordinates $(x(s),y(s))$, which by construction satisfy $K(x(s), y(s))=0$. For any $s \in {\bf T}$, there exists a unique $s'$  (resp.\ $s''$) such that $x(s')=x(s)$ (resp.\ $y(s'')=y(s)$). The values $x(s),x(s')$ (resp.\ $y(s),y(s'')$) are the two roots of the equation of the second degree $K(x, y(s))=0$ (resp.\ $K(x(s), y)=0$) in $x$ (resp.\ $y$). We define the automorphisms $\xi : {\bf T} \to {\bf T}$ (resp.\ $\eta : {\bf T} \to{\bf T}$) such that $\xi s =s'$ (resp.\ $\eta s=s''$) and call them {Galois automorphisms}, following the terminology of \cite{FIM}. Clearly $\xi s=s$ (resp.\ $\eta s=s$) if and only if $x(s)=x_i$, $i\in\{1,\ldots ,4\}$ (resp.\ $y(s)=y_i$, for some $i\in\{1,\ldots ,4\}$).

\subsection*{Universal covering and uniformization}

     \unitlength=0.6cm
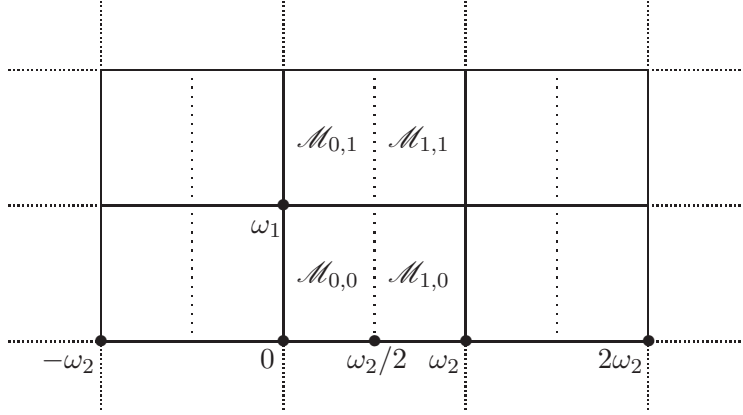
\begin{figure}[t]
    \vspace{40mm}
  \begin{center}
\begin{tabular}{cccc}

\begin{picture}(0,0)(0,0)
\put(-2,0){\line(1,0){4}}
\put(2,0){\line(1,0){4}}
\put(-6,0){\line(1,0){4}}
\put(-6,0){\line(0,1){3}}
\put(-2,0){\line(0,1){3}}
\put(2,0){\line(0,1){3}}
\put(6,0){\line(0,1){3}}
\put(-2,3){\line(1,0){4}}
\put(2,3){\line(1,0){4}}
\put(-6,3){\line(1,0){4}}
\put(-2,6){\line(1,0){4}}
\put(2,6){\line(1,0){4}}
\put(-6,6){\line(1,0){4}}
\put(-6,3){\line(0,1){3}}
\put(-2,3){\line(0,1){3}}
\put(2,3){\line(0,1){3}}
\put(6,3){\line(0,1){3}}
\put(-6,0){\dottedline{0.1}(0,0)(0,-1.5)}
\put(-2,0){\dottedline{0.1}(0,0)(0,-1.5)}
\put(2,0){\dottedline{0.1}(0,0)(0,-1.5)}
\put(6,0){\dottedline{0.1}(0,0)(0,-1.5)}
\put(-6,6){\dottedline{0.1}(0,0)(0,1.5)}
\put(-2,6){\dottedline{0.1}(0,0)(0,1.5)}
\put(2,6){\dottedline{0.1}(0,0)(0,1.5)}
\put(6,6){\dottedline{0.1}(0,0)(0,1.5)}
\put(6,6){\dottedline{0.1}(0,0)(2,0)}
\put(6,3){\dottedline{0.1}(0,0)(2,0)}
\put(6,0){\dottedline{0.1}(0,0)(2,0)}
\put(-6,6){\dottedline{0.1}(0,0)(-2,0)}
\put(-6,3){\dottedline{0.1}(0,0)(-2,0)}
\put(-6,0){\dottedline{0.1}(0,0)(-2,0)}
\put(0,0){\dottedline{0.2}(0,0)(0,6)}
\put(-4,0){\dottedline{0.2}(0,0)(0,6)}
\put(4,0){\dottedline{0.2}(0,0)(0,6)}
\put(-1.7,1.3){{$\mathscr{M}_{0,0}$}}
\put(0.3,1.3){{$\mathscr{M}_{1,0}$}}
\put(-1.7,4.3){{$\mathscr{M}_{0,1}$}}
\put(0.3,4.3){{$\mathscr{M}_{1,1}$}}
\put(-2.15,-0.15){{$\bullet$}}
\put(-6.15,-0.15){{$\bullet$}}
\put(-0.15,-0.15){{$\bullet$}}
\put(-2.15,2.85){{$\bullet$}}
\put(1.85,-0.15){{$\bullet$}}
\put(5.85,-0.15){{$\bullet$}}
\put(-7.3,-0.6){{$-\o_{2}$}}
\put(-2.5,-0.6){{$0$}}
\put(-2.7,2.4){{$\o_1$}}
\put(-0.6,-0.6){{$\o_{2}/2$}}
\put(1.2,-0.6){{$\o_{2}$}}
\put(4.9,-0.6){{$2\o_{2}$}}
\end{picture}
    \end{tabular}
  \end{center}
  \vspace{10mm}
\caption{The universal covering and the location of $\Pi_{0,0}=\mathscr{M}_{0,0}\cup\mathscr{M}_{1,0}$ on it}
\label{The_universal_covering}
\end{figure}

The torus ${\bf T}$ is isomorphic to a certain quotient space ${\bf C}/(\omega_1{\bf Z}+\omega_2{\bf Z})$, where $\omega_1,\omega_2$ are complex numbers linearly independent on ${\bf R}$.  This set can obviously be thought as the (fundamental) parallelogram $[0, \o_1]+[0, \o_2]$, whose opposed edges are identified. Up to a unimodular transform, $\omega_1,\omega_2$ are unique and are found in \cite[Lemma 3.3.2]{FIM}:
 \begin{equation}
     \label{expression_omega_1_2}
          \omega_1= i\int_{x_1}^{x_2}
          \frac{\text{d}x}{\sqrt{-d(x)}},
          \qquad \omega_2 = \int_{x_2}^{x_3} \frac{\text{d}x}{\sqrt{d(x)}}.
     \end{equation}
For an upcoming use, we define
     \begin{equation*}
     \label{def_f_x}
          g_x(t)=\left\{\begin{array}{lll}
          \displaystyle d''(x_{4})/6+d'(x_{4})/[t-x_{4}]& \text{if} & x_{4}\neq \infty,\\
          \displaystyle d''(0)/6+d'''(0)t/6 \phantom{{1^1}^{1}}& \text{if} & x_{4}=\infty,\end{array}\right.
     \end{equation*}
as well as
     \begin{equation*}
     \label{def_f_y}
          g_y(t)=\left\{\begin{array}{lll}
          \displaystyle \widetilde d''(y_{4})/6+\widetilde d'(y_{4})/[t-y_{4}]& \text{if} & y_{4}\neq \infty,\\
          \displaystyle \widetilde d''(0)/6+\widetilde d'''(0)t/6 \phantom{{1^1}^{1}}& \text{if} & y_{4}=\infty,\end{array}\right.
     \end{equation*}
and finally we introduce $\wp(\omega;\omega_1,\omega_2)$, the Weierstrass elliptic function with periods $\omega_1,\omega_2$. Throughout, we shall write $\wp(\omega)$ for $\wp(\omega;\omega_1,\omega_2)$. By definition (see, e.g., \cite{JS,WW}), we have
     \begin{equation}
     \label{eq:first_time_expansion_wp}
          \wp(\omega)=\frac{1}{\omega^2}+\sum_{(\ell_1,\ell_2)\in{\bf Z}^{2}\setminus\{ (0,0)\}}
          \left[\frac{1}{(\omega-\ell_1\omega_1-\ell_2\omega_2)^{2}}-
          \frac{1}{(\ell_1\omega_1+\ell_2\omega_2)^{2}}\right].
     \end{equation}
The universal covering of ${\bf T}$ has the form $({\bf C}, \lambda)$, where ${\bf C}$ is the usual complex plane (that can be viewed as the union of infinitely many parallelograms  $\Pi_{m,n}=\omega_1[m,m+1)+\omega_2[n,n+1)$ glued together for $m,n \in {\bf Z}$) and $\lambda : {\bf C}\to {\bf T}$ is a non-branching covering map; see Figure \ref{The_universal_covering}. For any $\o \in {\bf C}$ such that $\lambda \o=s \in {\bf T}$, we have $x(\lambda \o)=x(s)$ and $y(\lambda \o)=y(s)$. The uniformization formulas \cite[Lemma 3.3.1]{FIM} are
\begin{equation}
    \label{expression_uniformization}
          \left\{\begin{array}{l}
          \hspace{-3mm}\phantom{y}x( \lambda \omega)=g_x^{-1}(\wp(\omega)),\\
          \hspace{-3mm}\phantom{x}y(\lambda \omega)=g_y^{-1}(\wp(\omega-\omega_3/2)),
          \end{array}\right.
 \end{equation}
where
     \begin{equation}
     \label{expression_omega_3}
          \omega_3 = \int_{X(y_1)}^{x_1}
          \frac{\text{d}x}{\sqrt{d(x)}}.
     \end{equation}
Throughout, we shall write $x(\lambda \o)=x(\o)$ and $y(\lambda\o)=y(\o)$. These are elliptic functions on ${\bf C}$ with periods $\o_1,\o_2$. Clearly
 \begin{equation}
 \label{kxy0}
      K(x(\o),y(\o))=0, \qquad \forall \o \in {\bf C}.
 \end{equation}
Furthermore, since each parallelogram $\Pi_{m,n}$ represents a torus ${\bf T}$ composed of two complex spheres, the function $x(\o)$ (resp.\ $y(\o)$) takes each value of ${\bf C}\cup\{\infty\}$ twice within this parallelogram, except for the branch points $x_i$, $i\in\{1,\ldots ,4\}$ (resp.\ $y_i$, $i\in\{1,\ldots ,4\}$). Points $\o_{x_i} \in \Pi_{0,0}$ such that $x(\o_{x_i})=x_i$, $i\in\{1,\ldots,4\}$  are pictured on Figure \ref{The_fundamental_parallelogram}.
They are equal to $\o_{x_4}=0$, $\o_{x_1}=\o_2/2$, $\o_{x_3}=\o_1/2$ and $\o_{x_2}=(\o_1+\o_2)/2$.    Points $\o_{y_i}$ such that $y(\o_{y_i})=y_i$  are just the shifts of $\o_{x_i}$ by the real vector $\o_3/2$: $\o_{y_i}=\o_{x_i}+\o_3/2$ for $i\in\{1,\ldots,4\}$, see also Figure \ref{The_fundamental_parallelogram}.

  \unitlength=0.6cm
\begin{figure}[t]
    \vspace{65mm}
  \begin{center}
\begin{tabular}{cccc}

\begin{picture}(0,0)(0,0)
\put(-10,0){\line(1,0){20}}
\put(-10,10){\line(1,0){20}}
\put(-10,5){\line(1,0){20}}
\put(-10,0){\line(0,1){10}}
\put(10,0){\line(0,1){10}}
\put(0,0){\line(0,1){10}}
\put(4,0){\line(0,1){10}}
\put(-6,0){\line(0,1){10}}
\put(-10.23,-0.23){{\LARGE$\bullet$}}
\put(-6.23,-0.23){{\LARGE$\bullet$}}
\put(-0.23,-0.23){{\LARGE$\bullet$}}
\put(3.77,-0.23){{\LARGE$\bullet$}}
\put(-10.23,4.77){{\LARGE$\bullet$}}
\put(-6.23,4.77){{\LARGE$\bullet$}}
\put(-0.23,4.77){{\LARGE$\bullet$}}
\put(3.77,4.77){{\LARGE$\bullet$}}
\put(-9.8,0.3){{$\o_{x_4}$}}
\put(-5.8,0.3){{$\o_{y_4}$}}
\put(0.2,0.3){{$\o_{x_1}$}}
\put(4.2,0.3){{$\o_{y_1}$}}
\put(-9.8,5.3){{$\o_{x_3}$}}
\put(-5.8,5.3){{$\o_{y_3}$}}
\put(0.2,5.3){{$\o_{x_2}$}}
\put(4.2,5.3){{$\o_{y_2}$}}
\put(-10,-0.5){\vector(1,0){4}}
\put(-6,-0.5){\vector(-1,0){4}}
\put(-8.7,-1.2){{$\o_{3}/2$}}
\put(0,-0.5){\vector(1,0){4}}
\put(4,-0.5){\vector(-1,0){4}}
\put(1.3,-1.2){{$\o_{3}/2$}}
\put(-10,10.5){\vector(1,0){20}}
\put(10,10.5){\vector(-1,0){20}}
\put(-0.2,10.8){{$\o_{2}$}}
\put(-10.5,0){\vector(0,1){10}}
\put(-10.5,10){\vector(0,-1){10}}
\put(-11.3,5.3){{$\o_{1}$}}
\put(0.1,2.4){{$L_{x_1}^{x_2}$}}
\put(4.1,2.4){{$L_{y_1}^{y_2}$}}
\textcolor{red}{
\qbezier(-2.25,10)(-2.75,7.5)(-2.75,5)
\qbezier(2.25,10)(2.75,7.5)(2.75,5)
\qbezier(-2.25,0)(-2.75,2.5)(-2.75,5)
\qbezier(2.25,0)(2.75,2.5)(2.75,5)
\put(-2.11,9.4){{$\Delta_x$}}
}
\thicklines
\textcolor{blue}{
\qbezier(1.75,10)(1.25,7.5)(1.25,5)
\qbezier(6.25,10)(6.75,7.5)(6.75,5)
\qbezier(1.75,0)(1.25,2.5)(1.25,5)
\qbezier(6.25,0)(6.75,2.5)(6.75,5)
\put(5.4,9.4){{$\Delta_y$}}
}
\end{picture}
    \end{tabular}
  \end{center}
  \vspace{10mm}
\caption{The fundamental parallelogram $\Pi_{0,0}=\o_1[0,1)+\o_2[0,1)$, and important points and domains on it}
\label{The_fundamental_parallelogram}
\end{figure}
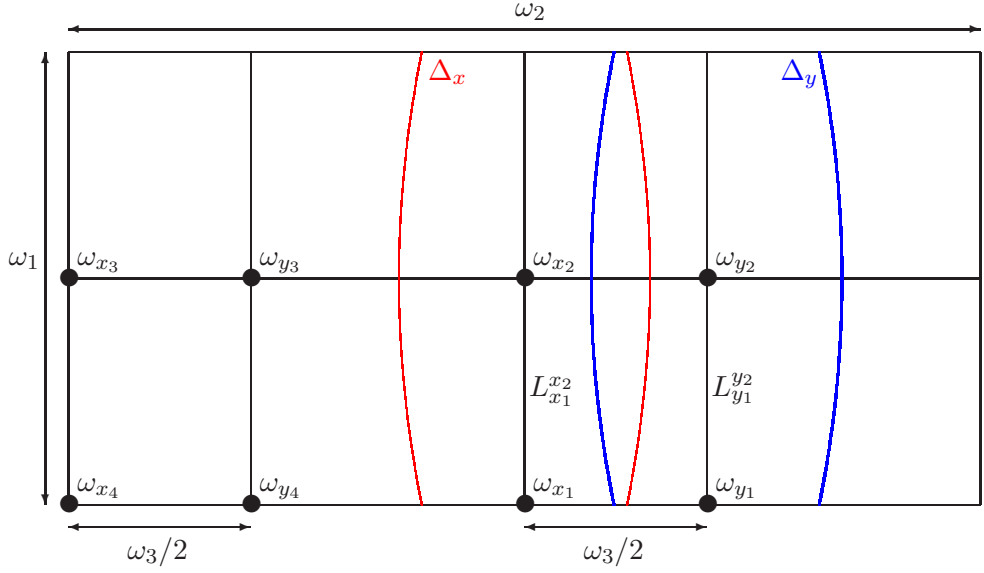

    The automorphisms $\xi$ and $\eta$ on ${\bf T}$ are lifted on the universal covering
    ${\bf C}$ in such a way that $\o_{x_2}$ and $\o_{y_2}$ stay their fixed points, respectively:
   \begin{equation}
   \label{hatxieta}
    \widehat \xi \o=-\o+2\o_{x_2},\qquad \widehat \eta \o=-\o+2\o_{y_2},\qquad \forall \o\in{\bf C}.
\end{equation}
Clearly, one has
\begin{equation}
\label{xyxieta}
x(\widehat \xi \o)=x(\o),\qquad y(\widehat\eta \o)=y( \o), \qquad \forall \o\in{\bf C}.
\end{equation}

\section{Meromorphic continuation of the GFs on the universal covering}
\label{sec:explicit_expression_GFs}
\subsection*{Meromorphic continuation of the GFs}

The domains
\begin{equation*}
     \{\o\in{\bf C} : |x(\o)|<1\},\qquad\{\o\in{\bf C}: |y(\o)|<1\}
\end{equation*}
     consist of infinitely many
curvilinear strips, which differ from translations by a multiple of $\o_2$.
    We denote by
 $\Delta_x$ (resp.\ $\Delta_y$) the strip that is within
 $\cup_{n \in {\bf Z}}\Pi_{0,n}$
  (resp.\ $\cup_{n \in {\bf Z}} \Pi_{0,n}+\o_3/2$). The domain
   $\Delta_x$ (resp.\ $\Delta_y$)
     is centered around the straight line $L_{x_1}^{x_2}=\{\o\in{\bf C} : \Re \o =\Re \o_{x_2} \}$
(resp.\ $L_{y_1}^{y_2}=\{\o\in{\bf C} : \Re \o =\Re \o_{y_2} \}$), see Figure \ref{The_fundamental_parallelogram}.
     Note that the function $Q(x(\o),0)$ (resp.\ $Q(0,y(\o))$) is
   well defined in $\Delta_x$ (resp.\ $\Delta_y$), by \eqref{def_CGF}.
   Let us put
 \begin{equation}
 \label{rD_x}
  r_x(\o)=K(x(\o),0)Q(x(\o), 0),\quad \forall \o \in \Delta_x, \quad
  r_y(\o)= K(0, y(\o)) Q(0, y(\o)),\quad \forall \o \in \Delta_y.
\end{equation}
    The domain $\Delta_x \cap \Delta_y$ is a non-empty open strip.
It follows from \eqref{functional_equation}, \eqref{kxy0} and \eqref{rD_x}
that
\begin{equation}
r_x(\o)+r_y(\o)-Q(0,0)K(0,0)-x(\o)y(\o)=0,\qquad \forall \o \in \Delta_x
\cap \Delta_y.
\end{equation}
Let
\begin{equation}
\label{Delta}
 \Delta =\Delta_x \cup \Delta_y.
\end{equation}
  Then the functions $r_x(\o)$ and $r_y(\o)$
     can be continued as meromorphic functions on
  the whole of $\Delta$, setting
\begin{align}
\label{rD} &r_x(\o)= -r_y(\o)+Q(0,0)K(0,0)+x(\o)y(\o), \qquad
\forall \o \in
\Delta_y,\\
 &r_y(\o)= -r_x(\o)+Q(0,0)K(0,0)+x(\o)y(\o),\qquad \forall \o
\in \Delta_x.\label{rDrD}
\end{align}
  The following theorem holds true, see \cite{KRIHES} or \cite{FIM}.
\begin{thm}[\cite{KRIHES},\cite{FIM}]
\label{thm2}
  We have
\label{lemde}
     \begin{equation}
     \label{delta}
          \bigcup_{n\in {\bf Z}}(\Delta+n\omega_3) ={\bf C}.
     \end{equation}
 The functions $r_x(\omega)$ and $r_y(\omega)$ can be
continued meromorphically from $\Delta$ to the whole of ${\bf C}$
via the following formulas:
  \begin{align}
       &r_x(\omega-\omega_3)=r_x(\omega)+y(\omega)[x(-\omega+2\omega_{y_2})-x(\omega)],
       \qquad \forall \o \in {\bf C}, \label{cont}
       \\
       &r_y(\omega+\omega_3)\hspace{0.12mm}=r_y(\omega)\hspace{0.12mm}+
       x(\omega)[y(-\omega+2\omega_{x_2})-y(\omega)],
       \qquad \forall \o \in {\bf C}\label{cont1}.
  \end{align}
   Further, we have, for any $\o\in{\bf C}$,
  \begin{align}
       &r_x(\omega)+r_y(\omega)-K(0,0)Q(0,0)-x(\omega)y(\omega) =0,
       \label{sqs2}
       \\
       &
       r_x(\widehat \xi \omega)=r_x(\omega),
       \label{xieta}\\
       &r_y(\widehat \eta\omega)=r_y(\omega),
      \label{xietaxieta}
       \\
       &
       r_x(\omega+\omega_1)=r_x(\omega), 
       \\
      & r_y(\omega+\omega_1)=r_y(\omega).
            \label{buzz}
  \end{align}
  \end{thm}

  The restrictions of $r_x(\o)/K(x(\o),0)$ on
     \begin{equation}
     \label{mkl}
          \mathscr{M}_{k,0}=\omega_1[0,1)+\omega_2[k/2,(k+1)/2)
     \end{equation}
for $k \in \{1,2,\ldots \}$ provide all branches  on ${\bf
C}\setminus ([x_1,x_2]\cup[x_3,x_4])$ of $Q(x,0)$ as follows:
     \begin{equation}
     \label{branches}
          Q(x,0)=\{r_x(\o)/K(x(\o),0):
           \o \text{ is the (unique) element of } \mathscr{M}_{k,0} \text{ such that } x(\o)=x\}.
     \end{equation}
     Indeed, due to the $\o_1$-periodicity of $r_x(\o)$ and $x(\o)$, the
restrictions of these functions on $\mathscr{M}_{k,\ell}=
\omega_1[\ell,\ell+1)+\omega_2[k/2,(k+1)/2)$ do not depend on
$\ell\in {\bf Z}$, and therefore determine the same branch as on
$\mathscr{M}_{k,0}$ for any~$\ell$. Further, by \eqref{xieta} the
restrictions of $r_x(\o)/K(x(\o),0)$ on $\mathscr{M}_{-k+1,0}$ and
on $\mathscr{M}_{k,0}$ lead to the same branches for any $k\in{\bf
Z}$. See Figure \ref{The_universal_covering}. The analogous statement holds for the restrictions of
$r_y(\o)/K(0, y(\o))$ on
     \begin{equation}
     \label{nkl}
          \mathscr{N}_{k,0}=\omega_3/2+\omega_1[0,1)+\omega_2[k/2,(k+1)/2)
     \end{equation}
for $k \in \{1,2,\ldots \}$.

 The continuation formulas \eqref{cont} and \eqref{cont1} for $r_x(\o)$ and $r_y(\o)$ involve the
 functions
\begin{equation}
\label{fy} f_x(\o)= y(\o)[x(-\omega+2\omega_{y_2})-x(\omega)],\qquad
f_y(\o)=x(\omega)[y(-\omega+2\omega_{x_2})-y(\omega)].
\end{equation}
   We close this section by showing that they admit alternative expressions,  namely as in \eqref{eq:first_def_f_y} used in
   Theorem~\ref{serie-introduction} and its analogue for the other coordinate.
\begin{lem}
\label{fyfy}
The following formulas hold, for any $\o\in{\bf C}$:
\begin{eqnarray}
   f_x(\o)\hspace{-2mm}&=&\hspace{-2mm}y(\omega)[x(-\omega+2\omega_{y_2})-x(\omega)]=
   \frac{1}{2z}\frac{y'(\o)}{\delta_{1,1}y(\o)+\delta_{1,0}+\delta_{1,-1}/y(\o)},\nonumber\\
 f_y(\o)\hspace{-2mm}&=&\hspace{-2mm}x(\omega)[y(-\omega+2\omega_{x_2})-y(\omega)]=
   \frac{1}{2z}\frac{x'(\o)}{\delta_{1,1}x(\o)+\delta_{0,1}+\delta_{-1,1}/x(\o)}.\label{fy2}
\end{eqnarray}
\end{lem}

\begin{proof}
By symmetry, it is enough to prove \eqref{fy2}. We first take some notation: we write the kernel \eqref{def_kernel} under the form
\begin{equation*}
     K(x,y)=a(x)y^2+b(x)y+c(x),
\end{equation*}
where for instance $a(x)=z(\delta_{1,1}x^2+\delta_{0,1}x+\delta_{-1,1})$. The next point is the equality (see \cite[Equations (3.3.3) and (3.3.4)]{FIM})
\begin{equation*}
     2a(x(\o))y(\o)+b(x(\o))=-x'(\o)/2,\qquad \forall \o\in{\bf C}.
\end{equation*}
On the other hand, one has (thanks to the root-coefficient relationships)
\begin{equation*}
     y(-\omega+2\omega_{x_2})-y(\omega)=-\frac{b(x(\o))}{a(x(\o))}-2y(\o)=-\frac{2a(x(\o))y(\o)+b(x(\o))}{a(x(\o))}=\frac{x'(\o)}{2a(x(\o))}.
\end{equation*}
The proof is concluded.
\end{proof}

\section{Properties of the poles of $f_y(\o)$ and of the principal parts at them}
\label{sec:focus_poles}

In this section we focus on  function $f_y(\o)$ defined in \eqref{fy}. There are
generally six poles of the function $f_y(\o)$ in any parallelogram
of periods $\o_1,\o_2$ on ${\bf C}$. The number $6$ comes from the
fact that each function $x(\omega)$, $y(-\omega+2\omega_{x_2})$ and
$y(\omega)$ has two poles in a fundamental parallelogram. (It is
worth noting that for some of the $74$ non-singular models, some of
these poles may coincide, so that their number may be smaller---see
\cite[Section 7]{KRIHES} and also Remark \ref{rem:strictly_less} in
this article.) We call $f_i^y$, $i\in\{1,\ldots, 6\}$, the poles in
$\Pi_{0,0}-\o_{x_2}$. We denote by
\begin{equation}
\label{fifi1}
     f_1^y=(x^*, \infty),\qquad f_2^y=(x^{**}, \infty)
\end{equation}  those where $y=\infty$, by
\begin{equation}
\label{fifi2}
     f_3^y=(\infty, y'),\qquad f_4^y=(\infty, y'')
\end{equation}
those where $x=\infty$, and by
\begin{equation}
\label{fifi}
     f_5^y=(x^*, y^*)=\widehat \xi f_1^y-\o_2,\qquad f_6^y=(x^{**}, y^{**})=\widehat \xi f_2^y-\o_2
\end{equation}
the remaining ones. The location of these poles heavily depends on the signs of the branch points $x_4$ and $y_4$, see \cite[Cases I, II et III in Section 7]{KRIHES} or the proof of Lemma~\ref{lemma:poles_f_y_r_y} in this article. The signs of $x_4$ and $y_4$ can be expressed in terms of the geometry of the step set $\mathcal{S}$, see \cite[Remark 24]{KRIHES}.

In this section we are interested in the computation of the functions $F_{f_i^y,y}(\o)$ determining our main result, that is
 the series of Theorem~\ref{serie-introduction}.
 We first (Proposition \ref{lem:order-poles} and its proof) compute the degree of $F_{f_i^y,y}(\o)$ (i.e., the order of the pole $f_i^y$ of the function $f_y(\o)$). Then (below Proposition \ref{lem:order-poles}) we explain how concretely to compute the functions $F_{f_i^y,y}(\o)$. First of all, we need the following technical lemma:
\begin{lem}
\label{lem:tech}
On the fundamental rectangle $\o_1[0,1)+\o_2[0,1)$, $x$ has one double pole or two simple poles. It has one double pole if and only if
\begin{equation}
\label{eq:conditions_double}
     \delta_{1,0}=\delta_{1,1}=0\qquad \text{or}\qquad\delta_{1,0}=\delta_{1,-1}=0,
\end{equation}
and two simple poles if and only if the two conditions in \eqref{eq:conditions_double} are not satisfied.
\end{lem}
\begin{proof}
The uniformization formulas \eqref{expression_uniformization} show that on $\o_1[0,1)+\o_2[0,1)$, $x$ has one double pole (if and only if $x_4=\infty$) or two simple poles (if and only if $x_4\neq \infty$). Moreover, it is shown in \cite[Lemma 2.3.8]{FIM} that $x_4=\infty$ if and only if $\delta_{1,0}^2-4\delta_{1,1}\delta_{1,-1}=0$. Lemma \ref{lem:tech} follows.
\end{proof}

\begin{prop}
\label{lem:order-poles}
The order of the poles of $f_y(\o)$ is $1$, $2$ or $3$.
\end{prop}
\begin{proof}
We use the alternative representation of $f_y(\o)$ given in Lemma
\ref{fyfy}:
\begin{equation*}
     f_y(\o) = \frac{x'(\o)x(\o)}{2a(x(\o))},
\end{equation*}
where
\begin{equation*}
     a(x)/z = \delta_{1,1}x^2+\delta_{0,1}x+\delta_{-1,1}.
\end{equation*}
To be complete, we have the following cases to consider (for the non-singular walks):
\begin{enumerate}[label={\rm \arabic{*}.},ref={\rm \arabic{*}}]
     \item \label{case:100} $a(x)/z=x^2$;
     \item \label{case:010} $a(x)/z=x$;
     \item \label{case:001} $a(x)/z=1$;
     \item \label{case:110} $a(x)/z=x^2+x$;
     \item \label{case:101} $a(x)/z=x^2+1$;
     \item \label{case:011} $a(x)/z=x+1$;
     \item \label{case:111} $a(x)/z=x^2+x+1$.
\end{enumerate}

In the cases \ref{case:100} and \ref{case:110}, the poles are simple. Indeed, we have, respectively,
\begin{equation*}
     f_y(\o) = \frac{1}{2z}\frac{x'(\o)}{x(\o)},\qquad f_y(\o) = \frac{1}{2z}\frac{(x(\o)+1)'}{x(\o)+1}.
\end{equation*}
Case \ref{case:100} in particular covers Kreweras' and Gessel's models.

In the cases \ref{case:101} and \ref{case:111}, we have (with $\delta=0$ and $1$, respectively)
  \begin{equation*}
     f_y(\o) = \frac{1}{2z}\frac{x'(\o)}{x(\o)+\delta +1/x(\o)}.
\end{equation*}
If $\o_0$ is a simple pole of $x(\o)$, then it is obviously also a simple pole of $f_y(\o)$. If $\o_0$ is such that $x(\o_0)+\delta +1/x(\o_0)=0$, then it can be a zero of order one or two (remember that $x(\o)$ is an elliptic function of order two). In both cases, an expansion in the expression of $f_y(\o)$ above shows that $f_y(\o)$ has a pole of order one at $\o_0$.

In case \ref{case:010}, $f_y(\o)=x'(\o)/(2z)$. Therefore, if $x(\o)$ has one simple (resp.\ double) pole, then $f_y(\o)$ has a double (resp.\ triple) pole. Both situations can happen: the simple walk (Figure \ref{ExExEx}) is such that the pole is simple, and the walk on the left on Figure \ref{ExExEx-6} has a double pole.

In case \ref{case:001}, one has
\begin{equation*}
     f_y(\o) =  \frac{x'(\o)x(\o)}{2z}.
\end{equation*}
Therefore, if $x(\o)$ has one simple (resp.\ double) pole, then $f_y(\o)$ has a triple (resp.\ quintuple) pole. By definition of case \ref{case:001}, one has $\delta_{1,1}=\delta_{0,1}=0$. Further, if the pole of $x(\o)$ is double, one of the two conditions in \eqref{eq:conditions_double} must hold. This would imply that the walk is singular; this is a contradiction, since we do not consider these models here.

  \unitlength=0.6cm
\begin{figure}[t]
  \begin{center}
\begin{tabular}{cccc}
    \hspace{-0.9cm}
        \begin{picture}(4,4.5)
    \thicklines
    \put(1,1){{\vector(1,0){3.5}}}
    \put(1,1){\vector(0,1){3.5}}
    \thinlines
    \put(3,3){\vector(0,1){1}}
    \put(3,3){\vector(-1,0){1}}
    \put(3,3){\vector(1,-1){1}}
    \linethickness{0.1mm}
    \put(1,2){\dottedline{0.1}(0,0)(3.5,0)}
    \put(1,3){\dottedline{0.1}(0,0)(3.5,0)}
    \put(1,4){\dottedline{0.1}(0,0)(3.5,0)}
    \put(2,1){\dottedline{0.1}(0,0)(0,3.5)}
    \put(3,1){\dottedline{0.1}(0,0)(0,3.5)}
    \put(4,1){\dottedline{0.1}(0,0)(0,3.5)}
    \end{picture}
    \hspace{0.85cm}
&
    \begin{picture}(4,4)
    \thicklines
    \put(1,1){{\vector(1,0){3.5}}}
    \put(1,1){\vector(0,1){3.5}}
    \thinlines
    \put(3,3){\vector(-1,-1){1}}
    \put(3,3){\vector(1,-1){1}}
    \put(3,3){\vector(-1,1){1}}
    \put(3,3){\vector(1,0){1}}
    \put(3,3){\vector(0,1){1}}
    \linethickness{0.1mm}
    \put(1,2){\dottedline{0.1}(0,0)(3.5,0)}
    \put(1,3){\dottedline{0.1}(0,0)(3.5,0)}
    \put(1,4){\dottedline{0.1}(0,0)(3.5,0)}
    \put(2,1){\dottedline{0.1}(0,0)(0,3.5)}
    \put(3,1){\dottedline{0.1}(0,0)(0,3.5)}
    \put(4,1){\dottedline{0.1}(0,0)(0,3.5)}
    \end{picture}
    \hspace{0.85cm}
&   \begin{picture}(4,4)
\thicklines
    \put(1,1){\vector(1,0){3.5}}
    \put(1,1){\vector(0,1){3.5}}
    \thinlines
    \put(3,3){\vector(1,-1){1}}
    \put(3,3){\vector(0,-1){1}}
        \put(3,3){\vector(-1,0){1}}
    \put(3,3){\vector(-1,1){1}}
    \put(3,3){\vector(0,1){1}}
        \linethickness{0.1mm}
    \put(1,2){\dottedline{0.1}(0,0)(3.5,0)}
    \put(1,3){\dottedline{0.1}(0,0)(3.5,0)}
    \put(1,4){\dottedline{0.1}(0,0)(3.5,0)}
    \put(2,1){\dottedline{0.1}(0,0)(0,3.5)}
    \put(3,1){\dottedline{0.1}(0,0)(0,3.5)}
    \put(4,1){\dottedline{0.1}(0,0)(0,3.5)}
    \end{picture}
    \end{tabular}
  \end{center}
  \vspace{-4mm}
\caption{Examples in the proof of Proposition \ref{lem:order-poles}}
\label{ExExEx-6}
\end{figure}
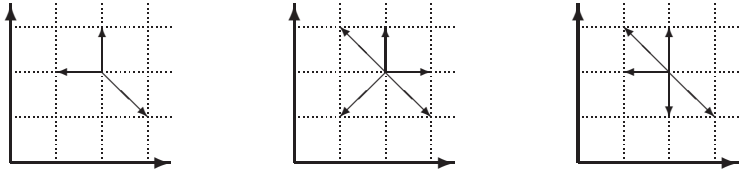

It remains to consider case \ref{case:011}. For similar reasons as for cases \ref{case:101} and \ref{case:111}, if $x(\o_0)+1=0$ then $\o_0$ is a pole of order one of $f_y(\o)$. If $x(\o)$ has a simple (resp.\ double) pole, then obviously $f_y(\o)$ has a double (resp.\ triple) pole. The two situations can happen, as the second and third examples on Figure \ref{ExExEx-6} illustrate.
\end{proof}

We now explain in which extent it is possible to find explicit expressions for the functions $F_{f_i^y,y}(\o)$. In fact, this heavily depends on the model under consideration. For instance, in the cases \ref{case:100} and \ref{case:110}, we do not need to simplify the uniformization \eqref{expression_uniformization}: if $\o_0$ is a pole (resp.\ zero) of order $p$ of $x(\o)$, then the residue of $f_y(\o)$ at $\o_0$ is $-p/(2z)$ (resp.\ $p/(2z)$). For the other cases, one needs more information on \eqref{expression_uniformization}. For example, in case \ref{case:010}, we clearly need to know the principal parts of $x(\o)$ at its poles in order to deduce the principal parts of $f_y(\o)$ at its poles. Unfortunately, there is no general formula, and this should be done model by model.

\section{Expression of the GFs in terms of the principal parts at their poles}
\label{sec:expression_GFs_series}
\label{subsec:para}

Although the dynamic of the meromorphic continuation procedure (see \eqref{rD}, \eqref{rDrD}, \eqref{cont} and \eqref{cont1}) is fairly easy, the initial values \eqref{rD_x} of $r_x(\o)$ in $\Delta_x$ and of $r_y(\o)$ in $\Delta_y$ can be obtained only via complicated integral representations of $Q(x,0)$ and $Q(0,y)$ given in \cite{Ra}, by solving certain boundary value problems. In this section we give another representation of $r_x(\o)$ and $r_y(\o)$ for any $\o \in {\bf C}$. It is valid for any $z \in \mathcal{H}$ (see \eqref{eq:first_definition_H} for its definition). We remind that for all $23$ models with finite group, $\mathcal{H} =(0, 1/\vert\mathcal{S}\vert)$, while for all $51$ non-singular models with infinite group, $\mathcal{H}$  is dense in $(0, 1/\vert\mathcal{S}\vert)$, but not equal to the latter, see \cite[Theorem 1]{KRIHES}. In \cite{KRIHES}, the set $\mathcal{H}$ has been also characterized as
\begin{equation*}
     \mathcal{H} = \{ z \in (0, 1/|\mathcal S|): \o_3/\o_2 \text{ is rational}\}.
\end{equation*}
Therefore, for any fixed $z\in\mathcal H$, one has \eqref{kn}, namely that ${\o_3}/{\o_2}={k}/{\ell}$ for some integers $\ell>k>0$ without common divisors. In what follows we work under this assumption.

The other representation of the GFs $r_x(\o)$ and $r_y(\o)$ that we shall obtain here is based on a constructive application of Mittag-Leffler's theorem. Specifically, the latter theorem states that it is possible to construct a meromorphic function with prescribed (discrete) set of poles and prescribed principal parts at these poles (for the definition see Section \ref{subsec:main_results}), see \cite[Section 4.13]{SG}. Our main result in this section (Theorem \ref{princi}) states that the GFs are the limits of the sums of the principal parts at their poles along certain sequences of parallelograms.

We first construct a sequence of parallelograms $(P_m)_{m\geq 1}$ on the universal covering ${\bf C}$ with the following properties:
 \begin{enumerate}
 \item \label{Property_para_1} $P_1 \subset P_2 \subset \cdots $;
 \item \label{Property_para_2} $\cup_{m\geq 1} P_m={\bf C}$;
 \item \label{Property_para_3} All the $P_m$'s are centered around $\o_{y_2}$;
 \item \label{Property_para_4} There exist constants $C_1,C_2>0$ such that for all
 $m\geq 1$,
   \begin{eqnarray}
 \sup_{\o \in \partial P_m} |r_y(\o)|\hspace{-2mm}& \leq&\hspace{-2mm} C_1 m, \label{c1} \\
  \inf_{\o \in \partial P_m} |\o- \o_{y_2}| \hspace{-2mm}&\geq&\hspace{-2mm} C_2 m \label{c2},
  \end{eqnarray}
where $\partial P_m$ denotes the boundary of $P_m$.
\end{enumerate}

The angle stone of this construction is the equality \eqref{kn}, under which there are many possibilities to perform it; below we just give one.
\begin{proof}[Construction of a family $(P_m)_{m\geq 1}$ satisfying \ref{Property_para_1}, \ref{Property_para_2}, \ref{Property_para_3} and \ref{Property_para_4}]
Let $f_y(\o)$ be the function defined in \eqref{fy}. It has been used in the meromorphic continuation procedure \eqref{cont1}. We recall that it has at most six distinct poles in the main parallelogram $\Pi_{0,0}+\o_3/2$ for $r_y(\o)$, we denote them by $p_1^{(0)},\ldots, p_6^{(0)}$.  Let $n\in\{1,\ldots, \ell-1\}$,  $i\in\{1,\ldots,6\}$,  and consider $p_i^{(0)}-n\o_3$. There exists a unique $m_{i,n} \in {\bf Z}_+$ such that $p_i^{(n)}=p_i^{(0)}-n\o_3+m_{i,n}\o_2 \in \Pi_{0,0}+\o_3/2$. Let
      \begin{equation*}
           \mathcal{R}=\{\o \in \Pi_{0,0}+\o_3/2 : \Re \o \ne p_i^{(n)}, \forall i\in\{1,\ldots,6\}, \forall n\in\{1,\ldots, \ell-1\}\}.
      \end{equation*}
In other words, $\mathcal{R}$ is the domain $\Pi_{0,0}+\o_3/2$ without the set of at most $6(\ell-1)$ vertical segments. With \eqref{cont1} combined with \eqref{kn}, one deduces that for any $\o \in \mathcal{ R}$ and any $m\in{\bf Z}$, $\o+m\o_3$ is not a pole of $r_y(\o)$.

Let us now fix $\o^+,\o^- \in \mathcal{R}$ symmetric with respect to $\o_{y_2}$ and such that $\Re \o^+> \Re \o^-$, and define the vertical straight lines
\begin{equation*}
     V_m^{+}=\{\o\in{\bf C} : \Re \o =\Re \o^+ + m\o_3\},\qquad V_m^{-}=\{\o\in{\bf C} : \Re \o =\Re \o^- - m\o_3\}.
\end{equation*}
Then by Equations \eqref{cont1} and \eqref{buzz}, and by the assumption \eqref{kn}
    the function $|r_y(\o)|$ on them can be estimated  as follows:
   \begin{equation}
\label{coucou}
        \sup_{\o \in V_m^+\cup V_m^-} |r_y(\o)|\leq  C_3 + m C_4,
   \end{equation}
where
\begin{equation}
\label{br}
C_3=\sup_{\o \in \Pi_{0,0}+\o_3/2 \atop \Re \o=\o^{+} \text{ or } \Re \o=\o^{-}} |r_y(\o)| <\infty
\end{equation}
 and
   \begin{equation*}
        C_4= \max_{n\in\{0,\ldots,\ell-1\}}\left\{ \sup_{\o \in \Pi_{0,0}+\o_3/2,\atop \Re \o=\o^{+}} |f_y(\o+n\o_3)|,  \sup_{\o \in \Pi_{0,0}+\o_3/2,\atop \Re \o=\o^{-}} |f_y(\o-n\o_3)|\right\}<\infty.
    \end{equation*}

Let us fix any $\o^{*}, \o^{** }\in \Pi_{0,0}+\o_3/2$ symmetric with respect to $\o_{y_2}$, $\Im \o^*> \Im \o^{**}$,
 and such that for all $i\in\{1,\ldots,6\}$,  $\Im \o^*\ne \Im p_i^{(0)}$ and $\Im \o^{**}\ne \Im p_i^{(0)}$.
Then due to $\o_2$-periodicity of $f_y(\o)$,
\begin{equation}
\label{c''}
C_5=\sup_{ \Im \o =\Im \o^* \hbox{\tiny\ or } \Im \o =\Im \o^{**}  } |f_y(\o)|<\infty.
\end{equation}
 It follows from \eqref{cont1}
 that $r_y(\o)$ does not have poles at points $\o \in {\bf C}$ such that $\Im\o \in \{\Im \o^*, \Im\o^{**}\}$,
  and in particular
\begin{equation}
\label{rc}
C_6=\sup_{ \Im \o =\Im \o^* \hbox{\tiny\ or } \Im \o =\Im \o^{**} , \atop
  \Re \o^{-}\leq   \Re \o \leq  \Re \o^{+} } |r_y(\o)|<\infty.
\end{equation}
Define now the horizontal straight lines
     \begin{equation*}
          H_p^{*}=\{\o\in{\bf C} : \Im \o =\Im \o^* + p|\o_1|\},\qquad H_p^{**}=\{\o \in{\bf C}: \Im \o =\Im \o^{**} - p|\o_1|\}
     \end{equation*}
for any $p \in {\bf Z}_+$.
Let us estimate $|r_y(\o)|$ on their segments bounded by the intersection points with $V_m^+$ and $V_m^-$. By
 \eqref{cont1}, \eqref{buzz}, \eqref{c''} and \eqref{rc}  the following bound   holds true for any $p \in {\bf Z}_+$:
\begin{equation}
\label{period}
 \sup_{\o \in H_p^* \cup H_p^{**},\atop \Re\o^- - m\o_3\leq   \Re \o \leq \Re\o^+ + m\o_3} |r_y(\o)|  =
 \sup_{ \Im \o=\Im \o^* \hbox{\tiny\ or } \Im \o= \Im \o^{**}, \atop  \Re\o^- - m\o_3\leq   \Re \o \leq \Re\o^+ + m\o_3} |r_y(\o)|
  \leq C_6 +m C_5<\infty.
\end{equation}

Let $(p_m)_{m \geq 1}$ be any  sequence of integers strictly increasing to infinity as $m\to \infty$ and such that $p_m \geq m$. Let us construct the parallelograms $P_m$'s bounded by $V_m^+$, $V_m^-$, $H_{p_m}^*$ and $H_{p_m}^{**}$, which are defined above. Properties
\ref{Property_para_1}, \ref{Property_para_2} and \ref{Property_para_3} are trivially satisfied for these $P_m$'s.
The estimate \eqref{c1} is ensured by \eqref{coucou} and \eqref{period}.  Note that \eqref{kn} was crucial to make \eqref{coucou} valid.
    Furthermore,
    \begin{equation*}
         \inf_{\o \in \partial P_m} |\o- \o_{y_2}| \geq m \min\{\o_3, \vert\o_1\vert\},
    \end{equation*}
which shows the estimate \eqref{c2}.  This concludes the construction of the $P_m$'s with required properties \ref{Property_para_1}, \ref{Property_para_2}, \ref{Property_para_3} and \ref{Property_para_4}.
\end{proof}

We are now ready to state and prove the main result of this section.

\begin{thm}
\label{princi}
Let $d_i$ be the poles of the function $r_y(\o)$ on ${\bf C}$ and let $R_{d_i,y}(\o)$ be the principal parts of $r_y(\o)$ at these poles. Then for any family of parallelograms $(P_m)_{m\geq 1}$ satisfying \ref{Property_para_1}, \ref{Property_para_2}, \ref{Property_para_3} and \ref{Property_para_4}, one has
\begin{equation}
\label{maindd}
     r_y(\o)-r_y(\o_{y_2})= \lim_{m \to \infty } \sum_{d_i\in P_m} \left\{R_{d_i,y}(\o)-R_{d_i,y}(\o_{y_2})\right\},\qquad \forall \o\in{\bf C}.
\end{equation}
\end{thm}

\begin{proof}
We follow \cite[Chapter V, proof of Theorem 2]{Shabat}, as suggested in \cite[page 163]{MA}.  We have
\begin{equation}
\label{sty}
     \frac{1}{2\pi i} \int_{\partial P_m} \frac{r_y(\zeta)}{\zeta-\o}\text{d}\zeta=\text{res}_{\o}\frac{r_y(\zeta)}{\zeta-\o} + \sum_{d_i \in P_m}
      \text{res}_{d_i}\frac{r_y(\zeta)}{\zeta-\o}=
   r_y(\o)+ \sum_{d_i \in P_m}
      \text{res}_{d_i }\frac{\sum_{d_i \in P_m}R_{d_i,y}(\zeta)}{\zeta-\o}.
      \end{equation}
 The rational function
\begin{equation*}
      G(\zeta)=\frac{\sum_{d_i \in P_m }R_{d_i,y}(\zeta)}{\zeta-\o}
\end{equation*}
has a residue at infinity equal to zero
   (if there is at least one pole of $r_y(\o)$ in $P_m$, that holds for all $m$ large enough).
     Then $\sum_{d_i \in P_m} \text{res}_{d_i} G(\zeta)+\text{res}_{\o} G(\zeta)=0$.
  Hence
  \begin{equation*}
  \sum_{d_i \in P_m} \text{res}_{d_i} G(\zeta)=-\text{res}_{\o} G(\zeta)= -\sum_{d_i
\in P_m} R_{d_i,y}(\o),
\end{equation*} so that
 \begin{equation}
 \label{sty1}
 \frac{1}{2\pi i} \int_{\partial P_m} \frac{r_y(\zeta)}{\zeta-\o}\text{d}\zeta=
      r_y(\o)-\sum_{d_i
\in P_m} R_{d_i,y}(\o).
\end{equation}
  In particular, for $\o=\o_{y_2}$ we have
 \begin{equation}
 \label{sty2}
 \frac{1}{2\pi i} \int_{\partial P_m} \frac{r_y(\zeta)}{\zeta-\o_{y_2}}\text{d}\zeta=
      r_y(\o_{y_2})-\sum_{d_i
\in P_m} R_{d_i,y}(\o_{y_2}).
\end{equation}
   Moreover, taking the derivative in \eqref{sty1} at $\o=\o_{y_2}$ leads to
\begin{equation}
 \label{sty3}
 \frac{1}{2\pi i} \int_{\partial P_m} \frac{r_y(\zeta)}{(\zeta-\o_{y_2})^2}\text{d}\zeta=
      r_y'(\o_{y_2})-\sum_{d_i
\in P_m} R_{d_i,y}'(\o_{y_2}).
\end{equation}
Note that since $r_y(\o)=r_y(-\o+2\o_{y_2})$ by \eqref{xieta}, then  for each $d_i$ pole of $r_y$, $-d_i+2\o_{y_2}$ is also a pole of $r_y$, and further$R_{d_i,y}'(\o_{y_2})= -R_{-d_i+2\o_{y_2},y}'(\o_{y_2})$. Since  $P_m$ is centered in $\o_{y_2}$, then either both $d_i,-d_i+2\o_{y_2}$ belong to $P_m$, or both do not. Hence $\sum_{d_i\in P_m} R_{d_i,y}'(\o_{y_2})=0$.

Subtracting from \eqref{sty1} the identity \eqref{sty2} and also \eqref{sty3} multiplied by the quantity $(\o-\o_{y_2})$, we find
\begin{multline}
\label{sty4} \frac{1}{2\pi i} \int_{P_m} r_y(\zeta)
\frac{(\o-\o_{y_2})^2}{(\zeta-\o_{y_2})^2 (\zeta-\o)}\text{d} \zeta
\\=
r_y(\o)-\sum_{d_i \in P_m} R_{d_i,y}(\o)-\left(
r_y(\o_{y_2})-\sum_{d_i \in P_m} R_{d_i,y}(\o_{y_2})\right)-
r_y'(\o_{y_2})(\o-\o_{y_2}).
\end{multline}
By the assumption \eqref{c1}, one has
   \begin{equation}
\label{vvv} \left| \frac{1}{2\pi i} \int_{P_m} r_y(\zeta)
\frac{(\o-\o_{y_2})^2}{(\zeta-\o_{y_2})^2 (\zeta-\o)}\text{d} \zeta \right|\leq \frac{|\o-\o_{y_2}|^2}{2\pi} \frac{C_1 m}{\inf_{\zeta \in \partial P_m}
|\zeta- \o|} \int_{\partial P_m}\frac{1}{|\zeta -\o_{y_2}|^2}\text{d}\zeta.
 \end{equation}
  By the assumption \eqref{c2}, one has
\begin{equation*}
\inf_{\zeta \in \partial P_m} |\zeta- \o| \geq \inf_{\zeta \in
\partial P_m} |\zeta- \o_{y_2}|-|\o_{y_2}-\o|\geq C_2 m
-|\o_{y_2}-\o|.
\end{equation*}
   We can also estimate
\begin{equation*}
     \int_{\partial P_m}\frac{1}{|\zeta -\o_{y_2}|^2}\text{d}\zeta \leq
  \frac{4\pi}{\inf_{\zeta \in \partial P_m}|\zeta-\o_{y_2}|} \leq
  \frac{ 4\pi}{C_2 m}.
  \end{equation*}
  It follows that \eqref{vvv} equals $O(1/m)$ as $m \to \infty$.
   Thus
\begin{equation*}
r_y(\o)-r_y(\o_{y_2})=\lim_{m \to \infty} \left(\sum_{d_i \in P_m}
R_{d_i,y}(\o)-\sum_{d_i \in P_m} R_{d_i,y}(\o_{y_2}  ) \right) +
r_y'(\o_{y_2})(\o-\o_{y_2}).
\end{equation*}
   The identity $r_y(\o)=r_y(-\o+2\o_{y_2})$ due to \eqref{xieta}
    implies
 $R_{d_i,y}(\o)=R_{-d_i+2\o_{y_2},y}(-\o+2\o_{y_2})$. Hence for any
 $m\geq 1$,
\begin{equation*}\sum_{d_i \in P_m} R_{d_i,y}(\o) = \sum_{d_i \in P_m} R_{d_i,y}(-\o+2\o_{y_2}).\end{equation*}
 Then for any $\o \in {\bf C}$,
  $r_y'(\o_{y_2})(\o-\o_{y_2})=r_y'(\o_{y_2})(-\o+\o_{y_2})$, thus
  $r_y'(\o_{y_2})=0$. This finishes  the proof of the theorem.
\end{proof}

The result of Theorem \ref{princi} has to be improved for several
reasons: first, it would be important to identify the poles of $r_y(\o)$
and the principal parts at them; second, it would be more convenient
to represent $r_y(\o)$ as an absolutely convergent series, independent
of the parallelograms $P_m$'s;  third, it would be useful to find
the unknown constant $r_y(\o_{y_2})$. These remarks give the
structure of the following sections. In Section~\ref{subsec:ppp} we
compute the principal parts at the poles of $r_y(\o)$ in terms of the
principal parts at the poles of $f_y(\o)$. In Section \ref{sec:mainth}
we represent $r_y(\o)$ as the sum of an absolutely convergent series,
where all terms are expressed via the principal parts at the (at
most) six poles of $f_y(\o)$ in the parallelogram $\Pi_{0,0}-\o_{x_1}$.
In Section \ref{sec:mainth} we also compute $r_y(\o_{y_2})$.

\section{Computation of the poles of the GFs and of the principal parts at them}
\label{subsec:ppp}

In this section we study the poles of $r_y(\o)$ and the principal parts at them.
 The first lemma deals with the poles of $r_y(\o)$ in the domain $\Delta$, which is defined in \eqref{Delta}.
 Let $f_y(\o)$ be the function defined in \eqref{fy}. We denote by $F_{d,y}(\o)$ the principal part of $f_y(\o)$ at the pole $d$.
\begin{lem}
\label{lemma:poles_f_y_r_y}
A point $d$ is a pole of $r_y(\o)$ in $\Delta$ if and only if it is a pole of the function $f_y(\o)$ in $\Delta$ with $\Re d< \Re \o_{x_1}$. Furthermore, at such a point $d$,
\begin{equation}
\label{dfgh}
     R_{d,y}(\o)=-F_{d,y}(\o),\qquad \forall \o\in{\bf C}.
\end{equation}
A point $d'$ is a pole of $r_y(\o)+f_y(\o)$ in $\Delta$ if and only if it is a pole of $f_y(\o)$ in $\Delta$ with $\Re d'> \Re \o_{x_1}$.
\end{lem}

\begin{proof}
We first prove the following facts:
\begin{enumerate}
\item\label{enu:lemma(i)}
$\Delta \cap \{\o\in{\bf C} : y(\o)=\infty\} \subset \{\o\in{\bf C} : \Re \o< \Re
\o_{x_1}\}\cup (\o_{x_1}+\o_1{\bf Z})$;
\item\label{enu:lemma(ii)}
$\Delta \cap \{\o\in{\bf C} : y(\widehat \xi\o)=\infty\} \subset \{\o\in{\bf C} : \Re \o> \Re
\o_{x_1}\}\cup (\o_{x_1}+\o_1{\bf Z})$;
\item\label{enu:lemma(iii)}
$\Delta \cap \{\o\in{\bf C} : x(\o)=\infty\} \subset \{\o\in{\bf C} : \Re \o> \Re
\o_{x_1}\}$.
\end{enumerate}

Consider first the models with $y_4<0$. For them, points $\o\in{\bf C}$ such that $y(\o)=\infty$
  are located on the vertical lines $\{\o \in{\bf C}: \Re\o=\Re \o_{y_4}+m \o_2\}$,
    $m\in {\bf Z}$. Only those on the line $\{\o\in{\bf C} : \Re\o=\Re
  \o_{y_4}\}$ can be in $\Delta$, and on this line we have
  $\Re \o_{y_4}<\Re \o_{x_1}$. This proves \ref{enu:lemma(i)} for these models.
        Points $\o$ with $y(\widehat \xi
   \o)=\infty$ are  on the  lines $\widehat \xi \{\o\in{\bf C} : \Re\o=\Re
  \o_{y_4}+m\o_2\}$, $m \in {\bf Z}$.
         None of them can be in $\Delta$,
    except for those on $\widehat \xi \{\o\in{\bf C} : \Re\o=\Re
  \o_{y_4}\}$. But we have $\Re \widehat \xi \o_{y_4}> \Re \o_{x_1}$. This concludes
   the proof of \ref{enu:lemma(ii)} for these models.

   Consider models with $y_4>0$ or $y_4=\infty$. It is proved in \cite{KRIHES} that points
     $\o$ with $y(\o)=\infty$ are located
    as follows: one of them, say $f_1$, is such that $\Im f_1=0$ and $\Re \o_{x_4}\leq \Re f_1\leq
    \Re \o_{y_4}<  \Re \o_{x_1}$;
 the other, say $f_2$, is symmetric to $f_1$ w.r.t.\ $\o_{y_4}$.  Furthermore $\Im f_2=0$ and  $\Re \o_{y_4}\leq \Re f_2\leq  \Im \o_{x_1}
    $ (in the limit case $y_4=\infty$, one has  $f_1=f_2=\o_{y_4}$).
    All other points where $y(\o)=\infty$ are  the shifts of $f_1,f_2$ by $m\o_2+p\o_1$, for
     $m,p \in {\bf
    Z}$. Only points of type $f_1+p\o_1$ and $f_2+p\o_1$ with $p\in {\bf Z}$
     can be in $\Delta$. We have for them  $\Re (f_i+p\o_1)= \Re f_i \leq  \Re \o_{x_1}$, $i\in\{1,2\}$.
Then either
   $\Re (f_i+p\o_1)<  \Re \o_{x_1}$ or  $\Re (f_i+p\o_1)= \Re \o_{x_1}=\Re (\o_{x_1}+p\o_1)$.
   In the last case the observation
 $\Im(f_i+p\o_1)=p\o_1=\Im(\o_{x_1}+p \o_1)$, $i\in\{1,2\}$, leads to $f_i+p\o_1=\o_{x_1}+p\o_1$
  and finishes the proof of \ref{enu:lemma(i)} for this class of models.
     Points $\o$ such that
     $y(\widehat \xi \o)=\infty$  are $\widehat \xi (f_1+m\o_2+p\o_1)$ and $\widehat \xi(f_2+
     m\o_2+p\o_1)$, for $m,p \in {\bf Z}$.
        Only points $\widehat \xi (f_i+p\o_1)$ , $i\in\{1,2\}$, can be in $\Delta$.
If $\Re( f_i+p\o_1) < \Re \o_{x_1}$, then
     $\Re \widehat \xi (f_i+p\o_1)> \Re \o_{x_1}$. If  $f_i+p \o_1=\o_{x_1}+p\o_1$, then
   $\widehat \xi (f_i+p\o_1)=\o_{x_1}+(-p+1)\o_1$.
 This proves assertion \ref{enu:lemma(ii)} for all these models.

  It remains to prove \ref{enu:lemma(iii)}.
   Consider $\o$ with $x(\o)=\infty$  first for models with $x_4<0$.
   Then the points of ${\bf C}$ with $x(\o)=\infty$
  lie on the vertical lines $\{\o\in{\bf C} : \Re\o=\Re \o_{x_4}+m \o_2\}$,
    $m\in {\bf Z}$. Only those on the line $\{\o\in{\bf C} : \Re\o=\Re
  \o_{x_4}+\o_2\}$ can be in $\Delta$, where we have
  $\Re \o_{x_4}+\o_2>\Re \o_{x_1}$.

  Let us finally prove  \ref{enu:lemma(iii)} for models such that $x_4>0$ or $x_4=\infty$. It has been proved in \cite{KRIHES} that points
     $\o$ with $x(\o)=\infty$ are located
    as follows: one of them, say $f_3$, is such that $\Im f_3=0$ and $\Re \o_{x_1}<\Re \o_{x_4}+\o_2\leq \Re f_3<
    \Re \o_{y_4}+\o_2$; the other, $f_4$, is symmetric to $f_3$ w.r.t.\ $\o_{x_4}+\o_2$ and $\Re \o_{x_1}< \Re \o_{y_2}\leq  \Re f_4 \leq \Re
    \o_{x_4}+\o_2$ (in the limit case $x_4=\infty$, one has  $f_3=f_4=\o_{x_4}+\o_2$).
    All other points with $x(\o)=\infty$ are their shifts by $m\o_2+p\o_1$, where $m,p \in {\bf
    Z}$. Only points $f_3+p\o_1$ and $f_4+p\o_1$ with $p\in {\bf Z}$
     can be in $\Delta$. We have $\Re (f_3+p\o_1)= \Re f_3 >\Re \o_{x_1}$ and $\Re (f_4+p\o_1)=\Re f_4
     > \Re \o_{x_1}$, that finishes the proof of item~\ref{enu:lemma(iii)}.

   Let $d$ be a pole of $r_y(\o)$ in $\Delta$.
 Since $r_y(\o)$ is holomorphic in $\Delta_y$,
    then $d\in \Delta_x$ with $|x(d)|<1$ and $r_x(d)\ne \infty$. Furthermore by \eqref{sqs2}, at $\o=d$,
     \begin{equation}
     \label{ddd}
  r_y(\o)=-r_x(\o)+Q(0,0)K(0,0)+x(\o)y(\o).
     \end{equation}
  Then necessarily $x(d)y(d)=\infty$, from where $y(d)=\infty$. By assertion \ref{enu:lemma(i)} above, we have
   either $\Re d< \Re \o_{x_1}$ or $d =\o_{x_1}+ p\o_1$ with some
$p\in {\bf Z}$. The last fact is impossible since for none of the models,
     $\lim_{x\to x_1}Y(x)=\infty$ and
   simultaneously  $\lim_{x\to x_1} xY(x)=\infty$  (remember that $x_1$ is a branch point of $Y(x)$ with $|x_1|<1$). Hence  $\Re d< \Re \o_{x_1}$.
        Consider now $\widehat \xi d$.
 We have $|x(\widehat \xi d)|=|x(d)|<1$, so that $\widehat \xi d \in \Delta_x\subset \Delta$
    and $\Re (\widehat \xi d) = \Re (-d + 2\o_{x_2})> \Re \o_{x_1}$. Then again by
    \ref{enu:lemma(i)}, $y(\widehat \xi d) \ne \infty$, so that $x(d)y(\widehat \xi d) \ne \infty$.
      Then $f_y(\o)=x(\o)[y(\widehat \xi
    \o)-y(\o)]$ has a pole at $\o=d$.
 Furthermore by \eqref{ddd}, one has
 \begin{equation}
 \label{ccc}
r_y(\o)=-r_x(\o)+K(0,0)-f_y(\o)+x(\o)y(\widehat \xi \o),
 \end{equation}
from where \eqref{dfgh} follows.

  Let $d$ be a pole of $f_y(\o)$ in $\Delta$ such that $\Re d< \Re \o_{x_1}$.
     Then necessarily one of the three following equalities holds: $y(d)=\infty$, $y(\widehat \xi d)=\infty$, or $x(d)=\infty$.
  By \ref{enu:lemma(i)}, \ref{enu:lemma(ii)} and \ref{enu:lemma(iii)}, $y(d)=\infty$, so that $d \in \Delta_x$ and
  $r_x(d) \ne \infty$.
   Thus $|x(d)|=|x(\widehat \xi d)|<1$, $\widehat \xi d \in \Delta_x \subset
   \Delta$ and $\Re (\widehat \xi d) >\Re \o_{x_1}$. By \ref{enu:lemma(i)} one has $y(\widehat \xi
   d )\ne \infty$, so that $x(d)y(\widehat \xi d)\ne \infty
    $. The point $d$ being a pole of $f_y(\o)$,  it follows that $x(d)y(d) =\infty$.
    By \eqref{ddd}, $d$ is a pole of $r_y(\o)$ in $\Delta$.
   Finally, \eqref{ccc} implies \eqref{dfgh}.

  It follows from \eqref{ccc} that any pole $d$ of $f_y(\o)$ in $\Delta$
  such that $\Re d < \Re \o_{x_1}$ (or equivalently, any pole $d$ of $r_y(\o)$ in
  $\Delta$)
  is not a pole of $r_y(\o)+f_y(\o)$ in $\Delta$.
    It was proved that any pole $d'$ of $f_y(\o)$ in $\Delta$
  with $\Re d > \Re \o_{x_1}$ is not a pole of $r_y(\o)$ in $\Delta$. Then
  $d'$ must be a pole of $r_y(\o)+f_y(\o)$ in $\Delta$.
    Finally, let us show that $f_y(\o)$ cannot have a pole $d''$ in $\Delta$ with
$\Re d''= \Re \o_{x_1}$.
In fact, for any pole $d''$ of $f_y(\o)$, one of the three equalities holds:
 $y(d'')=\infty$, $y(\widehat \xi d'')=\infty$, or $x(d'')=\infty$.  Assumption
   $\Re d''= \Re \o_{x_1}$
 combined with  \ref{enu:lemma(i)},  \ref{enu:lemma(ii)} and \ref{enu:lemma(iii)}  yields $d''=\o_{x_1}+p\o_1$,
with $p \in {\bf Z}$. In particular, $y(d'')=\infty$.  Then $x_1$ should  be a branch point of  $Y(x)$ with $|x_1|<1$,
   $\lim_{x\to x_1}Y(x)=\infty$ and $\lim_{x\to x_1}x(Y_0(x)-Y_1(x))=\infty$, which is true for none of the models.
  Hence $f_y(\o)$ has no poles $d''$ with $\Re d'' =\Re \o_{x_1}$, and
      Lemma \ref{lemma:poles_f_y_r_y} is proved.
\end{proof}

Our ultimate goal is to obtain explicit expressions for $r_x(\o)$ and $r_y(\o)$ from Theorem~\ref{princi}. The theorem which immediately follows is a key preliminary result: it gives expressions for the principal parts $R_{d_i,y}(\o)$ of $r_y(\o)$ in terms of the principal parts of the elliptic function $f_y(\o)$ at its poles.
The main tool for that are the meromorphic continuation formulas \eqref{cont} and \eqref{cont1}.

In what follows, we extend our notation $R_{d,y}(\o)$ of Theorem \ref{princi} for any $d \in {\bf C}$, as follows: $R_{d,y}(\o)$ will be the principal part of $r_y(\o)$ at $d$ (see Section \ref{subsec:main_results} for the definition) if $d$ is a pole of $r_y(\o)$, and $R_{d,y}(\o)=0$ if $d$ is not a pole of $r_y(\o)$. Further, we shall denote throughout by $\lfloor x\rfloor$ (resp.\ $\lceil x\rceil$) the lower integer part (resp.\ the upper integer part) of $x\in{\bf R}$.

\begin{thm}
\label{dva}
Assume  the assertion \eqref{kn}.
   Denote by $F_{f,y}(\o)$ the principal part of $f_y(\o)$ at a pole $f$.
\begin{enumerate}
\item\label{enu:eva:i}
     Let $d \in {\bf C}$ with $\Re d <  \Re \o_{y_1}$.
Let
\begin{equation*}\mathcal{N}_d^-=\{n\geq 0:
  d+n\o_3
\hbox{ is a pole of } f_y(\o) \hbox{ with } \Re \o_{x_1}- k \o_{2} < \Re
d+n\o_3< \Re \o_{x_1} \}.\end{equation*}
If  $d$ is  a pole of $r_y(\o)$,  then $\mathcal{N}_d^- \ne \emptyset$.
       Furthermore, for any $d  \in {\bf C}$ with $\Re d < \Re \o_{y_1}$
  and $\mathcal{N}_d^-\ne \emptyset$, we have:
   \begin{equation}
\label{rrdd}
R_{d,y}(\o)=\sum_{n \in \mathcal{N}_d^-}
  -(\lfloor n/\ell\rfloor +1)F_{d+n\o_3,y}(\o+n\o_3).
\end{equation}
\item\label{enu:eva:iii}
Let $d \in {\bf C}$ with $\Re d >  \Re \o_{y_1}$.  Let
\begin{equation*}
\mathcal{N}_d^+=\{n \geq 1 :
  d-n\o_3
\hbox{ is a pole of } f_y(\o) \hbox{ with } \Re \o_{x_1} < \Re d-n\o_3< \Re
\o_{x_1}+k\o_2\}.\end{equation*}
If  $d$ is  a pole of $r_y(\o)$,  then $\mathcal{N}_d^+ \ne \emptyset$.
   Furthermore, for any $d  \in {\bf C}$ with $\Re d >  \Re \o_{y_1}$
  and $\mathcal{N}_d^-\ne \emptyset$, we have:
\begin{equation}
\label{rrdd1}
R_{d,y}(\o)=\sum_{n \in \mathcal{N}_d^+}
  (\lfloor(n-1)/\ell\rfloor+1)F_{d-n\o_3,y}(\o-n\o_3).
  \end{equation}
\end{enumerate}
\end{thm}

\begin{proof}

Let $d $ be a pole of $r_y(\o)$ with $\Re d < \Re\o_{y_1}$.
  Let us first prove that $\mathcal{N}_d^-\ne \emptyset$.
Assume first that in addition $d \in \Delta$. By Lemma \ref{lemma:poles_f_y_r_y}, $d$ is a pole of $f_y(\o)$ in $\Delta $  and $\Re d < \Re \o_{x_1}$,
so that $\mathcal{N}_d^-\ne \emptyset$.   Let now $d$ be a pole of $r_y(\o)$ with $\Re d < \Re \o_{y_1}$, but $d \not\in \Delta$.  In what follows, we shall denote the domains of ${\bf C}$ on the left of $\Delta$ by
 ${D}^-$ and on the right of $\Delta$ by ${D}^{+}$.
   Then $d \in {D}^{-}$.  By \eqref{cont1}, we have that for any $\o \in D^-$,
\begin{equation}
\label{dfg}
 r_y(\o)=r_y(\o+n\o_3)-f_y(\o+(n-1)\o_3)-f_y(\o+(n-2)\o_3)-\cdots - f_y(\o),
 \end{equation}
  where $n$ is such that $\o+n\o_3 \in \Delta$, while $\o+(n-1)\o_3, \o+(n-2)\o_3, \ldots, \o
   \not\in \Delta$.
      If $d \in { D}^{-}$ is a pole of $r_y(\o)$,
then either $d+t\o_3$ is a pole of $f_y(\o)$ for some $t\in\{0,\ldots,
n-1\}$ (then of course $\Re (d+t\o_3)< \Re \o_{x_1}$ since $d+t\o_3 \in D^-$),  or $d+n\o_3$ is a pole of $r_y(\o)$ in $\Delta$, that is (by Lemma
\ref{lemma:poles_f_y_r_y}) a pole of $f_y(\o)$ in $\Delta$  with $\Re d+n\o_3 < \Re
\o_{x_1}$. This proves that $\mathcal{N}_d^-\ne \emptyset$.

   Let us note that by \eqref{kn}, for any $\o$ pole of $f_y(\o)$,
 $\o-\ell\o_3=\o-k\o_2$ is also a pole of $f_y(\o)$, since this function
   is $\o_2$-periodic, and so is $\o-m\ell \o_3$, for any $m\geq 1$.  Then
 \eqref{rrdd}   follows from Lemma \ref{lemma:poles_f_y_r_y},
   Equations \eqref{dfgh}  and  \eqref{dfg}.

Let  $d$ be a pole of $r_y(\o)$ with $\Re d >  \Re\o_{y_1}$.
Then $\Re d  > \Re \o_{y_1}>\Re \o_{x_1}$ and, by Lemma \ref{lemma:poles_f_y_r_y},
  $d \not \in \Delta$. Then $d \in { D}^{+}$.  By \eqref{cont1}, we have that for any
  $\o \in D^+$,
\begin{equation}
\label{dfg1}
r_y(\o)=r_y(\o-n\o_3)+f_y(\o-n\o_3)+f_y(\o-(n-1)\o_3)+\cdots +
f_y(\o-\o_3),
\end{equation}
 where $n$ is chosen such that $\o-n\o_3 \in \Delta$, but $\o-(n-1)\o_3,\o-(n-2)\o_3,\ldots, \o\not \in
 \Delta$.
 If $d \in { D}^{+}$ is a pole of $r_y(\o)$,
then either $d-t\o_3$ is a pole of $f_y(\o)$ for some $t\in\{1,\ldots,
n-1\}$ (then of course $\Re d-t\o_3> \Re \o_{x_1}$, since $d-t\o_3 \in D^+$), or $d-n\o_3$ is a pole of $f_y(\o)+r_y(\o)$.
  By Lemma \ref{lemma:poles_f_y_r_y}, the last fact is possible if and only if $d-n\o_3$ is a pole
    of $f_y(\o)$ in $\Delta$ with
  $\Re (d-n\o_3) > \Re \o_{x_1}$.
    Hence  $\mathcal{N}_d^+\ne \emptyset$.

     By \eqref{kn}, for any $\o$ pole of $f_y(\o)$ and any $m\geq 1$,
 $\o+m\ell\o_3=\o+mk\o_2$ is also a pole of $f_y(\o)$, due to
    $\o_2$-periodicity of $f_y(\o)$.  Then  \eqref{rrdd1} follows from Lemma \ref{lemma:poles_f_y_r_y}  and Equation \eqref{dfg1}.
\end{proof}

We remark that a point $d$ such that $\Re d<  \Re\o_{y_1}$  (resp.\ $\Re d>  \Re\o_{y_1}$) and $\mathcal{N}_d^{-}\ne \emptyset$   (resp.~$\mathcal{N}_d^{+}\ne \emptyset$) is not  necessarily a pole of $r_y(\o)$, since in the sum \eqref{dfg} (resp.~\eqref{dfg1}),  terms producing poles of $r_y(\o)$ can compensate each other. However, the formula  \eqref{rrdd}  (resp.~\eqref{rrdd1}) is valid for any such $d$; if $d$ is not a pole, then obviously the RHS in the latter equation equals $0$.

\section{Representation of the GFs in terms of convergent  series}
\label{sec:mainth}

Using Theorem \ref{princi} and the analysis of poles and principal
parts of $r_y(\o)$ in Section~\ref{subsec:ppp}, we now compute the GFs
$r_x(\o)$ and $r_y(\o)$ independently of the construction of the
parallelograms $(P_m)_{m\geq 1}$. In the next theorem the GFs are
expressed in terms of the principal parts of function
$f_y(\o)$ (resp.\
$f_x(\o)$) in \eqref{fy} at their (at most six)
poles in the parallelogram $\Pi_{0,0}-\o_{x_1}$  (resp.~$\Pi_{0,0} +
\o_{y_1}$). This is the main result of the present paper; it is valid for
any of $74$ non-singular models of walks and any $z \in
\mathcal{H}$. It has been announced in
Theorem~\ref{serie-introduction}.

\begin{thm}
\label{serie}
Assume the assertion \eqref{kn}. Let the $f_i^y$'s be the poles of the function $f_y(\o)$, and let $F_{f_i^y,y}(\o)$ be the principal parts of this function at them. We have
\begin{equation}
\label{AA}
 r_y(\o)-r_y(\o_{y_2})
 =\sum_{p=-\infty}^{\infty}\sum_{n=0}^{\infty}\sum_{s=0}^{k-1}\sum_{f_i ^y\in \Pi_{0,0}-\o_{x_1}}
 A_{s,p,n}^{f_i^y} (\o),
 \end{equation}
 where
 \begin{align}
  A_{s,p,n}^{f_i^y}(\o) =
 -&(\lfloor n/\ell\rfloor+1)
     F_{f_i^y,y}(\o+s\o_2+n\o_3+p\o_1)\label{ana}\\
   -&
      (\lfloor n/\ell\rfloor+1)F_{f_i^y,y}(-\o+2\o_{y_2}+s\o_2+n\o_3+p\o_1)
   \nonumber\\
+2&(\lfloor n/\ell\rfloor+1)
F_{f_i^y,y}(\o_{y_2}+s\o_2+n\o_3+p\o_1).\nonumber
\end{align}
   The series \eqref{AA} of terms  $(\sum_{f_i ^y\in \Pi_{0,0}-\o_{x_1}}
 A_{s,p,n}^{f_i^y} (\o))$  is absolutely convergent.

Let the $f_i^x$'s be the poles of the function $f_x(\o)$,
 and let $F_{f_i^x,x}(\o)$ be the principal parts of this function at them.  We have
\begin{equation}
\label{BB} r_x(\o)-r_x(\o_{x_2})
 =\sum_{p=-\infty}^{\infty}\sum_{n=0}^{\infty}\sum_{s=0}^{k-1}
\sum_{f_i ^y\in \Pi_{0,0}+\o_{y_1}}
 B_{s,p,n}^{f_i^x} (\o) ,
 \end{equation}
 where
 \begin{align}
  B_{s,n,p}^{f_i^x}(\o)=
 -&(\lfloor n/\ell\rfloor+1) F_{f_i^x,x}(\o-s\o_2-n\o_3-p\o_1)\label{BRRR} \\
   -&(\lfloor n/\ell\rfloor+1)F_{f_i^x,x}(-\o+2\o_{x_2}-s\o_2-n\o_3-p\o_1)\nonumber\\
+2&(\lfloor n/\ell\rfloor+1)
       F_{f_i^x,x}(\o_{x_2}-s\o_2-n\o_3-p\o_1).\nonumber
\end{align}
   The series \eqref{BB} of terms  $(\sum_{f_i ^y\in \Pi_{0,0}+\o_3/2+\o_{y_1}}
 B_{s,p,n}^{f_i^x} (\o))$  is absolutely convergent.

Let $\o_0^x \in \Delta_x$ be such that $x(\o_0^x)=0$ and let $\o_0^y
\in \Delta_y$ be such that $y(\o_0^y)=0$.
  Then $r_x(\o)$ and $r_y(\o)$ can be found from \eqref{AA} and \eqref{BB} as follows:
   \begin{align}
r_y(\o)&=(r_x(\o_x^0)-r_x(\o_{x_2}) )-(r_x(\o)-r_x(\o_{x_2})) +x(\o)y(\o), \label{cs}\\
r_x(\o)&= (r_y(\o_y^0)-r_y(\o_{y_2}) )\hspace{0.45mm}-(r_y(\o)-r_y(\o_{y_2})) \hspace{0.42mm}+x(\o)y(\o).
   \label{cs1}
\end{align}
\end{thm}

\begin{proof}
Denote by
\begin{equation*}
     F_{f_i^y,y}(\o)=\sum_{n=1}^{t}\frac{r_{i,n}}{(\o-f_i^y)^n}
\end{equation*}
the principal part (see Section \ref{subsec:main_results} for the definition) of $f_y(\o)$
      at the pole $f_i^y$. Since $f_y(\o)$ is an elliptic function with
      periods $\o_1,\o_2$, it follows that $\sum_{i=1}^6 r_{i,1}=0$ (it is a classical property of elliptic functions that the sum of the residues in a fundamental parallelogram is zero, see \ref{sum_residues}, Lemma~\ref{lemma_properties_wp} in Appendix \ref{appendix-elliptic},
  where we have gathered all needed properties of elliptic functions).
    This fact implies the  convergence of the series
\begin{equation*}
\sum_{p=-\infty}^{\infty}\sum_{n=0}^{\infty}\sum_{s=0}^{k-1}
 \left|\sum_{f_i ^y\in \Pi_{0,0}-\o_{x_1}}
 A_{s,p,n}^{f_i^y} (\o)\right|.
 \end{equation*}

 We use the construction of the parallelograms $P_m$'s in Section \ref{sec:expression_GFs_series}
 bounded by $V_m^+$, $V_m^-$, $H_{p_m}^*$ and $H_{p_m}^{**} $.
    They are all centered in $\o_{y_2}$, so that  $d \in P_m$  is equivalent to $ \widehat \eta d \in P_m$.
 Since $r_y(\o)=r_y(\widehat \eta \o)$, then for any pole $d$ of $r_y(\o)$,
   $R_{\widehat \eta d,y}(\o)= R_{d,y}(\widehat \eta \o) = R_{d,y}(-\o+2\o_{y_2})$.   For any $d \in P_m$ with $\Re d < \Re \o_{y_2}= \Re \o_{y_1}$,
  there corresponds $ \widehat \eta d \in P_m$ with $\Re \widehat \eta d > \Re \o_{y_1}$ and vice versa.
   Points  with real part equal to $\Re \o_{y_1}$ belong to $\Delta_y$ and therefore cannot be poles of $r_y(\o)$.
  Then by Theorem \ref{princi},
\begin{align}
r_y(\o)-r_y(\o_{y_2})& =
\lim_{m \to \infty} \sum_{d_i \in P_m \text{ poles of } r_y \atop\phantom{o}} \left\{R_{d_i,y}(\o) -R_{d_i,y}(\o_{y_2})\right\}\nonumber\\
&=
\lim_{m \to \infty} \sum_{d_i \in P_m  \text{ poles of } r_y\atop \Re d_i < \Re \o_{y_1} }\left\{ R_{d_i,y}(\o)+R_{d_i,y}(-\o+2\o_{y_2})-2R_{d_i,y}(\o_{y_2})\right\}.
\label{qsqs}
\end{align}
    By Theorem \ref{dva} \ref{enu:eva:i},  for any $d$ pole of $r_y(\o)$ with $\Re d < \Re \o_{y_1}$, we have $\mathcal{N}_d^-\ne \emptyset$. Then  for any $\o \in {\bf C}$,
\begin{equation}
\label{skl}
 \sum_{d_i \in P_m \text{ poles of } r_y\atop \Re d_i < \Re \o_{y_1} } R_{d_i,y}(\o)
   =   \sum_{d_i \in P_m\atop  \Re d_i < \Re \o_{y_1} \text{ and }  \mathcal{N}_{d_i}^-\ne \emptyset } R_{d_i,y}(\o),
\end{equation}
  where the sum in the RHS above is taken over all points of $P_m$ (not only the poles)
 with $\Re d <\Re \o_{y_1}$   and  $ \mathcal{N}_{d}^- \ne \emptyset$.
     For any $d$   with  $\Re d <\Re \o_{y_1}$   and  $ \mathcal{N}_d^- \ne \emptyset$,  there exist
 $n \geq 0$ and $f_i^y$  pole of $f_y(\o)$  with $ \Re \o_{x_1}- k \o_2 < \Re f_i^y<  \Re \o_{x_1}$  such that $d+n\o_3=f_i^y$.
     If in addition  $d \in P_m$, then $\Re d \geq \Re \o^- - m\o_3$, so that $n \leq m - \Re(\o^--f_i^y)/\o_3$ and also
   $   \Im(\o^*-p_m\o_1)\leq   \Im d =\Im f_i^y \leq  \Im(\o^{**}+p_m\o_1)$.
On the other hand, for any $f_i^y$ with real and imaginary parts bounded as above and for
   any non-negative integer $n \leq m - \Re(\o^- -f_i^y)/\o_3$, the point $f_i^y-n\o_3$ belongs to $P_m$ and it is such that $ \mathcal{N}_d^- \ne \emptyset$.  Therefore by \eqref{rrdd} of Theorem \ref{dva}, the quantity in the RHS of \eqref{skl} equals
\begin{equation}
 \label{eq:the_sum_in}
 \sum_{ f_i^y :
    \Im(\o^{*}-p_m\o_1)\leq \Im f_i^y \leq \Im (\o^{**}+p_m\o_1)\atop
     \Re \o_{x_1}- k \o_2 < \Re f_i^y<  \Re \o_{x_1} } \sum_{n=0}^{ \lfloor m - \Re(\o^--f_i^y)/\o_3 \rfloor}-(\lfloor n/\ell \rfloor +1) F_{f_i^y,y}(\o+n\o_3).
\end{equation}

  For any $f_i^y$ in the sum above, there exist unique  $s \in \{0,1,\ldots, k-1\}$ and $p \in {\bf Z}$ such that
  $f_i^y=\bar{f}_i^y -s \o_2 -p\o_1$, where $\bar f_i^y \in \Pi_{0,0} -\o_{x_1}$ is a pole of $f_y(\o)$.
   Furthermore, the bounds for $\Im f_i^y$  are equivalent to
   $ -p_m+\Im (\bar f_i^y- \o^{**})/|\o_1| \leq   p \leq p_m+\Im (\bar f_i^y-\o^{*})/|\o_1|$. We deduce that for any $\o \in {\bf C}$, the sum in \eqref{eq:the_sum_in} equals
\begin{equation}
\label{lmlm}\sum_{ f_i^y \in \Pi_{0,0}-\o_{x_1}}\hspace{-0.5mm}\sum_{s=0}^{k-1}\sum_{p= \lceil -p_m+  \Im (\o^{**}- f_i^y)/|\o_1| \rceil }^{ \lfloor p_m + \Im (\o^{*}- f_i^y)/|\o_1| \rfloor}\hspace{-1mm}
 \sum_{n=0}^{ \lfloor m - \Re(\o^--f_i^y)/\o_3 \rfloor}\hspace{-0.5mm} -(\lfloor n/\ell \rfloor +1)
  F_{f_i^y,y}(\o+n\o_3+s\o_2+p\o_1).
\end{equation}
 By \eqref{ana}, \eqref{qsqs},  \eqref{skl} and \eqref{lmlm}, we obtain
\begin{equation}
r_y(\o)-r_y(\o_{y_2})=\lim_{m\to \infty}
 \sum_{ f_i^y \in \Pi_{0,0}-\o_{x_1}}\sum_{s=0}^{k-1} \sum_{p = \lceil -p_m + \Im (\o^{**}-f_i^y)/|\o_1|\rceil }^{ \lfloor
 p_m+ \Im (\o^{*}-f_i^y)/|\o_1| \rfloor}
 \sum_{n=0}^{ \lfloor m - \Re(\o^--f_i^y)/\o_3 \rfloor} A_{s,n,p}^{f^i_y}(\o).\label{dklp}
\end{equation}
    To conclude the proof of \eqref{AA}, it remains to show that the difference
\begin{multline}
 {\sum_{ f_i^y \in \Pi_{0,0}-\o_{x_1}}\sum_{s=0}^{k-1}
  \sum_{p = \lceil -p_m + \Im (\o^{**}-f_i^y)/|\o_1| \rceil }^{ \lfloor p_m+  \Im (\o^{*}-f_i^y)/|\o_1| \rfloor}
 \sum_{n=0}^{ \lfloor m - \Re(\o^--f_i^y)/\o_3 \rfloor} A_{s,n,p}^{f^i_y}(\o)}
{} \\-
\sum_{p=-p_m}^{p_m} \sum_{n=0}^m \sum_{s=0}^{k-1} \sum_{ f_i^y \in \Pi_{0,0}-\o_{x_1}} A_{s,n,p}^{f^i_y}(\o)  \label{poj}
\end{multline}
converges to zero as $m\to \infty$. To that purpose, it is sufficient to prove that for any fixed $f_i^y \in \Pi_{0,0}-\o_{x_1}$ and $s \in \{0,1,\ldots, k-1\}$, the limit of
\begin{equation}
\label{klmp}
 \sum_{p = \lceil -p_m+  \Im (\o^{**}-f_i^y)/|\o_1| \rceil }^{ \lfloor  p_m+\Im (\o^{*}-f_i^y)/|\o_1| \rfloor}
 \sum_{n=0}^{ \lfloor m - \Re(\o^--f_i^y)/\o_3 \rfloor} A_{s,n,p}^{f^i_y}(\o) -  \sum_{p=-p_m}^{p_m} \sum_{n=0}^m  A_{s,n,p}^{f^i_y}(\o)
\end{equation}
is $0$ as $m\to \infty$. Using the definition \eqref{ana}, it is
easy to show that for any $f_i^y$ and any $s \in \{0,1,\ldots, k-1\}$,
\begin{equation}
\label{exploit}
     \sum_{p=-\infty}^{\infty}| A_{s,n,p}^{f^i_y}(\o)|=O(1/n),\qquad n\to \infty,\qquad
     \sum_{n=0}^{\infty}| A_{s,n,p}^{f^i_y}(\o)|=O(1/p),\qquad p\to \infty,
\end{equation}
from where \eqref{klmp} follows. Hence \eqref{AA} is proved.

The proof of \eqref{BB} for $r_x(\o)$ is completely analogous and is omitted. In order to show \eqref{cs} and \eqref{cs1},
we notice that $Q(0,0) K(0,0)=r_x(\o_0^x)=r_y(\o_0^y)$. This way,
\eqref{cs} and \eqref{cs1} are immediate corollaries of
\eqref{sqs2}.
\end{proof}

\section{Algebraicity and holonomy of the GFs}
\label{subsec:orbit-sums}

Theorem~\ref{serie} determines the GFs $Q(x,0)$  and $Q(0,y)$ in terms of the principal parts of functions $f_x(\o)$ and $f_y(\o)$ at their poles in a parallelogram of periods. We would like now to identify the algebraic nature of these functions (question \ref{Challenge_2} stated at the very beginning of this paper) in the same terms.

Given a pole $f_i^y \in \Pi_{0,0}-\o_{x_1}$, consider the sum
\begin{equation*}
     \mathcal{F}_{f_i^y}^y(\o)= F_{f_i^y,y}(\o)+F_{f_i^y+\o_3,y}(\o+\o_3)+\cdots+ F_{f_i^{y}+(\ell-1)\o_3,y}(\o+(\ell-1)\o_3).
\end{equation*}
(We recall that above, for $t\in\{0,\ldots,\ell-1\}$, $F_{f_i^y+t\o_3,y}$ is the principal part of $f_y(\o)$ at $f_i^y+t\o_3$. In particular, for $t$ such that $f_i^y+t\o_3$ is not a pole of $f_y(\o)$, the corresponding term $F_{f_i^y+t\o_3,y}$ is identically zero.)

\begin{prop}
\label{prop:fff}
Under the assumption \eqref{kn}, the
functions $x \mapsto Q(x,0)$ and $y\mapsto Q(0,y)$  are
 algebraic functions with $2k$ branches if and only if
\begin{equation}
\label{fff}
     \mathcal{F}_{f_i^y}^y (\o)=0,\qquad \forall \o \in {\bf C},\qquad \forall i\in\{1,\ldots,4\},
\end{equation}
with the numbering of the poles \eqref{fifi1}--\eqref{fifi} as above.
\end{prop}

\begin{proof}
Under the assumption \eqref{fff}, for all $i\in\{1,\ldots,4\}$, there exists $j \in \{1,\ldots,6\}$ (with an appropriate $F_{f_j^y}(\o)$) such that $f_i^y+n\o_3=f_j^y+s\o_2$ for some $n\in\{1,\ldots, \ell-1\}$ and $s\in\{0,\ldots, k-1\}$. Since
\begin{equation*}
     f_5^y=\widehat \xi \widehat \eta f_1^y=f_1^y-\o_3,\qquad f_6^y=\widehat \xi \widehat \eta f_2^y=f_2^y-\o_3,
\end{equation*}
then under \eqref{fff}, the sum $\mathcal{F}_{f_i^y}^y(\o)$ is also zero for $i\in\{5,6\}$ and for all $\o \in {\bf C}$. Therefore, under \eqref{fff}, the elliptic function (which in some sense is an orbit sum, see \cite{BMM})
\begin{equation}
\label{eq:orbit_sum}
     \mathcal{O}(\o)=f_y(\o)+f_y(\o+\o_3)+\cdots +f_y(\o+(\ell-1)\o_3),\qquad \forall \o\in{\bf C},
\end{equation}
has no poles in ${\bf C}$ and thus, by a well-known property of elliptic functions (see Property \ref{elliptic_poles_0}) it is a constant: there exists $C\in{\bf C}$ such that $\mathcal{O}(\o)=C$, for all $\o\in{\bf C}$. We now prove that $C=0$. One has in particular $\mathcal{O}(\o_{y_2}+\o_3)=\mathcal{O}(\o_{y_2}-\ell \o_3)=C$, and thus, by Equations \eqref{dfg} and \eqref{dfg1},
\begin{align*}
     r_y(\o_{y_2}-\ell\o_3)=&\ r_y(\o_{y_2})- \mathcal{O}(\o_{y_2}-\ell \o_3)=r_y(\o_{y_2})-C,\\
r_y(\o_{y_2}+\ell\o_3)=&\ r_y(\o_{y_2})+\mathcal{O}(\o_{y_2}+\o_3)\hspace{1.6mm}= r_y(\o_{y_2})+C.
\end{align*}
Moreover, it follows from \eqref{xieta} that $r_y(\o_{y_2}-\ell\o_3)=r_y(\o_{y_2}+\ell\o_3)$, from where $C=0$.
Hence, under \eqref{fff}, we conclude that $\mathcal{O}(\o)=0$ for all $\o \in {\bf C}$.
 It follows again from \eqref{dfg} and \eqref{dfg1} that the function $r_y(\o)$ is $\ell\o_3=k\o_2$-periodic, and then by \eqref{sqs2}, so is $r_x(\o)$.
 By Property \ref{algebraic_theorem}, $r_x(\o)$ and $r_y(\o)$ are algebraic in $x(\o)$ and $y(\o)$ respectively.
On the other hand, if \eqref{fff} is not satisfied, then $\mathcal{O}(\o)$ is not identically zero. It has been shown in \cite[Section 6]{KRIHES}
 that under this last condition, the functions $x \mapsto Q(x,0)$ and $y\mapsto Q(0,y)$ are holonomic, but not algebraic.
\end{proof}

It is worth noting that the condition \eqref{fff}, which is stated for $f_y(\o)$, is equivalent to its analogue for the function $f_x(\o)$. Further, it is immediate from Proposition \ref{prop:fff} that under \eqref{fff}, the residues at the poles of $r_x(\o)$ and $r_y(\o)$ are bounded on ${\bf C}$.
  Moreover,  under \eqref{fff}, the terms
  \begin{equation*}
  A_{s,n,p}^{f_i^y}(\o)
  \end{equation*} defined in \eqref{ana}  can be permuted and regrouped in their sum over
\begin{equation*}
p\in {\bf Z},\qquad n\geq 0,\qquad s\in\{0,\ldots, k-1\},\qquad f_i^y \in \Pi_{0,0}-\o_{x_1},
\end{equation*}
in such a way that the coefficients in front of $F_{f_i^y,y}$ are  bounded and this infinite sum stays equal to the one of \eqref{AA}.
This condition can never hold true for models where $x_4>0$ or $y_4>0$ because of the position of the poles of $f_y(\o)$, see \cite{KRIHES}.  We do not know if for models with infinite group and such that $x_4\in(-\infty,0]\cup\{\infty\}$ and $y_4\in(-\infty,0]\cup\{\infty\}$, there exist values of $z \in \mathcal{H}$ such that \eqref{fff} is satisfied. But among the $23$ models with finite group, it is satisfied exactly for $4$ models, and in particular for Kreweras' and Gessel's walks. 

\section{Expressions of $Q(x,0)$ and $Q(0,y)$ in terms of the variables $x$ and $y$}
\label{sec:expressions_x/y}

Our main result (Theorem \ref{serie-introduction}, restated in full detail in Theorem \ref{serie}) gives expressions for the functions $r_x(\o)$ and $r_y(\o)$ as infinite series of rational functions of $\o$. The functions $x(\o)$ and $y(\o)$ being themselves expressed with Weierstrass $\wp$-functions, our results provide expressions for $Q(x,0)$ and $Q(0,y)$ as infinite series of rational functions of inverse Weierstrass functions. Though explicit, these expressions may appear complicated, and it is natural to search for simpler expressions of $Q(x,0)$ and $Q(0,y)$. Furthermore, in some particular cases (as Kreweras' one, see Proposition \ref{prop:BMM_Kreweras}, taken from \cite{BK2,BM,BMM,FIM,FL1,FL2,Kreweras}, which gives a very simple expression of $Q(x,0)=Q(0,x)$), we know that there are simple algebraic or holonomic expressions for the GFs.

In this section we explain how to obtain such expressions of $Q(x,0)$ and $Q(0,y)$ in terms of $x$ and $y$. We shall not state a general result (which, in some sense, is not reachable), but we shall explain how, model by model, it is possible to obtain such expressions. As we shall see, this reasoning has some algorithmic insights. We shall consider two different cases, according to whether the orbit-sum \eqref{eq:orbit_sum}
\begin{equation*}
     \mathcal{O}(\o)=f_y(\o)+f_y(\o+\o_3)+\cdots +f_y(\o+(\ell-1)\o_3),\qquad \forall \o\in{\bf C},
\end{equation*}
is identically zero or not. We already know (see Section \ref{subsec:orbit-sums}) that the orbit-sum if identically zero if and only if the GFs are algebraic. We recall that $r_y(\o+\o_3)= r_y(\o)+ f_y(\o)$ (see \eqref{cont1}), from where (with $\o_3/\o_2=k/\ell$)
\begin{equation}
\label{eq:trans}
     r_y(\o+k\o_2)= r_y(\o)+\mathcal{O}(\o),\qquad \forall \o\in{\bf C}.
\end{equation}

\subsection{Algebraic case}
\label{subsec:alca}

We first assume that $\mathcal{O}(\o)$ is identically equal to zero. It follows from \eqref{eq:trans} that the function $r_y(\o)$ is elliptic with periods $\o_1,k\o_2$. We now describe the three things to do so as to obtain an expression of $Q(0,y)$ in terms of $y$ (up to an additive constant\footnote{This additive constant is not a problem, since we can use the duality of $Q(x,0)$ and $Q(0,y)$ to find it; see Theorem \ref{serie-introduction}, where we already used this idea.}):
\begin{itemize}
     \item Find the poles $\widetilde\o_i$ of $r_y(\o)$ in a parallelogram of size $\o_1,k\o_2$;
     \item Deduce an expression of $r_y(\o)$ in terms of the functions $\zeta(\o-\widetilde\o_i;\o_1,k\o_2)$ and its derivatives (we recall that $\zeta'(\o;\o_1,k\o_2)=-\wp(\o;\o_1,k\o_2)$);
     \item Deduce an expression of $r_y(\o)$ in terms of $y(\o)$, and finally an expression of $K(0,y)Q(0,y)$ in terms of the variable $y$.
\end{itemize}
Here are more details. The first point is sufficient to determine $r_y(\o)$ up to an additive constant, thanks to standard properties of elliptic functions, see, e.g., \ref{expression_elliptic_zeta_generalization} (this is here that we use the fact that the orbit-sum is zero). To find the poles of $r_y(\o)$ we can use Theorem \ref{dva} (and more generally the results of Section \ref{subsec:ppp}), which precisely gives the poles of $r_y(\o)$ in any domain. All results of Section \ref{subsec:ppp} are based on the simple relation $r_y(\o+\o_3)= r_y(\o)+ f_y(\o)$. For the second point there are general theorems, such as \ref{expression_elliptic_zeta_generalization}, which give an expression for an elliptic function with prescribed poles. Concerning the last point, we can use in a constructive way Property \ref{algebraic_theorem}, saying that if $f(\o)$ and $g(\o)$ are non-constant elliptic functions with the same periods, there exists a non-zero polynomial $P$ such that $P(f(\o),g(\o))=0$, for all $\o\in{\bf C}$. (The function $y(\o)$, being elliptic with periods $\o_1,\o_2$, is of course also elliptic for the periods $\o_1,k\o_2$.)

This part will be illustrated in Part \ref{part:examples}, by the detailed example of Kreweras' model.

\subsection{Holonomic case}

The main point here consists in introducing a function $\phi(\o)$ such that
\begin{enumerate}[label=($\arabic{*}_{\phi}$),ref={\rm($\arabic{*}_{\phi}$)}]
\item\label{item:phi_mero}$\phi(\o)$ is meromorphic on ${\bf C}$;
\item\label{item:phi_o_1}$\phi(\o)$ is $\o_1$-periodic;
\item\label{item:phi_o_2}$\phi(\o+k\o_2)=\phi(\o)+1$, for all $\o\in{\bf C}$.
\end{enumerate}
An example of such functions is given by (this is simply verified, and this can be found in \cite[Equation (4.3.7)]{FIM})
\begin{equation}
\label{eq:expression_good_phi}
     \phi(\o) = \frac{\o_1}{2i\pi}\zeta(\o;\o_1,k\o_2)-\frac{\o}{i\pi}\zeta(\o_1/2;\o_1,k\o_2),\qquad \forall \o\in{\bf C}.
\end{equation}
The reason of introducing $\phi(\o)$ is the following: Equation \eqref{eq:trans} is equivalent to saying that the function $r_y(\o+k\o_2)-\phi(\o)\mathcal{O}(\o)$ is $k\o_2$-periodic, and hence elliptic with periods $\o_1,k\o_2$. We are then faced with the same problem as in the algebraic case (Section \ref{subsec:alca}), namely, to find the expression in terms of $y(\o)$ of a function (namely, $r_y(\o+k\o_2)-\phi(\o)\mathcal{O}(\o)$) elliptic with periods $\o_1,k\o_2$, with known poles (the poles are known since we know the poles of $r_y(\o)$ with Section \ref{subsec:ppp}, we also know those of the orbit-sum with the formula \eqref{eq:orbit_sum}, and those of $\phi(\o)$ via \eqref{eq:expression_good_phi}). We thus refer to the algebraic case for the details.

There is, however, an additional difficulty around the function $\phi(\o)$: how to obtain an expression of it in terms of $y(\o)$? In fact, the function $\phi'(\o)$ is easily found in terms of $y(\o)$, since it is an elliptic function. The function $\phi'(\o)$ is thus an algebraic function of $y(\o)$, and the function $\phi(\o)$ turns out to be the primitive of an algebraic function of $y(\o)$, which needs not to be (and in fact, which is not) algebraic.


To conclude we comment on functions $\phi(\o)$ satisfying to \ref{item:phi_mero}, \ref{item:phi_o_1} and \ref{item:phi_o_2}. If $\phi(\o)$ and $\psi(\o)$ both satisfy \ref{item:phi_mero}, \ref{item:phi_o_1} and \ref{item:phi_o_2}, then the function $\phi(\o)-\psi(\o)$ is elliptic with periods $\o_1,k\o_2$. This means that it essentially suffices to find one good function $\phi(\o)$. Let us show that \eqref{eq:expression_good_phi} is suitable. That $\phi(\o+\o_1)=\phi(\o)$ simply follows from Property \ref{half_period_translated_zeta} in Appendix \ref{appendix-elliptic}. Finally, the identity $\phi(\o+\o_2)=\phi(\o)+1$ comes from Property \ref{Legendre_identity} (known as Legendre's identity).

In a fundamental parallelogram the function $\phi(\o)$ has a unique pole, at $0$, of order $1$ and with residue $1$. In some sense, the function $\phi(\o)$ in \eqref{eq:expression_good_phi} is thus the minimal function (i.e., with the smallest order) satisfying \ref{item:phi_mero}, \ref{item:phi_o_1} and \ref{item:phi_o_2}.

\part{Examples}
\label{part:examples}

In Part \ref{part:examples} we illustrate our results, with three examples (finite group and algebraic, finite group and holonomic but non-algebraic, infinite group). The first example, which we treat in full detail, is that of Kreweras (see Figure \ref{ExExEx}). In Section \ref{sec:compute} we obtain an expression for the GF in terms of Weierstrass $\zeta$-functions. In Sections \ref{sec:Q(0,0)} and \ref{sec:proof_complete_Kreweras} we show how to derive the well-known expressions for $Q(0,0;z)$ and $Q(x,0;z)=Q(0,x;z)$ by the theory of transformation of elliptic functions. The second example is the simple walk (Figure \ref{ExExEx}), in Section \ref{sec:simple_walk}. We close Part \ref{part:examples} by Section \ref{sec:infinite_group_example}, which is devoted to an example of walk having an infinite group (Figure \ref{Exinfinite}).

\section{Expression of Kreweras' GF in terms of Weierstrass elliptic functions}
\label{sec:compute}

\subsection{A guide for the next three sections}
In Sections \ref{sec:compute}, \ref{sec:Q(0,0)}
 and \ref{sec:proof_complete_Kreweras}, we illustrate our computational results on the example of Kreweras' model.
  The step set $\mathcal S$ of this model can be seen on Figure \ref{ExExEx}.
  We decompose this illustration into three steps.
  The first one (Section \ref{sec:compute}) consists in deriving from
Theorems \ref{dva} and \ref{serie}
   an expression for the function
$
     r_y(\o) =zy(\o)Q(0,y(\o))
$
   in terms of Weierstrass $\zeta$-functions, see Proposition \ref{prop:Kreweras_consequence_serie}.
 This expression is new, up to our knowledge.
  The second step  (Section \ref{sec:Q(0,0)})
  is to show how obtaining from this expression an algebraic function of the series of the excursions $Q(0,0)$.
  Finally (Section \ref{sec:proof_complete_Kreweras}),
   we explain how to find an algebraic expression for $Q(y,0)=Q(0,y)$.

\begin{prop}[\cite{BK2,BM,BMM,FIM,FL1,FL2,Kreweras}]
\label{prop:BMM_Kreweras}
Consider Kreweras' walks. Let $W$ be the only power series in $z$ satisfying $W=z(2+W^3)$. Then
\begin{equation}
\label{eq:Kreweras_Q(0,0)}
     Q(0,0)=\frac{1}{2z}(W-W^4/4).
\end{equation}
Further, one has
\begin{equation}
\label{eq:Kreweras_Q(0,y)}
     Q(y,0)=Q(0,y) = \frac{1}{zy}\left(\frac{1}{2z}-\frac{1}{y}-\left(\frac{1}{W}-\frac{1}{y}\right)\sqrt{1-yW^2}\right).
\end{equation}
\end{prop}
(Of course, \eqref{eq:Kreweras_Q(0,0)} is an immediate consequence of \eqref{eq:Kreweras_Q(0,y)}, but in order to illustrate our approach, we shall first prove \eqref{eq:Kreweras_Q(0,0)}, and then \eqref{eq:Kreweras_Q(0,y)}.) There already exist many proofs of Proposition \ref{prop:BMM_Kreweras}. The first one was proposed by Kreweras himself \cite{Kreweras}. See also Flatto and Hahn \cite{FL1} and Flatto \cite{FL2}, where methods using elliptic functions were used (see also \cite[Section 4.6]{FIM}, where similar ideas are used). More recently, proofs were done in \cite{BM,BMM}, using an extension of the well-known kernel method. Finally, in \cite{BK2}, there is a proof based on a guessing-proving approach (which is also applied to Gessel's walks).

\subsection{Uniformization in Kreweras' case}
\label{sec:Preliminaries}

The function $f_y(\o)$ defined in \eqref{fy}---which appears in the meromorphic continuation procedure \eqref{cont1}---will be particularly important. To compute it, we need to simplify the coordinates of the uniformization $x(\o),y(\o)$ in \eqref{expression_uniformization}. In the case of Kreweras, $x_4=y_4=\infty$, and one has
\begin{equation}
\label{eq:uniformization}
     \left\{\begin{array}{lll}
     x(\o)&=&\displaystyle\frac{\wp(\o)-d''(0)/6}{d'''(0)},\phantom{\frac{1}{\frac{1}{1}}}\\
     y(\o)&=&\displaystyle-\frac{b(x(\o))+\wp'(\o)/(2d_3)}{2a(x(\o))},
     \end{array}\right.
\end{equation}
with $d''(0)/2=1$, $d'''(0)/6=-4z^2$, $b(x)=-x+z$
and $a(x)=zx^2$, see \eqref{dx}.
We recall 
that in \eqref{eq:uniformization}, $\wp(\o)$ denotes the Weierstrass elliptic function with periods $\o_1,\o_2$,
see \eqref{expression_omega_1_2} and \eqref{eq:first_time_expansion_wp}.
 We shall sometimes write, instead, $\wp(\o;\o_1,\o_2)$. Finally, we recall that with $\o_3$ defined
 as in \eqref{expression_omega_3}, in the case of Kreweras one has $\o_3/\o_2=2/3$ (this in particular implies
 that the group \eqref{group} has order $6$). To apply our main results (Theorems \ref{dva} and \ref{serie}),
  we need to know where  the poles of $f_y(\o)$ are located.
\begin{lem}
\label{lem:poles_f_y_rectangle}
In the fundamental rectangle $\o_1[0,1)+\o_2[0,1)$, the
 function $f_y(\o)$ has poles at $0$, $\o_2/3$ and $2\o_2/3$. These poles are simple, with residues equal to $-1/z$, $1/(2z)$ and $1/(2z)$, respectively.
\end{lem}
\begin{proof}
Observe that with \eqref{fy}, \eqref{eq:uniformization}, and with the equalities $\wp(-\o)=\wp(\o)$, $\wp'(-\o)=-\wp'(\o)$, we can write
\begin{equation}
\label{eq:expression_f_y_wp}
     f_y(\o)=\frac{\wp'(\o)}{2z(\wp(\o)-1/3)}.
\end{equation}
The latter has a simple pole at $0$, with residue $-1/z$ (remember that in the neighborhood of $0$, $\wp(\o)=1/\o^2+ O(\o^2)$). It has also poles at points where $\wp(\o)-1/3=0$. For the model under consideration, one has $\wp(\o_2/3)=1/3$, see \cite[Page 773]{Ra}. Accordingly, $\wp(2\o_2/3)=1/3$ as well, and one concludes that $\o_2/3$ and $2\o_2/3$ are the two remaining poles of $f_y(\o)$ on the fundamental rectangle. Making an expansion in \eqref{eq:expression_f_y_wp}, we deduce that the residues at $\o_2/3$ and $2\o_2/3$ are equal to $1/(2z)$.
\end{proof}

\begin{rem}
\label{rem:strictly_less}
{\rm The function $f_y(\o)$ has a priori $6$ poles (see Section \ref{sec:focus_poles}). Lemma \ref{lem:poles_f_y_rectangle} shows that it may have strictly less than $6$ poles.}\hfill$\Box$
\end{rem}



The Weierstrass $\wp$-function with periods $\o_1,\o_2$ can alternatively be characterized by its periods $\o_1,\o_2$ (see \eqref{eq:first_time_expansion_wp}), or by its invariants $g_2,g_3$, by the formula
\begin{equation}
\label{eq:expression_invariants_o_1_o_2}
     \wp'(\o)^2=4\wp(\o)^3-g_2\wp(\o)-g_3.
\end{equation}
\begin{lem}
\label{lemma:invariants_o_1_o_2}
For Kreweras' model, we have
\begin{equation}
\label{eq:invariants_o_1_o_2}
     g_2 = 4/3-32z^3,\qquad
     g_3=-8/27+32z^3/3-64z^6.
\end{equation}
\end{lem}
\begin{proof}
The construction of the uniformization (see \cite{FIM} or \cite[Page 770]{Ra}, and \eqref{expression_uniformization} in this paper) is the following: with \eqref{eq:uniformization} one has $\wp(\o)=g(x(\o))$, where $g(x)=1/3-4z^2x$. Then
\begin{equation*}
     g_2 = -4[g(x_1)g(x_2)+g(x_1)g(x_3)+g(x_2)g(x_3)],\qquad g_3=4g(x_1)g(x_2)g(x_3).
\end{equation*}
One concludes by computing the branch points above (which are the roots of the polynomial \eqref{dx}) in terms of $z$.
\end{proof}

\subsection{Expression of the GFs}

Applying Theorem \ref{serie} to Kreweras' walks, we first obtain an expression of $r_y(\o)$
in terms of special functions. In what follows, we denote by $\zeta_{1,2}$
and $\wp_{1,2}$ the Weierstrass $\zeta$- and $\wp$-functions with periods $\omega_1,2\omega_2$.

\begin{prop}
\label{prop:Kreweras_consequence_serie}
One has
\begin{equation}
\label{eq:first_r_y_zeta}
     r_y(\o)=c+\frac{1}{2z}\zeta_{1,2}(\o+2\o_2/3)-\frac{1}{z}\zeta_{1,2}(\o+\o_2/3)+\frac{1}{z}\zeta_{1,2}(\o)-\frac{1}{2z}\zeta_{1,2}(\o-\o_2/3),\qquad \forall \o\in{\bf C}.
\end{equation}
\end{prop}
Theorem \ref{serie} also gives the expression of the constant $c$ in \eqref{eq:first_r_y_zeta}, but we shall compute it in Section \ref{sec:proof_complete_Kreweras} only.

\begin{proof}
There are several ways to deduce Proposition
\ref{prop:Kreweras_consequence_serie}. A first one is to use Theorem
\ref{serie}, which gives an expression for $r_y(\o)$: one should group
some terms of \eqref{AA} together in order to deduce
$\zeta$-functions and to exploit \eqref{exploit}.

The second one is to apply Theorem \ref{dva}, which only concerns
the poles of the GF (and not its expression), combined
  with the ellipticity of $r_y(\o)$ proper to Kreweras' model.
Namely, it is elementary to check the condition \eqref{fff} of
Proposition \ref{prop:fff} using our analysis of poles of $f_y(\o)$
above. Then by Proposition \ref{prop:fff} (see also \cite[Section
6]{KRIHES}) $r_y(\o)$ is algebraic with $2k=4$ branches and hence
elliptic with periods $\o_1, 2\o_2$. This remark allows us to study
the poles of $r_y(\o)$ only in the parallelogram
$\o_1[-1/2,1/2)+\o_2[-3/2,1/2)$.
According to Theorem
\ref{dva}, the poles of $r_y(\o)$ must satisfy $\mathcal{N}_d^-\neq \emptyset$,
 where (we recall that $\Re \o_{x_1}=\o_2/2$)
\begin{equation*}
     \mathcal{N}_d^-=\{n\geq 0:d+n\o_3\hbox{ is a pole of } f_y(\o) \hbox{ with } \o_2/2- k \o_{2} < \Re d+n\o_3< \o_2/2 \}.
\end{equation*}
For Kreweras' model, one has $\o_2/\o_3=3/2$, so that $k=2$,
and the only poles $d$ of $f_y(\o)$ with $-3 \o_{2}/2 < \Re d+n\o_3< \o_2/2$
and $|\Im d|\leq |\o_1/2|$ are the points of the set
\begin{equation*}
     P=\{-4\o_2/3,-\o_2,-2\o_2/3,-\o_2/3,0,\o_2/3\}.
\end{equation*}
We thus have $\mathcal{N}_d^-=\{n\geq 0:d+n\o_3\in P \}$, and it is
obvious that the points of the parallelogram
$\o_1[-1/2,1/2)+\o_2[-3/2,1/2)$ such that $\mathcal{N}_d^-\neq
\emptyset$ are among the points of $P$.
Let us study closer each of them.
  We start with $d\in\{0,\o_2/3\}$.
We have $\mathcal{N}_d^-=\{0\}$, and according to Theorem \ref{dva}
and Lemma \ref{lem:poles_f_y_rectangle}, we find the following
principal parts of $r_y(\o)$:
\begin{align*}
     R_{\o_2/3,y}(\o) &= -F_{\o_2/3,y}(\o)=\frac{-1/(2z)}{\o-\o_2/3},\\
     R_{0,y}(\o) &= -F_{0,y}(\o)=\frac{1/z}{\o}.
\end{align*}
Consider now the cases $d\in\{-2\o_2/3,-\o_2/3\}$, for which one has $\mathcal{N}_d^-=\{0,1\}$. Then
\begin{align*}
     R_{-\o_2/3,y}(\o) &= -F_{-\o_2/3,y}(\o)-F_{\o_2/3,y}(\o+2\o_2/3)=\frac{-1/z}{\o+\o_2/3},\\
     R_{-2\o_2/3,y}(\o) &= -F_{-2\o_2/3,y}(\o)-F_{0,y}(\o+2\o_2/3)=\frac{1/(2z)}{\o+2\o_2/3}.
\end{align*}
For the two remaining cases $d\in\{-4\o_2/3,-\o_2\}$, one has $\mathcal{N}_d^-=\{0,1,2\}$, and similar computations as above show that
\begin{equation*}
     R_{-\o_2,y}(\o)=R_{-4\o_2/3,y}(\o)=0.
\end{equation*}
In other words, the points $-4\o_2/3$ and $-\o_2$ are removable singularities of the GF $r_y(\o)$.

To conclude, we notice that the elliptic function $r_y(\o)$  with
periods $\o_1,2\o_2$
  has in the fundamental parallelogram $\o_1[-1/2,1/2)+\o_2[-3/2,1/2)$ four poles,
 with principal parts given above.
   It is immediate from the theory of elliptic functions
  (in particular Property \ref{expression_elliptic_zeta})
  that Equation \eqref{eq:first_r_y_zeta} holds. The proof is completed.
\end{proof}


\begin{rem}
{\rm Starting from \eqref{eq:first_r_y_zeta}, we can  recover
\eqref{cont1}. Introduce
\begin{equation*}
     k(\o) = \frac{1}{2z}\zeta_{1,2}(\o+\o_2/3)-\frac{1}{z}\zeta_{1,2}(\o)+\frac{1}{2z}\zeta_{1,2}(\o-\o_2/3).
\end{equation*}
An easy computation starting from \eqref{eq:first_r_y_zeta} yields
\begin{equation*}
     r_y(\o+\o_3)-r_y(\o) = k(\o)+k(\o+\o_2).
\end{equation*}
Since $k(\o)$ is $\o_1,2\o_2$ elliptic (see Property \ref{expression_elliptic_zeta_direct}), we deduce that $k(\o)+k(\o+\o_2)$ is $\o_1,\o_2$ elliptic, hence $r_y(\o+\o_3)-r_y(\o)$ is $\o_1,\o_2$ elliptic. We first show that $r_y(\o+\o_3)-r_y(\o)$ has the same poles as $f_y(\o)$ in $\o_1[0,1)+\o_2[0,1)$. Since in $\o_1[0,1)+\o_2[0,1)$, $\zeta_{1,2}$ has only one pole (which is of order $1$, at $0$ and with residue $1$, see Property \ref{expression_zeta}), we obtain that in $\o_1[0,1)+\o_2[0,1)$, $r_y(\o+\o_3)-r_y(\o)$ has only simple poles, at $0$, $\o_2/3$ and $2\o_2/3$, with respective residues $-1/z$, $1/(2z)$ and $1/(2z)$. The same holds for $f_y(\o)$ (Lemma \ref{lem:poles_f_y_rectangle}). Thus, there exists a constant $c$ such that
\begin{equation}
\label{eq:before_finding_c}
     r_y(\o+\o_3)-r_y(\o)=c+f_y(\o),\qquad \forall \o\in{\bf C}.
\end{equation}
Evaluating \eqref{eq:before_finding_c} at $\o=\o_2/2$ and using that $f_y(\o_2/2)=0$ (see Equation \eqref{eq:expression_f_y_wp}), we obtain that $c=r_y(\o_2/2+\o_3)-r_y(\o_2/2)=k(\o/2)+k(3\o_2/2)$. Using the fact that $\zeta_{1,2}$ is odd and $2\o_2$-periodic, we obtain that $c=0$, which yields \eqref{cont1}.}\hfill$\Box$
\end{rem}

\section{Finding an algebraic expression for the Kreweras' GF of the excursions}
\label{sec:Q(0,0)}

In this section we prove the first part of Proposition \ref{prop:BMM_Kreweras} (Equation \eqref{eq:Kreweras_Q(0,0)}). Surprisingly, we shall do it without finding the value of the additive constant $c$ in \eqref{eq:first_r_y_zeta}.

\subsection{Beginning of the proof}

One has from \eqref{sqs2} that
\begin{equation*}
     r_x(\o)+r_y(\o)-K(0,0)Q(0,0)-x(\o)y(\o)=0,\qquad \forall \o\in{\bf C}.
\end{equation*}
Setting $\o_0^y=2\o_2/3$, one has $y(\o_0^y)=0$; further, it turns out that $x(\o_0^y)=0$ (see \eqref{expression_uniformization} and \eqref{eq:uniformization}). Since $K(0,0)=0$ for Kreweras' model, one deduces that
\begin{equation*}
     \frac{r_x(\o)}{x(\o)}=y(\o)-\frac{r_y(\o)-r_y(\o_0^y)}{x(\o)} = y(\o)-\frac{r_y(\o)-r_y(\o_0^y)}{\o-\o_0^y}\frac{\o-\o_0^y}{x(\o)-x(\o_0^y)}.
\end{equation*}
Then, if $\o\to \o_0^y$, one finds
\begin{equation*}
     zQ(0,0) = -\frac{r_y'(\o_0^y)}{x'(\o_0^y)}.
\end{equation*}
By \eqref{eq:uniformization}, $x'(\o_0^y) = \wp'(2\o_2/3)/(d'''(0)/6)$. Since $\wp(2\o_2/3)=1/3$, using \eqref{eq:expression_invariants_o_1_o_2} and \eqref{eq:invariants_o_1_o_2} leads to $\wp'(2\o_2/3)=\pm 8z^3$. But $\wp'(\o)$ is positive on $(\o_2/2,\o_2)$, so that finally $\wp'(2\o_2/3)=8z^3$. Since $d'''(0)/6=-4z^2$, we reach the conclusion that $x'(\o_0^y)=-2z$. Finally, one finds
\begin{equation}
\label{eq:Kreweras_Q(0,0)_btc}
     Q(0,0) = \frac{1}{4z^3}[\wp_{1,2}(\o_2/3)-\wp_{1,2}(4\o_2/3)+2(\wp_{1,2}(\o_2)-\wp_{1,2}(2\o_2/3))].
\end{equation}
The difficulty now consists in transforming the above expression in an algebraic function of $z$.
 So far, we have already computed the values of $\wp(\o)$ for some particular $\o$
 (e.g., we saw that $\wp(\o_2/3)=\wp(2\o_2/3)=1/3$,
 and that $\wp(0)=\wp(\o_2)=\infty$).
 On the other hand, we never computed $\wp_{1,2}(\o)$ for some given $\o$.
 To do so, the simplest thing is to express $\wp_{1,2}$ (with periods $\o_1,2\o_2$)
 in terms of $\wp$ (with periods $\o_1,\o_2$). This is the aim of the next section.

\subsection{Intermezzo: transformation theory of elliptic functions}

Let $\wp_{1,2}$ be the Weierstrass elliptic function with periods $\o_1,2\o_2$. Denote its invariants by $g_2^{1,2}$ and $g_3^{1,2}$. We also define $e_{2}^{1,2}=\wp_{1,2}(\o_2)$, $e_{1}^{1,2}=\wp_{1,2}(\o_1/2)$ and $e_{1+2}^{1,2}=\wp_{1,2}(\o_1/2+\o_2)$. Note that the latter are the three solutions of $4X^3-g_2^{1,2}X-g_3^{1,2}=0$. To compute the latter quantities, it is convenient to introduce $r$, $\widetilde r$ and $\widehat r$ as the roots of
\begin{equation}
\label{eq:simili_invariants}
     4X^3-g_2X+g_3=0,
\end{equation}
with $g_2,g_3$ defined as in Lemma \ref{lemma:invariants_o_1_o_2}.\footnote{It is worthwhile comparing these roots with $\wp(\o_2/2)$, $\wp((\o_2+\o_1)/2)$ and $\wp(\o_1/2)$, which are solutions to $4X^3-g_2X-g_3=0$.}  We enumerate them in the following way: only one solution of \eqref{eq:simili_invariants} is a power series in $z$ (see \cite[Proposition 6.1.8]{St}), we call it $r$. It admits the expansion
\begin{equation*}
     r= 2/3-8z^3-48z^6-640z^9+O(z^{12}).
\end{equation*}

\begin{lem}
\label{lemma:expression_invariants_o_1_2o_2}
One has
\begin{align*}
     e_{2}^{1,2}&= r/2,\\
     g_2^{1,2} &= 15r^2/4-g_2/4,\\
     g_3^{1,2} & = 11g_3/32-7r g_2/32.
\end{align*}
\end{lem}
\begin{proof}
Using the properties \ref{addition_theorem} and \ref{principle_transformation}, one can write, for any $\o\in{\bf C}$,
\begin{equation}
\label{eq:wp_wp12_ue1}
     \wp(\o)=\wp_{1,2}(\o)+\wp_{1,2}(\o+\o_2)-\wp_{1,2}(\o_2)=-2\wp_{1,2}(\o_2)+\left(\frac{\wp_{1,2}'(\o)}{\wp_{1,2}(\o)-\wp_{1,2}(\o_2)}\right)^{2}.
\end{equation}
We then make an expansion of the LHS and the RHS of the above equation in the neighborhood of $0$; we obtain
\begin{multline*}
     \frac{1}{\o^2}+\frac{g_2}{20}\o^2+\frac{g_3}{28}\o^4+O(\o^6) \\= \frac{1}{\o^2}+\left(3(e_2^{1,2})^2-\frac{g_2^{1,2}}{5}\right)\o^2+\left(4(e_2^{1,2})^3-\frac{g_2^{1,2}e_2^{1,2}}{2}-\frac{3g_3^{1,2}}{14}\right)\o^4+O(\o^6).
\end{multline*}
Identifying the expansions above, we obtain two equations for the three unknowns $e_2^{1,2}$, $g_2^{1,2}$ and $g_3^{1,2}$ (remember that $g_2$ and $g_3$ are known from Lemma \ref{lemma:invariants_o_1_o_2}). We add a third equation by noticing that $e_2^{1,2}$ is a root of $4X^3-g_2^{1,2}X-g_3^{1,2}=0$. We then have a (non-linear) system of three equations with three unknowns. Some computations finally lead to the expressions of $e_2^{1,2}$, $g_2^{1,2}$ and $g_3^{1,2}$ given in Lemma \ref{lemma:expression_invariants_o_1_2o_2}.
\end{proof}
\begin{rem}
{\rm Contrary to $g_2$ and $g_3$, $g_2^{1,2}$ and $g_3^{1,2}$ are not rational functions of $z$. However, they are algebraic (and so is $e_2^{1,2}$).}\hfill$\Box$
\end{rem}

Thanks to Lemma \ref{lemma:expression_invariants_o_1_2o_2}, one can consider that the function $\wp_{1,2}$ is completely known, as its periods and its invariants are expressed in the variable $z$ in an explicit way. The next result proposes an expression of $\wp_{1,2}$ in terms of $\wp$.

\begin{lem}
\label{lem:expression_wp_12_wp}
One has
     \begin{equation*}
          2\wp_{1,2}(\o)=\wp(\o)+e_{2}^{1,2}\pm\sqrt{(\wp(\o)-e_{2}^{1,2})^2+g_2^{1,2}-12(e_{2}^{1,2})^2},\qquad \forall \o\in{\bf C},
     \end{equation*}
where $e_{2}^{1,2}$ and $g_2^{1,2}$ are as in Lemma \ref{lemma:expression_invariants_o_1_2o_2}.\footnote{The sign $\pm$ above depends on which half-parallelogram $\o$ is located.}
\end{lem}
\begin{proof}
We start from Equation \eqref{eq:wp_wp12_ue1}, which can be rewritten as
\begin{equation}
\label{eq:wp_wp12_ue2}
     \wp(\o) = -2e_2^{1,2}+\frac{(\wp_{1,2}(\o)-e_1^{1,2})(\wp_{1,2}(\o)-e_{1+2}^{1,2})}{\wp_{1,2}(\o)-e_2^{1,2}},\qquad \forall \o\in{\bf C}.
\end{equation}
Further, we have the equalities
\begin{equation}
\label{eq:identities_invariants_122}
     e_{1}^{1,2}+e_{1+2}^{1,2}+e_{2}^{1,2}=0,\qquad
     e_{1}^{1,2}e_{1+2}^{1,2}+e_{1}^{1,2}e_{2}^{1,2}+e_{1+2}^{1,2}e_{2}^{1,2}=-\frac{g_{2}^{1,2}}{4}.
\end{equation}
To conclude, we solve Equation \eqref{eq:wp_wp12_ue2} as an equation (of the second order) in $\wp_{1,2}(\o)$. Making use of the identities \eqref{eq:identities_invariants_122}, we obtain Lemma \ref{lem:expression_wp_12_wp}.
\end{proof}

\subsection{End of the proof}
Using Lemma \ref{lem:expression_wp_12_wp} and the fact that $\wp(\o_2/3)=\wp(2\o_2/3)=1/3$, we obtain that
\begin{align*}
     \wp_{1,2}(\o_2/3)&=\frac{1}{2}(1/3+r/2+\sqrt{-2/9-r/3+r^2+8z^3}),\\
     \wp_{1,2}(2\o_2/3)=\wp_{1,2}(4\o_2/3)&=\frac{1}{2}(1/3+r/2-\sqrt{-2/9-r/3+r^2+8z^3}).
\end{align*}
By \eqref{eq:Kreweras_Q(0,0)_btc}, it follows that
\begin{equation}
\label{eq:Kreweras_Q(0,0)_btc1}
     Q(0,0) = \frac{1}{4z^3}(-1/3+r/2+2\sqrt{-2/9-r/3+r^2+8z^3}).
\end{equation}
The GF $Q(0,0)$ is now expressed as an algebraic function of $z$ (in \eqref{eq:Kreweras_Q(0,0)_btc} it was not clear). To conclude, we notice that we can express $r$ with $W$ as follows (this can be achieved by computing the minimal polynomial of $W+W^4/4$ starting from the minimal polynomial of $W$):
\begin{equation*}
     r=\frac{2}{3}-4z^2(W+W^4/4).
\end{equation*}
It is now obvious that \eqref{eq:Kreweras_Q(0,0)_btc1} implies \eqref{eq:Kreweras_Q(0,0)}.

\section{Algebraic expression for the Kreweras' GF of walks ending on one axis}
\label{sec:proof_complete_Kreweras}

In what follows, we would like to express $r_y(\o)$ in terms of $y(\o)$ (to eventually find \eqref{eq:Kreweras_Q(0,y)}). To that purpose we notice that the function $x(\o)$ is simpler than $y(\o)$, see \eqref{eq:uniformization}. Further, due to \cite[Equation (3.3)]{KRIHES} and the symmetry of the model, we have $y(\o) = x(\o-\o_3/2)$. In other words, it is equivalent to express $r_y(\o)$ in terms of $y(\o)$, or to express $r_y(\o+\o_3/2)$ in terms of $x(\o)$. Since $\o_3/2=\o_2/3$, \eqref{eq:first_r_y_zeta} implies that
\begin{equation}
\label{eq:expression_translated}
     r_y(\o+\o_3/2)=c+\frac{1}{2z}\zeta_{1,2}(\o-\o_2)-\frac{1}{2z}\zeta_{1,2}(\o-2\o_2)+\frac{1}{z}\zeta_{1,2}(\o-5\o_2/3)-\frac{1}{z}\zeta_{1,2}(\o-4\o_2/3).
\end{equation}
We shall then structure the proof as follows: in Section \ref{subsec:finding_c} we find the constant $c$ in \eqref{eq:first_r_y_zeta} (or equivalently in \eqref{eq:expression_translated}). In Section \ref{subsec:replacing}, for computational reasons, we replace $\zeta_{1,2}$-functions by $\wp_{1,2}$-functions. In Section \ref{subsec:technical} we are interested in (needed) technical lemmas. Finally, in Section \ref{subsec:conclusion} we give the proof of Equation \eqref{eq:Kreweras_Q(0,y)} of Proposition \ref{prop:BMM_Kreweras}.\footnote{In Section \ref{sec:Q(0,0)}, we have shown how our theoretical results work for finding an expression for $Q(0,0)$ in Kreweras' case. In Section \ref{sec:proof_complete_Kreweras} we now show how our methods also provide an expression for $Q(x,0)$ or $Q(0,y)$. For conciseness, we decided to state all necessary intermediate results to obtain such an expression, but not to prove all of them.}

\subsection{Finding the constant}
\label{subsec:finding_c}
To find the constant $c$ in \eqref{eq:first_r_y_zeta}, we could use Theorem \ref{serie}, which gives the exact expression for $r_y(\o)-r_y(\o_{y_2})$. It is also possible to use the following property of Kreweras' model: for $\o$ such that $y(\o)=0$, we must have $r_y(\o)=0$. Indeed, remember that one has $r_y(\o)=zy(\o)Q(0,y(\o))$. Since in this case $y(2\o_2/3)=y(\o_2)=0$, we immediately find
\begin{equation}
\label{eq:value_c_1}
     c = \frac{1}{2z}(\zeta_{1,2}(\o_2/3)-\zeta_{1,2}(4\o_2/3))+\frac{1}{z}(\zeta_{1,2}(\o_2)-\zeta_{1,2}(2\o_2/3)).
\end{equation}

\subsection{Replacing the Zeta-functions by Pe-functions}
\label{subsec:replacing}
With \eqref{eq:expression_translated} and \eqref{eq:value_c_1}, the function $r_y(\o)$ is completely known. However, it would be more convenient to have $\wp_{1,2}$-functions instead of $\zeta_{1,2}$-functions in \eqref{eq:expression_translated}. To do so, we shall use (four times) the addition theorem \ref{addition_theorem} in \eqref{eq:first_r_y_zeta}. Following this way, we obtain
\begin{align}
     \label{eq:equality_r_y}
     r_y(\o&+\o_3/2-\o_2/2)=r_y(\o+\o_3/2+3\o_2/2) \\\nonumber
     &=c+\frac{1}{2z}(\zeta_{1,2}(\o+\o_2/2)-\zeta_{1,2}(\o-\o_2/2))+\frac{1}{z}(\zeta_{1,2}(\o-\o_2/6)-\zeta_{1,2}(\o+\o_2/6))\\&=c+\frac{\zeta_{1,2}(\o_2/2)}{z}-\frac{\wp'_{1,2}(\o_2/2)/(2z)}{\wp_{1,2}(\o)-\wp_{1,2}(\o_2/2)}-\frac{2\zeta_{1,2}(\o_2/6)}{z}+\frac{\wp'_{1,2}(\o_2/6)/z}{\wp_{1,2}(\o)-\wp_{1,2}(\o_2/6)}.\nonumber
\end{align}

\subsection{Some technical computations}
\label{subsec:technical}

Equation \eqref{eq:equality_r_y} expresses $r_y(\o+\o_3/2)$ in terms of one single function, namely, $\wp_{1,2}(\o+\o_2/2)$. This is why we now need to compute $\wp_{1,2}(\o+\o_2/2)$ in terms of $x(\o)$. This will be done in Lemma \ref{lemma:expression_wp_1_2}. But before stating and proving this result, we first express $x(\o+\o_2/2)$ in terms of $x(\o)$.

\begin{lem}
\label{lemma:expression_x_translated}
One has
     \begin{equation*}
          x(\o+\o_2/2)=\sqrt{x_1}\frac{\sqrt{x_1}x(\o)+1}{x(\o)-x_1}, \qquad \forall \o\in{\bf C}.
     \end{equation*}
\end{lem}
\begin{proof}
Introduce, as in the proof of Lemma \ref{lemma:invariants_o_1_o_2}, the function $g(x)=1/3-4z^2x$. We have, for any $\o\in{\bf C}$,
     \begin{equation*}
          \wp(\o+\o_2/2)=e_2+\frac{(e_2-e_1)(e_2-e_{1+2})}{\wp(\o)-e_2} = g(x_1)+\frac{(g(x_1)-g(x_2))(g(x_1)-g(x_3))}{g(x(\o))-g(x_1)}.
     \end{equation*}
Using the facts that $x_1+x_2+x_3 = 1/(4z^2)$, $x_1x_2x_3=1/4$ and $d(x_1)=0$, and after some computations, we obtain that
     \begin{equation*}
          \wp(\o+\o_2/2)=\frac{(1/3-4z^2x_1)x(\o)+(-x_1/3+2z-2z^2/x_1)}{3(x(\o)-x_1)}, \qquad \forall \o\in{\bf C}.
     \end{equation*}
Finally, using that $x(\o+\o_2/2)=g^{-1}(\wp(\o+\o_2/2))$, we obtain Lemma \ref{lemma:expression_x_translated}.
\end{proof}

In Section \ref{sec:Q(0,0)}, we introduced $r$ to be the only power series solution to $4X^3-g_2X+g_3=0$ (Equation \eqref{eq:simili_invariants}).
The two other solutions of \eqref{eq:simili_invariants} are Puiseux series, and we set
\begin{align*}
     \widetilde r &= -1/3+4z^3-8z^{9/2}+24z^6-84z^{15/2}+320z^9+O(z^{21/2}),\\
     \widehat r &= -1/3+4z^3+8z^{9/2}+24z^6+84z^{15/2}+320z^9+O(z^{21/2}).
\end{align*}

\begin{lem}
\label{lemma:expression_wp_1_2}
One has
     \begin{equation}
     \label{eq:expression_wp_1_2_translated}
          \wp_{1,2}(\o+\o_2/2)=\frac{B_0+B_1x(\o)+\sqrt{A_0}\sqrt{1-x(\o)W^2}}{12(x(\o)-x_1)}, \qquad \forall \o\in{\bf C},
     \end{equation}
where
     \begin{equation}
     \label{eq:B_0_B_1_A_0}
          B_0 = -(2x_1+24z^2\sqrt{x_1}+3rx_1), \qquad B_1=3(r-\widetilde r), \qquad A_0 = \frac{(B_0+B_1x_1)^2}{1-x_1W^2}.
     \end{equation}
\end{lem}

\begin{proof}
The first step of the proof of \eqref{eq:expression_wp_1_2_translated} consists in evaluating at $\o+\o_2/2$ the expression in Lemma \ref{lem:expression_wp_12_wp}. Together with Lemma \ref{lemma:expression_x_translated}, this gives an expression of $\wp_{1,2}(\o+\o_2/2)$ in terms of $x(\o)$. Some simplifications (that we do not write down here) eventually lead to \eqref{eq:expression_wp_1_2_translated}.
\end{proof}

\begin{cor}
\label{for:two_particular_values}
One has
\begin{equation}
     \wp_{1,2}(\o_2/2) = \frac{B_1}{12},\qquad \wp_{1,2}(\o_2/6) = \frac{\sqrt{A_0}-B_0}{12x_1}.
\end{equation}
\end{cor}
\begin{proof}
We evaluate at $\omega=0$ the expression \eqref{eq:expression_wp_1_2_translated} of Lemma \ref{lemma:expression_wp_1_2}. Since $x(\o)$ has a pole of order $2$ at $0$, the first part of Corollary \ref{for:two_particular_values} follows. We then evaluate \eqref{eq:expression_wp_1_2_translated} at $\o=\o_2/3$. Since $x(\o_2/3)=0$, we obtain the second part of Corollary \ref{for:two_particular_values}.
\end{proof}

Because of Equation \eqref{eq:equality_r_y}, we also need to have an expression of $1/(\wp_{1,2}(\o+\o_2/2)-\wp_{1,2}(\o_2/2))$ and $1/(\wp_{1,2}(\o+\o_2/2)-\wp_{1,2}(\o_2/6))$ in terms of $x(\o)$.
\begin{lem}
One has
\begin{equation}
\label{eq:ue1}
     \frac{1}{\wp_{1,2}(\o+\o_2/2)-\wp_{1,2}(\o_2/2)} = \frac{12}{A_0W^2}(B_0+B_1x_1-\sqrt{A_0}\sqrt{1-x(\o)W^2}),\qquad \forall \o\in{\bf C},
\end{equation}
as well as
\begin{multline}
\label{eq:ue2}
     \frac{1}{\wp_{1,2}(\o+\o_2/2)-\wp_{1,2}(\o_2/6)} \\= \frac{12x_1}{C_1x(\o)}(\sqrt{A_0}x_1+(B_0-\sqrt{A_0}+B_1x_1)x(\o)- x_1\sqrt{A_0}\sqrt{1-x(\o)W^2}),\qquad \forall \o\in{\bf C},
\end{multline}
with
\begin{equation}
\label{eq:definition_C_1}
     C_1 = 2x_1B_1B_0-2B_1\sqrt{A_0}x_1-2\sqrt{A_0}B_0+x_1^2B_1^2+A_0+B_0^2.
\end{equation}
\end{lem}

\begin{proof}
The proof of both \eqref{eq:ue1} and \eqref{eq:ue2} is immediate, and just uses Lemma \ref{lemma:expression_wp_1_2} and Corollary \ref{for:two_particular_values}.
\end{proof}

\subsection{Conclusion}
\label{subsec:conclusion}

Using the expression \eqref{eq:equality_r_y} of $r_y(\o+\o_3/2-\o_2/2)$, we obtain that with $c$ defined as in \eqref{eq:value_c_1}, one has
     \begin{multline*}
          r_y(\o+\o_3/2)\\=c+\frac{\zeta_{1,2}(\o_2/2)}{z}-\frac{\wp'_{1,2}(\o_2/2)/(2z)}{\wp_{1,2}(\o+\o_2/2)-\wp_{1,2}(\o_2/2)}-\frac{2\zeta_{1,2}(\o_2/6)}{z}+\frac{\wp'_{1,2}(\o_2/6)/z}{\wp_{1,2}(\o+\o_2/2)-\wp_{1,2}(\o_2/6)}.
     \end{multline*}
Using \eqref{eq:ue1} and \eqref{eq:ue2}, we obtain that
     \begin{equation}
     \label{eq:before_to_conclude}
          r_y(\o+\o_3/2)=\alpha +\frac{\beta}{x(\o)}+\left(\gamma+\frac{\delta}{x(\o)}\right)\sqrt{1-x(\o)W^2},
     \end{equation}
where
     \begin{align}
     \label{eq:def_alpha}
          \alpha&=c+\frac{\zeta_{1,2}(\o_2/2)}{z}-\frac{2\zeta_{1,2}(\o_2/6)}{z}-\frac{6\wp'_{1,2}(\o_2/2)(B_0+B_1x_1)}{zA_0W^2}\\&\qquad\qquad\qquad\qquad\qquad\qquad\qquad+\frac{12x_1\wp'_{1,2}(\o_2/6)(B_0-\sqrt{A_0}+B_1x_1)}{C_1z},\nonumber\\
          \label{eq:def_beta}\beta&=\frac{12x_1^2\sqrt{A_0}\wp'_{1,2}(\o_2/6)}{C_1z},\\
          \label{eq:def_gamma}\gamma&=\frac{6\wp'_{1,2}(\o_2/2)}{z\sqrt{A_0}W^2},\\
          \label{eq:def_delta}\delta&= \frac{-12x_1^2\sqrt{A_0}\wp'_{1,2}(\o_2/6)}{C_1 z}.
 \end{align}

\begin{lem}
\label{lem:values_constants}
One has $\alpha=1/(2z)$, $\beta=-1$, $\gamma=-1/W$ and $\delta=1$.
\end{lem}
\begin{proof}
The proof of Lemma \ref{lem:values_constants} is long, and is based on arguments similar to those already used in other proofs of Sections \ref{sec:compute}--\ref{sec:proof_complete_Kreweras}. For brevity, we omit it.
\end{proof}

With Equation \eqref{eq:before_to_conclude}, Lemma \ref{lem:values_constants}, and since $y(\o)=x(\o-\o_3/2)$, the proof of Proposition \ref{prop:BMM_Kreweras} is completed.

\section{A holonomic non-algebraic example: the simple walk}
\label{sec:simple_walk}

In this section we consider the simple walk (see on the left in Figure \ref{ExExEx}). This is an example of non-algebraic but holonomic model \cite{BMM}. We would like to illustrate the main theorem in this case. At the end of the section we will comment on the result given by the alternative approach (based on Section~\ref{sec:expressions_x/y}).

\subsection{First derivation of the simple walk GF}
We thus start by applying the formula \eqref{AA} of Theorem \ref{serie}. The most important quantity in \eqref{AA} is the function $f_y(\omega)$. The simple walk belongs to case \ref{case:010} (classification of Section \ref{sec:focus_poles}), so that $f_y(\omega)$ has the following particularly simple expression
\begin{equation*}
     f_y(\omega) = \frac{x'(\omega)}{2z}.
\end{equation*}
\begin{lem}
\label{lem:principal_part_srw}
For the simple walk, the poles of $f_y(\o)$ in the rectangle $\omega_1[0,1)+\omega_2[0,1)$ are at $\o_2/8$ and $7\o_2/8$. These poles are double, with respective principal parts equal to
\begin{equation*}
     \frac{1}{4z^2}\frac{1}{(\o-\o_2/8)^2},\qquad \frac{-1}{4z^2}\frac{1}{(\o-7\o_2/8)^2}.
\end{equation*}
\end{lem}

\begin{proof}
Our starting point is the formula
\begin{equation}
\label{eq:expression_uniformization_srw}
     x(\o)=x_4+\frac{d'(x_4)}{\wp(\o)-d''(x_4)/6},
\end{equation}
which is an immediate consequence of \eqref{expression_uniformization}.
The order of the Weierstrass function being equal to two, the function $x(\o)$ has one pole of order two or two poles of order one in the fundamental rectangle $\omega_1[0,1)+\omega_2[0,1)$. Let us first show that
\begin{equation}
\label{eq:wp_at_o_2/8}
     \wp(\o_2/8) = d''(x_4)/6.
\end{equation}
We know that $\wp(\o_2/2) = g_x(x_1)$. It remains to use the dissection formula twice and to simplify. It follows that $\wp(7\o_2/8) = d''(x_4)/6$. Using \eqref{eq:expression_uniformization_srw} the function $x(\o)$ has two simple poles in the fundamental rectangle. Then the residue at $\o_2/8$ is given by $d'(x_4)/\wp'(\o_2/8)$. After simplifications we find $-1/(2z)$. The first part of Lemma \ref{lem:principal_part_srw} follows. The second part of the lemma follows from the parity of $x(\o)$.
\end{proof}

Since $\o_3/\o_2=1/2$ we have $k=1$ and $\ell=2$ in \eqref{kn}. Moreover, with Lemma \ref{lem:principal_part_srw}, the function $f_y(\o)$ has two poles $f_i^y$ in the parallelogram $\Pi_{0,0}-\o_{x_1}$, namely $\pm \o_2/8$, with corresponding principal part
\begin{equation*}
     \frac{\pm 1}{4z^2}\frac{1}{(\o\mp \o_2/8)^2}.
\end{equation*}
We then deduce from Theorem \ref{serie} that (we recall that $r_y(\o) = y(\o) zQ(0,y(\o))$ for the simple walk)
 \begin{align*}
  r_y(\o)-r_y(\o_{y_2})
 =\sum_{p=-\infty}^{\infty}\sum_{n=0}^{\infty}\Bigg(
 -&\frac{\lfloor n/2\rfloor+1}{4z^2}
     \frac{1}{(\o+(n/2-1/8)\o_2+p\o_1)^2}\\
   -&\frac{\lfloor n/2\rfloor+1}{4z^2}\frac{1}{(-\o+(n/2-1/8+3/2)\o_2+(p+1)\o_1)^2}\
   \nonumber\\
+2&\frac{\lfloor n/2\rfloor+1}{4z^2}
\frac{1}{(\o_{y_2}+(n/2-1/8)\o_2+p\o_1)^2}\\
+&\frac{\lfloor n/2\rfloor+1}{4z^2}
     \frac{1}{(\o+(n/2+1/8)\o_2+p\o_1)^2}\\
   +&\frac{\lfloor n/2\rfloor+1}{4z^2}\frac{1}{(-\o+(n/2+1/8+3/2)\o_2+(p+1)\o_1)^2}\
   \nonumber\\
-2&\frac{\lfloor n/2\rfloor+1}{4z^2}
\frac{1}{(\o_{y_2}+(n/2+1/8)\o_2+p\o_1)^2}\Bigg).
\end{align*}
In particular we deduce from the above formula that the function $r_y(\o)$ has poles at any point of the form
\begin{equation*}
     \o_2/8+n\o_2/4+p\o_1,\qquad \forall n,p\in{\bf Z},
\end{equation*}
except for $n\in\{1,2,3,4\}$. All these poles are of order two. It is interesting to note that the coefficients in front of the principal parts (typically the $\pm(\lfloor n/2\rfloor+1)/(4z^2)$) are not bounded (they grow linearly in $n$). As a straightforward consequence, it is impossible for the function $r_y(\o)$ to be elliptic; this is directly related to the fact that the model is not algebraic. However, the linearity in $n$ of the coefficients $\pm(\lfloor n/2\rfloor+1)/(4z^2)$ would allow us to express $r_y(\o)$ in terms of the quasi-elliptic function $\phi(\o)$ defined in \eqref{eq:expression_good_phi}, which is holonomic but not algebraic in $y(\o)$.

We shall not simplify the above expression so as to obtain an expression of $Q(0,y)$ in terms of $y$: this is first because such expressions already exist in the literature, second because we have already illustrated in full detail our approach with Kreweras' example, the last reason is that the second approach (see below) would be more appropriate for this.

\subsection{Second derivation of the simple walk GF}
We now pass to the second approach, that of Section~\ref{sec:expressions_x/y}. The starting point is Equation \eqref{cont1}, which says that $r_y(\o+\o_3)= r_y(\o)+ f_y(\o)$. Applying it twice, we obtain
\begin{equation*}
     r_y(\o+\o_2)= r_y(\o)+ f_y(\o)+f_y(\o+\o_2/2),\qquad \forall \o\in{\bf C}.
\end{equation*}
The above equation can be rewritten as (for all $\o\in{\bf C}$), with $\phi(\o)$ as in \eqref{eq:expression_good_phi},
\begin{equation*}
     r_y(\o+\o_2)-\phi(\o+\o_2)[f_y(\o+\o_2)+f_y(\o+3\o_2/2)]=r_y(\o)-\phi(\o)[f_y(\o)+f_y(\o+\o_2/2)].
\end{equation*}
In other words, the function
\begin{equation}
\label{eq:function_to_study}
     r_y(\o)-\phi(\o)[f_y(\o)+f_y(\o+\o_2/2)]
\end{equation}
is elliptic with periods $\o_1,\o_2$. To find it, it is therefore enough to find its poles in the rectangle $\o_1[0,1)+\o_2[0,1)$.

First, thanks to Theorem \ref{dva}, the function $r_y(\o)$ has only one pole in $\o_1[0,1)+\o_2[0,1)$; it is double, with principal part equal to
\begin{equation*}
     \frac{-1}{4z^2}\frac{1}{(\o-\o_2/8)^2}.
\end{equation*}
The function \eqref{eq:function_to_study} has also poles due to the functions $\phi(\o)$, $f_y(\o)$ and $f_y(\o+\o_2/2)$. It follows that \eqref{eq:function_to_study} has poles at $0,\o_2/8,3\o_2/8,5\o_2/8,7\o_2/8$. At $0$ in fact there is no pole. At $\o_2/8$ the principal part is
\begin{equation*}
     \frac{-(1+\phi(\o_2/8))}{4z^2}\frac{1}{(\o-\o_2/8)^2}+\frac{-\phi'(\o_2/8)}{4z^2}\frac{1}{\o-\o_2/8}.
\end{equation*}
At $3\o_2/8$ the principal part is
\begin{equation*}
     \frac{\phi(3\o_2/8)}{4z^2}\frac{1}{(\o-3\o_2/8)^2}+\frac{\phi'(3\o_2/8)}{4z^2}\frac{1}{\o-3\o_2/8}.
\end{equation*}
At the point $5\o_2/8$ the principal part is equal to
\begin{equation*}
     \frac{-\phi(5\o_2/8)}{4z^2}\frac{1}{(\o-5\o_2/8)^2}+\frac{-\phi'(5\o_2/8)}{4z^2}\frac{1}{\o-5\o_2/8}.
\end{equation*}
With the same reasoning, the principal part at $7\o_2/8$ is
\begin{equation*}
     \frac{\phi(7\o_2/8)}{4z^2}\frac{1}{(\o-7\o_2/8)^2}+\frac{\phi'(7\o_2/8)}{4z^2}\frac{1}{\o-7\o_2/8}.
\end{equation*}

We conclude that the function \eqref{eq:function_to_study} is equal to the following function (up to an additive constant), that we denote by $R(\o)$:
\begin{align*}
     R(\o)=&-\frac{(1+\phi(\o_2/8))}{4z^2} \wp(\o-\o_2/8)-\frac{\phi'(\o_2/8)}{4z^2}\zeta(\o-\o_2/8)\\
     &+\frac{\phi(3\o_2/8)}{4z^2} \wp(\o-3\o_2/8)+\frac{\phi'(3\o_2/8)}{4z^2}\zeta(\o-3\o_2/8)\\
      &-\frac{\phi(5\o_2/8)}{4z^2} \wp(\o-5\o_2/8)-\frac{\phi'(5\o_2/8)}{4z^2}\zeta(\o-5\o_2/8)\\
      &+\frac{\phi(7\o_2/8)}{4z^2} \wp(\o-7\o_2/8)+\frac{\phi'(7\o_2/8)}{4z^2}\zeta(\o-7\o_2/8),
\end{align*}
where all elliptic functions above have the periods $\o_1,\o_2$. Finally, Equation \eqref{eq:function_to_study} says that (up to an additive constant)
\begin{equation*}
     r_y(\o)=\phi(\o)[f_y(\o)+f_y(\o+\o_2/2)]+R(\o),\qquad \forall \o\in{\bf C}.
\end{equation*}
We do not pursue the computations which would allow us to obtain an expression of $Q(0,y)$ in terms of $y$. This would follow from similar computations as those done concerning Kreweras' case.

To conclude, let us mention that the fact---already observed---that the function $r_y(\o)$ has no bounded coefficients in the principal parts comes from the identity $\phi(\o+n\o_2)=n+\phi(\o)$, see \ref{item:phi_o_2}.

%


\section{An infinite group case example}
\label{sec:infinite_group_example}

In this section we treat the infinite group case
 $\mathcal S=\{(-1,0),(-1,-1),(0,-1),(1,1)\}$, as on Figure \ref{Exinfinite}.
 We use the formula \eqref{AA} of Theorem \ref{serie}, so we first need to study the poles of $f_y(\o)$.

\begin{lem}
\label{lem:principal_part_infinite}
For the walk of Figure \ref{Exinfinite}, the poles of $f_y(\o)$
 in the rectangle $\omega_1[0,1)+\omega_2[0,1)$ are at $0$,
  $\o_3/2\in(0,\o_2/2)$ and $\o_2-\o_3/2\in(\o_2/2,\o_2)$. These poles are simple, with residues equal
  to $-1/z$, $1/(2z)$ and $1/(2z)$, respectively.
\end{lem}

\unitlength=0.6cm
\begin{figure}[t]
  \begin{center}
\begin{tabular}{cccc}
    \hspace{-0.9cm}
        \begin{picture}(5,5.5)
    \thicklines
    \put(1,1){{\vector(1,0){4.5}}}
    \put(1,1){\vector(0,1){4.5}}
    \thinlines
    \put(4,4){\vector(1,1){1}}
    \put(4,4){\vector(-1,0){1}}
    \put(4,4){\vector(0,-1){1}}
    \put(4,4){\vector(-1,-1){1}}
    \put(4,1){\vector(1,1){1}}
    \put(4,1){\vector(-1,0){1}}
    \put(1,4){\vector(1,1){1}}
    \put(1,4){\vector(0,-1){1}}
    \put(1,1){\vector(1,1){1}}
    \linethickness{0.1mm}
    \put(1,2){\dottedline{0.1}(0,0)(4.5,0)}
    \put(1,3){\dottedline{0.1}(0,0)(4.5,0)}
    \put(1,4){\dottedline{0.1}(0,0)(4.5,0)}
    \put(1,5){\dottedline{0.1}(0,0)(4.5,0)}
    \put(2,1){\dottedline{0.1}(0,0)(0,4.5)}
    \put(3,1){\dottedline{0.1}(0,0)(0,4.5)}
    \put(4,1){\dottedline{0.1}(0,0)(0,4.5)}
    \put(5,1){\dottedline{0.1}(0,0)(0,4.5)}
    \end{picture}
    \end{tabular}
  \end{center}
  \vspace{-4mm}
\caption{Model analyzed in Section \ref{sec:infinite_group_example}}
\label{Exinfinite}
\end{figure}
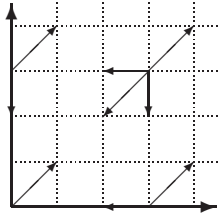

\begin{proof}
The poles of $f_y(\o)=x(\o)[y(-\o+2\o_{x_2})-y(\o)]$
can be only at points $\o$ such that $x(\o)=\infty$,
 $y(\o)=\infty$ or $y(\widehat\xi\o)=y(-\o+2\o_{x_2})=\infty$.
Since the degree of $d(x)$ in \eqref{dx} is three in this example, we have $x_4=\infty$,
 so that there is only one point in the fundamental parallelogram for which $x(\o)=\infty$,
 this is $\o_{x_4}$. We have $x(\o_{x_4})=\infty$ and $y(\o_{x_4})=0$.
It comes from similar arguments that the only point with
$y(\o)=\infty$ is $\o_{y_4}=\o_{x_4}+\o_3/2$. At this point the
equalities $y(\o_{y_4})=\infty$ and $x(\o_{y_4})=0$ hold.

If now $y(\widehat \xi \o)=\infty$,
then $x(\widehat \xi \o)=x(\o)=0$, and then either
$y(\o)=\infty$ or $y(\o)=-1$, so that $\o_{y_4}-\o_3=\o_{x_4}-\o_3/2$ is
 the unique point with $y(\widehat\xi\o)=\infty$.
 We have $x(\o_{x_4}-\o_3/2)=0$ and $y(\o_{x_4}-\o_3/2)=-1$.

It is easy to verify that all three points
$\o_{x_4}-\o_3/2,\o_{x_4}, \o_{x_4}+\o_3/2=\o_{y_4}$ are poles of $f_y(\o)$ and by the alternative formula
(according to the classification of Section \ref{sec:focus_poles}, the model on
Figure \ref{Exinfinite} belongs to case \ref{case:100})
   \begin{equation*}
     f_y(\omega) = \frac{1}{2z}\frac{x'(\omega)}{x(\omega)}
\end{equation*}
of Lemma~\ref{fyfy}, they are of the first order, and have the residues as given in the statement of the lemma.
\end{proof}

There are three poles in $\Pi_{0,0}-\o_{x_1}$, namely $\pm \o_3/2$
and $0$. Let us call them $f_\pm^y$ and $f_0^y$. We have by
\eqref{AA}
\begin{equation*}
 r_y(\o)-r_y(\o_{y_2})
 =\sum_{p=-\infty}^{\infty}\sum_{n=0}^{\infty}\sum_{s=0}^{k-1}
 \left\{A_{s,p,n}^{f_+^y} (\o)+A_{s,p,n}^{f_-^y} (\o)+A_{s,p,n}^{f_0^y} (\o)\right\},
 \end{equation*}
 where (using the fact that $(\o_1+\o_2+\o_3)/2=\o_{y_2}$)
 \begin{align}
 \label{eq:formula_+_inf}
  A_{s,p,n}^{f_+^y}(\o) =
 -&\frac{\lfloor n/\ell\rfloor+1}{2z}\frac{1}{\o+s\o_2+(n-1/2)\o_3+p\o_1}\\
   -&\frac{\lfloor n/\ell\rfloor+1}{2z}\frac{1}{-\o+(s+1)\o_2+(n+1/2)\o_3+(p+1)\o_1}
   \nonumber\\
+2&\frac{\lfloor n/\ell\rfloor+1}{2z}\frac{1}{(s+1/2)\o_2+
n\o_3+(p+1/2)\o_1}.\nonumber
\end{align}
The formula for $A_{s,p,n}^{f_-^y}(\o)$  is the same as
\eqref{eq:formula_+_inf} with the $n$ in the denominator
(and only there!) be replaced by $n+1$.
  The formula  for $A_{s,p,n}^{f_0^y}(\o)$
is the same as \eqref{eq:formula_+_inf} provided that the $n$ in the
denominator  be replaced by $n+1/2$ and $(2z)$ by $(-z)$.
 The function $r_y(\o)$ has thus poles at any point of the form
\begin{equation*}
     s\o_2+n\o_3/2+p\o_1,\qquad \forall s,n,p\in{\bf Z}.
\end{equation*}
It is possible to make some simplifications in the above formulas,
for instance thanks to the fact that $A_{s,p,n+1}^{f^y_+}(\o)$ and
$A_{s,p,n}^{f^y_-}(\o)$ are very similar.
 However, there are no major simplifications or cancellations,
  and it is difficult to say more, since $k$ and $\ell$ in the formula
   $\o_3/\o_2=k/\ell$ vary with $z$ in a complicated way.

     In the same way
the poles of $f_x(\o)$ in the rectangle
$\omega_1[0,1)+\omega_2[0,1)$ are at $0$,
  $\o_3/2\in(0,\o_2/2)$ and $\o_3\in(0,\o_2)$. These poles are simple, with residues equal
  to $1/2z$, $-1/z$ and $1/2z$, respectively.
Then there are three poles in $\Pi_{0,0}+\o_{y_1}$, namely
$\o_2+\o_3/2 \pm \o_3/2$ and $\o_2+\o_3/2$ that we call $f_\pm^x$
and $f_0^x$. We have
\begin{equation*}
r_x(\o)-r_y(\o_{x_2})
 =\sum_{p=-\infty}^{\infty}\sum_{n=0}^{\infty}\sum_{s=0}^{k-1}
 \left\{B_{s,p,n}^{f_+^x} (\o)+B_{s,p,n}^{f_-^x} (\o)+B_{s,p,n}^{f_0^x} (\o)\right\},
\end{equation*}
  where
\begin{align}
 \label{eq:formula_+_inf1}
  B_{s,p,n}^{f_-^x}(\o) =
 -&\frac{\lfloor n/\ell\rfloor+1}{2z}\frac{1}{\o-(s+1)\o_2-n\o_3-p\o_1}\\
   -&\frac{\lfloor n/\ell\rfloor+1}{2z}\frac{1}{-\o-s\o_2-n\o_3-(p-1)\o_1}
   \nonumber\\
+2&\frac{\lfloor n/\ell\rfloor+1}{2z}\frac{1}{-(s+1/2)\o_2-
n\o_3-(p-1/2)\o_1}.\nonumber
\end{align}
  The formula for $B_{s,p,n}^{f_+^x}(\o)$  is the same as
\eqref{eq:formula_+_inf1} provided that the $n$ in the denominator
be replaced by $n+1$.
  The formula  for $B_{s,p,n}^{f^0_x}(\o)$
is the same as \eqref{eq:formula_+_inf1} provided that the $n$ in
the denominator  be replaced by $n+1/2$ and $(2z)$ by $(-z)$.

  Finally, there are two points $\o_x^0 \in \Pi_{0,0}$ such that
   $x(\o_0^x)=0$, these are $\o_3/2$ and $\o_2-\o_3/2$.
   There are also two points $\o_y^0 \in \Pi_{0,0}+\o_3/2$ such that
  $y(\o_0^y)=0$, these are $\o_3$ and $\o_2$.
      Due to (\ref{eq:first_def_AA12}), (\ref{eq:first_def_AA13}) with any of these two points $\o_x^0$ and any
        of these two $\o_y^0$, we obtain
   $Q(x(\o),0,z)$ and $Q(0,y(\o),z)$.

\appendix
\section{Another possible approach}
\label{sec:another_possible_approach}

In this section we discuss another approach which could also lead to expressions for $r_x(\o)$ and $r_y(\o)$ on the universal covering. Contrary to the representations in terms of infinite series of meromorphic functions we obtain in Theorem \ref{serie}, the expressions for the GFs we could obtain following this alternative approach would be directly in terms of elliptic functions. This other approach is based on methods proposed in \cite[Chapter 4]{FIM}, and it could be applied in principle to any model of walks with non-singular set of increments
$\mathcal{S}$ and any $z \in \mathcal{H}$. The difficulty of obtaining in this way closed formulas substantially depends on the values of $\mathcal{S}$ and $z$. Moreover, this method does not produce  explicit
 representations that can be unified for all models and all $z\in \mathcal{H}$, contrary to the one preferred in this paper.

We sketch this alternative approach below. It heavily relies on \cite[Theorem 4.4.1]{FIM} . This theorem states that if $\o_2/\o_3$ is rational, the function $r_x(\o)$ can be written as
\begin{equation}
\label{eq:Theorem441}
     r_x(\o)=w_1(\o)+\widetilde\Phi(\o)\phi(\o)+w(\o)/s(\o),
\end{equation}
where $w_1(\o)$ and $s(\o)$ are rational functions in the variable $x(\o)$, while $\phi(\o)$ and $w(\o)$ are algebraic in $x(\o)$. Further, in \cite[Lemma 2.1]{FR} it  is shown that $\widetilde\Phi(\o)$ is holonomic (but not algebraic) in $x(\o)$. Notice that \cite[Theorem 4.4.1]{FIM} is proved for $z=1/|\mathcal{S}|$ only, but in \cite{FR} it is observed that this result also holds for all $z \in (0, 1/|\mathcal{S}|)$.

Theorem 4.4.1 of \cite{FIM} has already been used in \cite{FR} to give another proof, after \cite{BK2,BMM}, that the $23$ walks having a finite group of the walk have holonomic GFs. It has also been shown in \cite{FR} that for the five models with a finite group and a positive covariance
\begin{equation*}
     \textstyle\sum_{(i,j)\in\mathcal S} ij-[\sum_{(i,j)\in\mathcal S} i][\sum_{(i,j)\in\mathcal S} j]
\end{equation*}
(including Kreweras' and Gessel's models), $\widetilde\Phi(\o)$ is identically zero. This implies that these models have algebraic GFs. Finally, \cite[Theorem 4.4.1]{FIM} was also applied in \cite{KRIHES} to show that for the $51$ non-singular walks with an infinite group \eqref{group}, the GFs are holonomic for all $z\in\mathcal H$ (though the trivariate GF $Q(x,y;z)$ is not holonomic).  

In these two examples of use, \cite[Theorem 4.4.1]{FIM} was used to answer to Problem \ref{Challenge_2}, stated at the very beginning of this article (i.e., to give some qualitative informations on the GFs). Later on, we realized that the same theorem could also be used to solve Problem \ref{Challenge_1}, on an explicit expression for the GFs. The idea consists in using the two following facts simultaneously:
\begin{itemize}
     \item With the exception of $w$, all the functions in \eqref{eq:Theorem441} could be found (in particular, their poles could be computed) by using the same approach (via Galois theory) as in \cite[Chapter 4]{FIM}.
As for $w$, its only known properties are that for all $\o\in{\bf C}$,
\begin{align}
     w(\o)&=w(\o+\o_1),\label{eq:w-periodic-1}\\
     w(\o)&=w(\o_2-\o),\label{eq:w-qperiodic-2}\\
     w(\o)&=w(\o-\o_3).\label{eq:w-periodic-3}
\end{align}
Equations \eqref{eq:w-periodic-1}, \eqref{eq:w-qperiodic-2} and \eqref{eq:w-periodic-3} are proved in \cite[Lemma 4.3.3]{FIM}, \cite[Theorem 4.3.1]{FIM} and \cite[Lemma 4.3.3]{FIM}, respectively. Below we shall give some consequences of these three equations, but before we pass to the second fact.
     \item For $z\in(0,1/|\mathcal S|)$, $r_x(\o)$ is a GF which is analytic on $\{\o\in{\bf C} : |x(\o)|<1\}$, see \eqref{def_CGF}. Therefore, it cannot have any singularities in the part of the fundamental parallelogram $\Pi_{0,0}$ where $|x(\o)|<1$.
\end{itemize}
For these reasons, the role of the function $w(\o)$ is to compensate for the possible poles which appear because of $w_1(\o)$, $\widetilde\Phi(\o)\phi(\o)$ and $1/s(\o)$ in \eqref{eq:Theorem441}.

We now show how this compensation property of $w(\o)$ can lead to eventually determine it. The function $w(\o)$ is $\o_1,\o_3$ elliptic (\eqref{eq:w-periodic-1} and \eqref{eq:w-periodic-3}). Further, with \eqref{eq:w-qperiodic-2}, we obtain that the function $v(\o) = w(\o+\o_2/2)$ is even and $\o_1,\o_3$ elliptic. Using Property \ref{even_rational_wp}, we reach the conclusion that there exists a rational function $R$ such that $v(\o)=R(\wp(\o;\o_1,\o_3))$. This way, we obtain the existence of a rational function $R$ such that
\begin{equation}
\label{eq:first_definition_w}
     w(\o) = R(\wp(\o-\o_2/2;\o_1,\o_3)),\qquad \forall \o\in{\bf C}.
\end{equation}
Next, instead of finding an expression for $r_x(\o)$ given in \eqref{eq:Theorem441}, we can be interested in the following weaker problem: finding a rational function $R$ such that, with $w(\o)$ defined as in \eqref{eq:first_definition_w}, the function $r_x(\o)$ has no pole in the domain $\{\o\in{\bf C} : |x(\o)|<1\}\cap \Pi_{0,0}$. To be able to find $R$, we need to overcome three difficulties:
\begin{enumerate}
     \item\label{enumi:compare}We first have to compare the domain $\{\o\in{\bf C} : |x(\o)|<1\}\cap \Pi_{0,0}$ with a translated fundamental parallelogram $\o_1[0,1)+\o_3[0,1)$. Indeed, since $w(\o)$ is elliptic with periods $\o_1,\o_3$, to determine it we need to know its poles in a translated parallelogram of $\o_1[0,1)+\o_3[0,1)$.
     \item\label{enumi:poles}If, in point \ref{enumi:compare}, $\{\o\in{\bf C} : |x(\o)|<1\}\cap \Pi_{0,0}$ contains a translated fundamental parallelogram $\o_1[0,1)+\o_3[0,1)$, then we have to deduce from the expressions of $w_1(\o)$, $\widetilde\Phi(\o)\phi(\o)$ and $1/s(\o)$ the zeros and poles that $w(\o)$ should have in a translated parallelogram of $\o_1[0,1)+\o_3[0,1)$. 
     \item \label{enumi:deducing} Finally, we have to deduce from \ref{enumi:poles} an expression for the rational function $R$ and an expression for $\wp(\o-\o_2/3;\o_1,\o_3)$ in terms of $x(\o)$.
\end{enumerate}
Item \ref{enumi:compare} is difficult to prove---though purely technical. The difficulty of item \ref{enumi:poles} depends on the expression of the functions $w_1(\o)$, $\widetilde\Phi(\o)\phi(\o)$ and $1/s(\o)$ in \eqref{eq:Theorem441}. In the general case they are quite complicated, so that it is almost impossible to obtain an expression for $R$ (this explains why we chose another approach in this article). However, in some particular cases, it may happen that these functions are simple to deal with, see just below for Kreweras' example. Item \ref{enumi:deducing} is doable, and similar computations are done for Kreweras' case (Sections \ref{sec:compute}--\ref{sec:proof_complete_Kreweras}).

To conclude Section \ref{sec:another_possible_approach}, we have a closer look at Kreweras' example. One could show that the following choices\footnote{It is worth noting that there is no uniqueness of $w_1(\o)$ and $1/s(\o)$, as if we add some constants to the latter, we can subtract them in $w(\o)$.} are suitable: $w_1(\o) = -1/x(\o)$, $\widetilde\Phi(\o)\phi(\o)=0$ and $1/s(\o)=1$. We could then show that the function $R$
has the following simple form
\begin{equation*}
     R(X) = \frac{\wp_{1,3}'(\o_2/6)/(2z)}{X-\wp_{1,3}(\o_2/6)}+\beta.
\end{equation*}
To find $\beta$, we could use the same idea as in Theorem \ref{serie}. Then it would remain to express $\wp(\o-\o_2/3;\o_1,\o_3)$ in terms of $x(\o)$. For this we could use the same ideas and techniques as in Sections \ref{sec:compute}--\ref{sec:proof_complete_Kreweras}.

\section{Some properties of elliptic functions}
\label{appendix-elliptic}

In this appendix, we gather the results we used on Weierstrass $\wp$- and $\zeta$-functions.

\begin{lem}
\label{lemma_properties_wp} Let ${\zeta}$ and $\wp$ be the Weierstrass functions with certain periods $\overline{\omega},\widehat{\omega}$.
\begin{enumerate}[label={\rm (P\arabic{*})},ref={\rm (P\arabic{*})}]
\item\label{expression_zeta}We have the expansion
     \begin{equation*}
          \zeta(\o) = \frac{1}{\o}+\sum_{(\overline{n},\widehat{n})\in{\bf Z}^2\setminus \{(0,0)\}}\left( \frac{1}{\o-(\overline{n}\overline{\o}+\widehat{n}\widehat{\o})}+\frac{1}{\overline{n}\overline{\o}+\widehat{n}\widehat{\o}}+\frac{\o}{(\overline{n}\overline{\o}+\widehat{n}\widehat{\o})^2}\right),\qquad  \forall \omega\in{\bf C}.
     \end{equation*}
    As for the expansion of $\wp(\o)$, it is given in \eqref{eq:first_time_expansion_wp}.
     In particular, in the fundamental rectangle $\overline{\omega}[0,1)+\widehat{\omega}[0,1)$, $\zeta$ {\rm(}resp.\ $\wp${\rm)} has a unique pole. It is of order $1$ {\rm(}resp.\ $2${\rm)}, at $0$, and has residue $1$ {\rm(}resp.\ $0$, and principal part $1/\o^2${\rm)}.
\item\label{elliptic_poles_0}An elliptic function with no poles in the fundamental rectangle $\overline{\omega}[0,1)+\widehat{\omega}[0,1)$ is constant.
\item\label{addition_theorem}We have the addition theorems
     \begin{equation*}
          \zeta({\omega}+\widetilde{\omega})=\zeta({\omega})+\zeta(\widetilde{\omega})+\frac{1}{2}\frac{\wp'({\omega})-
          \wp'(\widetilde{\omega})}{\wp({\omega})-
          \wp(\widetilde{\omega})},\qquad  \forall \omega,\widetilde{\omega}\in{\bf C},
     \end{equation*}
     and
          \begin{equation*}
          \wp({\omega}+\widetilde{\omega})=-\wp({\omega})-\wp(\widetilde{\omega})+\frac{1}{4}\left(\frac{\wp'({\omega})-
          \wp'(\widetilde{\omega})}{\wp({\omega})-
          \wp(\widetilde{\omega})}\right)^2,\qquad  \forall \omega,\widetilde{\omega}\in{\bf C}.
     \end{equation*}
\item\label{expression_elliptic_zeta_direct}For given $\widetilde{\o}_1,\ldots,\widetilde{\o}_p\in{\bf C}$, define
     \begin{equation}
     \label{eq:expression_elliptic_zeta}
          f(\o)=c+\sum_{1\leq \ell\leq p}r_\ell\zeta(\o-\widetilde{\o}_\ell),\qquad  \forall \omega\in{\bf C}.
     \end{equation}
The function $f$ above is elliptic if and only if $\sum_{1\leq \ell\leq p}r_\ell=0$.
\item\label{expression_elliptic_zeta}Let $f$ be an elliptic function with periods $\overline{\omega},\widehat{\omega}$ such that in the fundamental rectangle $\overline{\omega}[0,1)+\widehat{\omega}[0,1)$, $f$ has only poles of order $1$, at $\widetilde{\o}_1,\ldots,\widetilde{\o}_p$, with residues $r_1,\ldots,r_p$, respectively. Then there exists a constant $c$ such that \eqref{eq:expression_elliptic_zeta} holds.

\item\label{expression_elliptic_zeta_generalization}Let $f$ be an elliptic function with periods $\overline{\omega},\widehat{\omega}$ such that in the fundamental rectangle $\overline{\omega}[0,1)+\widehat{\omega}[0,1)$, $f$ has poles at $\widetilde{\o}_1,\ldots,\widetilde{\o}_p$, with principal parts $F_1,\ldots ,F_p$:
\begin{equation*}
     F_i(\o) =\sum_{k=1}^{n_i} \frac{r_{i,k}}{(\o-\widetilde\o_i)^k},\qquad \forall i\in\{1,\ldots ,p\},\qquad \forall \o\in{\bf C}.
\end{equation*}
Then up to an additive constant, we have the equality
\begin{equation*}
     f(\o) = \sum_{i=1}^p r_{i,1}\zeta(\o-\widetilde \o_i)+\sum_{i=1}^p \sum_{k=2}^{n_i} r_{i,k} \wp^{(k-1)}(\o-\widetilde \o_i),\qquad \forall \o\in{\bf C},
\end{equation*}
where $\wp^{(k)}$ means the $k$-th derivative of $\wp$, and where the $\zeta$ and $\wp$ functions above have the same periods $\overline{\omega},\widehat{\omega}$.
\item \label{algebraic_theorem} Let $f$ and $g$ be non-constant elliptic functions with the same periods $\overline{\omega},\widehat{\omega}$. Then there exists a non-zero polynomial $P$ such that $P(f(\o),g(\o))=0$, for all $\o\in{\bf C}$.
\item\label{even_rational_wp}Let $f$ be an even elliptic function with periods $\overline{\omega},\widehat{\omega}$, and let $\wp$ be the Weierstrass elliptic function with periods $\overline{\omega},\widehat{\omega}$. Then there exists a rational function $F$ such that $f(\o)=F(\wp(\o))$, for all $\o\in{\bf C}$.
\item\label{principle_transformation}Let $p$ be some positive integer. The Weierstrass elliptic function with periods $\overline{\omega},\widehat{\omega}/p$ can be written in terms of ${\wp}$ as
     \begin{equation*}
          {\wp}(\omega)+\sum_{1\leq \ell\leq p-1}\{{\wp}(\omega+\ell\widehat{\omega}/p)-{\wp}(\ell\widehat{\omega}/p)\},\qquad  \forall \omega\in{\bf C}.
     \end{equation*}
\item\label{half_period_translated_zeta}The function $\zeta$ is quasi-periodic, in the sense that
\begin{equation*}
     \zeta(\o+\overline{\o}) = \zeta(\o)+2\zeta(\overline{\o}/2),\qquad  \zeta(\o+\widehat{\o}) = \zeta(\o)+2\zeta(\widehat{\o}/2),\qquad \forall \omega\in{\bf C}.
\end{equation*}
\item\label{Legendre_identity}Assume that $\overline{\omega}$ is real and $\widehat{\omega}$ purely imaginary. We have
\begin{equation*}
     \zeta(\overline{\o}/2)\widehat{\o}-\zeta(\widehat{\omega}/2)\overline{\omega}=i\pi.
\end{equation*}
\item\label{sum_residues}The sum of the residues of an elliptic function in a fundamental parallelogram is zero.
\end{enumerate}
\end{lem}

\begin{proof}
Property \ref{expression_zeta} is proved in \cite[Sections 20.2 and 20.4]{WW}. \ref{elliptic_poles_0} is known as Liouville's theorem, and can be found in \cite[Section 20.1]{WW}. The addition theorems \ref{addition_theorem} are stated in \cite[Equations (5.5.3) and (5.5.16)]{SG}. For Properties \ref{expression_elliptic_zeta_direct} and \ref{expression_elliptic_zeta}, we refer to \cite[Section 20.52]{WW}. \ref{expression_elliptic_zeta_generalization} is the subject of \cite[Theoerem 3.14.4]{JS}. \ref{algebraic_theorem} is shown in \cite[Section 20.54]{WW}, and \ref{even_rational_wp} in \cite[Section 20.51]{WW}. \ref{principle_transformation} can be found in \cite[Example 9 in page 456]{WW}. Property \ref{Legendre_identity} is demonstrated in \cite[Equation (5.5.19]{SG}. Finally, \ref{half_period_translated_zeta} is proved in \cite[Section 20.41]{WW}, and \ref{sum_residues} is proved in \cite[Section 20.12]{WW}.
\end{proof}

\section*{Acknowledgments}
We would like to thank Alin Bostan for interesting discussions. We also thank an anonymous referee for very useful comments and suggestions, which led us to substantially improve the first version of this paper.

%

\end{document}